\documentclass[11pt]{amsart}

\usepackage{amsmath, amscd, amssymb}
\usepackage[frame,cmtip,arrow,matrix,line,graph,curve]{xy}
\usepackage{graphpap, color}
\usepackage[mathscr]{eucal}
\usepackage{cancel}
\usepackage{verbatim}
\usepackage{adjustbox}

\numberwithin{equation}{section}

\newtheorem{theorem}{Theorem}[section]
\newtheorem{corollary}[theorem]{Corollary}
\newtheorem{proposition}[theorem]{Proposition}
\newtheorem{lemma}[theorem]{Lemma}

\theoremstyle{definition}
\newtheorem{definition}[theorem]{Definition}
\newtheorem{question}[theorem]{Question}
\newtheorem{example}[theorem]{Example}
\newtheorem{remark}[theorem]{Remark}

\newcommand{\Z}{\mathbb{Z}}
\newcommand{\Q}{\mathbb{Q}}

\newcommand{\C}{\mathbb{C}}

\newcommand{\A }{\mathbb{A}}
\newcommand{\PP}{\mathbb{P}}

\def\BB{\mathbb{B}}
\def\CC{\mathbb{C}}
\def\DD{\mathbb{D}}
\def\EE{\mathbb{E}}
\def\FF{\mathbb{F}}
\def\GG{\mathbb{G}}

\def\II{\mathbb{I}}

\def\KK{\mathbb{K}}
\def\LL{\mathbb{L}}

\def\TT{\mathbb{T}}

\def\sO{{\mathscr O}}

\def\sG{{\mathscr G}}
\def\sM{{\mathscr M}}

\def\sL{{\mathscr L}}

\def\sE{\mathscr{E}}
\def\sF{\mathscr{F}}

\def\sP{\mathscr{P}}

\def\sD{\mathscr{D}}

\def\sX{\mathscr{X}}

\newcommand{\cal}{\mathcal}

\def\cF{{\cal F}}

\def\cH{{\cal H}}
\def\cI{{\cal I}}
\def\cJ{{\cal J}}

\def\cL{{\cal L}}

\def\cX{{\cal X}}
\def\cY{{\cal Y}}
\def\cZ{{\cal Z}}
\def\cX{\mathcal{X} }

\def\bX{\mathbf{X}}
\def\bY{\mathbf{Y}}

\def\bP{\mathbf{P}}

\def\bV{\mathbf{V}}
\def\bF{\mathbf{F}}

\def\bL{{\mathbf L}}
\def\bR{{\mathbf R}}

\def\fC{\mathfrak{C}}

\def\fN{\mathfrak{N}}

\def\fq{\mathfrak{q}}
\def\fQ{\mathfrak{Q}}

\def\fe{\mathfrak e}

\def\tphi{\widetilde{\phi} }

\def\tf{\widetilde{f}}
\def\tg{\widetilde{g}}
\def\th{\widetilde{h}}
\def\tk{\widetilde{k}}

\def\tp{\widetilde{p}}
\def\tq{\widetilde{q}}

\def\tcX{\widetilde{\cX}}
\def\tcY{\widetilde{\cY}}

\def\hO{\hat{\O}}

\def\hbar{\overline{h}}

\DeclareMathOperator{\Ext}{Ext}

\DeclareMathOperator{\id}{id} 
\DeclareMathOperator{\rank}{rank}

\DeclareMathOperator{\tr}{tr}

\DeclareMathOperator{\Hom}{Hom}
\DeclareMathOperator{\spec}{Spec}
\DeclareMathOperator{\Sym}{Sym}

\def\coker{\mathrm{coker} }

\def\ch{\mathrm{ch} }
\def\td{\mathrm{td} }

\def\hom{\cH om }

\def\and{\quad{\rm and}\quad}
\def\lra{\longrightarrow }
\def\mapright#1{\,\smash{\mathop{\lra}\limits^{#1}}\,}

\def\beq{\begin{equation}}
\def\eeq{\end{equation}}
\def\ben{\begin{enumerate}}
\def\een{\end{enumerate}}

\def\sp{\mathrm{sp}}
\def\cone{\mathrm{cone}}

\def\Perf{\mathrm{Perf}}

\def\spl{{\mathrm{spl}}}

\def\virt{^{\mathrm{vir}}}

\def\dual{^{\vee}}
\def\sqe{\sqrt{e}}
\def\uperp{^{\perp}}
\def\un{^{[n]}}

\def\trunc{\tau^{\ge-1}}

\def\and{\quad\text{and}\quad}

\def\O{\mathscr{O}}
\def\I{\mathcal{I}}

\title[Virtual pullbacks in DT4 theory]{Virtual pullbacks in Donaldson-Thomas theory of Calabi-Yau 4-folds}
\date{2021.10.07}

\author{Hyeonjun Park}
\address{Department of Mathematical Sciences, Seoul National University, Seoul 08826, Korea}
\email{hyeonjun93@snu.ac.kr}

\begin{document}

\begin{abstract}
Recently, Oh and Thomas constructed algebraic virtual cycles for moduli spaces of sheaves on Calabi-Yau 4-folds. The purpose of this paper is to provide a virtual pullback formula between these Oh-Thomas virtual cycles. We find a natural compatibility condition between 3-term symmetric obstruction theories that induces a virtual pullback formula. There are two types of applications.

Firstly, we introduce a Lefschetz principle in Donaldson-Thomas theory, which relates the tautological DT4 invariants of a Calabi-Yau 4-fold with the DT3 invariants of its divisor. As corollaries, we prove the Cao-Kool conjecture on the tautological Hilbert scheme invariants for very ample line bundles and the Cao-Kool-Monavari conjecture on the tautological DT/PT correspondence for line bundles with Calabi-Yau divisors when the tautological complexes are vector bundles.
 
Secondly, we present a correspondence between the Oh-Thomas virtual cycles on the moduli spaces of pairs and the moduli spaces of sheaves by combining the virtual pullback formula and a pushforward formula for virtual projective bundles. As corollaries, we prove the Cao-Maulik-Toda conjecture on the primary PT/GV correspondence for irreducible curve classes and the Cao-Toda conjecture on the primary JS/GV correspondence under the coprime condition, assuming the Cao-Maulik-Toda conjecture on the primary Katz/GV correspondence. Moreover, we also prove tautological versions of these two correspondences.
\end{abstract}

\maketitle
\tableofcontents

\section*{Introduction}

\subsection*{Background}

Donaldson-Thomas invariants are virtual numbers of counting sheaves, defined as the integrals of cohomology classes over the virtual cycles of moduli spaces of sheaves. For Calabi-Yau 3-folds and Fano 3-folds, such virtual cycles were constructed by Thomas \cite{Tho} through 2-term perfect obstruction theories \cite{BeFa,LiTi}. The associated Donaldson-Thomas invariants have been studied intensively during the last two decades and it turned out that they have rich structures and properties (e.g., correspondence to Gromov-Witten invariants \cite{MNOP1,MNOP2} and rationality \cite{PT,PT3}, motivic property \cite{Behrend,JS}, modularity \cite{GS,TT},  etc.).

Generalizing Donaldson-Thomas theory to higher-dimensional algebraic varieties is not obvious. The standard method of constructing virtual cycles in \cite{BeFa,LiTi} does not work for higher-dimensional varieties since the natural obstruction theories on the moduli spaces of sheaves are no longer 2-term. In particular, for Calabi-Yau 4-folds, the obstruction theories are 3-term symmetric, which are never 2-term.

In the groundbreaking work \cite{BJ}, Borisov and Joyce constructed real virtual cycles for schemes with 3-term symmetric obstruction theories based on the Darboux theorem \cite{BBJ,BBBJ} and derived differential geometry. Thus Donaldson-Thomas invariants for Calabi-Yau 4-folds can be defined via the Borisov-Joyce virtual cycles. However, computation of these DT4 invariants through the Borisov-Joyce virtual cycles is believed to be very difficult.

Recently, Oh and Thomas \cite{OT} lifted Borisov-Joyce virtual cycles to Chow groups by generalizing Cao-Leung's algebraic approach \cite{CL}. The key idea is to localize Edidin-Graham's square root Euler class \cite{EG} by an isotropic section via Kiem-Li's cosection-localized Gysin map \cite{KL}. This algebraic method enables us to compute DT4 invariants in some cases.

Currently, there are three known computational tools in DT4 theory:
\begin{enumerate}
\item {\em Reduction to Edidin-Graham/Behrend-Fantechi classes} \cite{CL};
\item {\em Torus localization} \cite{OT};
\item {\em Cosection localization} \cite{KP}.
\end{enumerate}
These tools are shown to be effective when they can be applied, i.e., the moduli space is smooth/virtually smooth or has a torus action/cosection (cf. \cite{CL,CKp,CMT,CMT2,CKc,CTd,CTt,CKMk}). However, since there are many examples that are not in the above cases, it is desired to develop additional tools.


\subsection*{Main Result: Virtual pullback formula}

The purpose of this paper is to provide a new computational tool for DT4 invariants: a {\em virtual pullback formula} between Oh-Thomas virtual cycles.

Recall \cite{Man} that Manolache introduced the notion of virtual pullbacks as a relative version of Behrend-Fantechi virtual cycles \cite{BeFa}. More specifically, given a morphism $f:\cX \to \cY$ of schemes equipped with a relative 2-term perfect obstruction theory $\phi_f:\KK_f \to \LL_f$, Manolache constructed a map
\[f^! : A_*(\cY) \to A_*(\cX),\]
called virtual pullback, satisfying the functorial property. In particular, if the schemes $\cX$ and $\cY$ are also equipped with 2-term perfect obstruction theories $\phi_{\cX}$ and $\phi_{\cY}$ that fit into a compatibility diagram
\beq\label{Compatibility.Man}
\xymatrix{
f^*\KK_{\cY} \ar[r] \ar[d]^{f^*{\phi_{\cY}}} & \KK_{\cX} \ar[r] \ar[d]^{\phi_{\cX}} & \KK_f  \ar[d]^{\phi_f} \ar[r] &\\
f^*\LL_{\cY} \ar[r]  &\LL_{\cX} \ar[r] & \LL_f \ar[r] & }
\eeq
then there is a virtual pullback formula
\beq\label{VirtualPullbackFormula.BeFa}
f^![\cY]\virt_{BF} = [\cX]\virt_{BF} \in A_*(\cX)
\eeq
between the associated Behrend-Fantechi virtual cycles.

The main result of this paper is a generalization of the virtual pullback formula \eqref{VirtualPullbackFormula.BeFa} to the Oh-Thomas virtual cycles:

\begin{theorem}[Virtual pullback formula]\label{Thm.VirtualPullbackFormula}
Let $f:\cX \to \cY$ be a morphism of quasi-projective schemes equipped with the following obstruction theories:
\begin{enumerate}
\item [D1)] symmetric obstruction theories $\phi_{\cX}:\EE_{\cX} \to \LL_{\cX}$ and $\phi_{\cY}:\EE_{\cY} \to \LL_{\cY}$ of tor-amplitude $[-2,0]$, oriented, and isotropic (see Definition \ref{Def.Isotropic});
\item [D2)] a perfect obstruction theory $\phi_f:\EE_f\to\LL_f$ of tor-amplitude $[-1,0]$.
\end{enumerate}
Assume that there exist morphisms of distinguished triangles
\beq\label{Compatibility.OT}
\xymatrix{
\DD\dual[2] \ar[r]^{\alpha\dual} \ar[d]^{\beta\dual} & \EE_{\cX} \ar[r] \ar[d]^{\alpha} & \EE_f  \ar@{=}[d] \ar[r] &\\
f^*\EE_{\cY} \ar[r]^{\beta} \ar[d]^{f^*{\phi_{\cY}}} & \DD \ar[r] \ar[d]^{\phi_{\cX}'} & \EE_f  \ar[d]^{\phi_f} \ar[r] &\\
f^*\LL_{\cY} \ar[r]  &\LL_{\cX} \ar[r] & \LL_f \ar[r] & }
\eeq
for some perfect complex $\DD$ and maps $\alpha$, $\beta$, $\phi'_{\cX}$ such that $\phi_{\cX} =\phi'_{\cX} \circ \alpha$. We further assume that the orientations of $\EE_{\cX}$ and $\EE_{\cY}$ are compatible with the isomorphism $\det(\EE_{\cX}) \cong f^* \det(\EE_{\cY})$ induced by \eqref{Compatibility.OT}. Then we have a {\em virtual pullback formula}
\begin{equation}\label{VPF.OT}
f^![\cY]\virt_{OT} = [\cX]\virt_{OT} \in A_*(\cX)
\end{equation}
between the associated Oh-Thomas virtual cycles.
\end{theorem}

In DT4 theory, there is a general philosophy that we should use operations that are {\em symmetric}. The compatibility diagram \eqref{Compatibility.OT} is a typical example. If we use the compatibility diagram \eqref{Compatibility.Man} in the situation of Theorem \ref{Thm.VirtualPullbackFormula}, then we will have
\[f^*\EE_{\cY} \cong \cone(\EE_{\cX} \to \EE_f)[-1],\]
which violates the fact that both $\EE_{\cX}$ and $f^*\EE_{\cY}$ are symmetric complexes. Thus we need to consider a symmetric operation such as
\beq\label{Intro.Reduction}
f^*\EE_{\cY} \cong \cone( \cone(\EE_f\dual[2] \to \EE_{\cX})  \to \EE_f)[-1],
\eeq
which is exactly what the compatibility diagram \eqref{Compatibility.OT} does. The above operation \eqref{Intro.Reduction} will be called the {\em reduction} of the symmetric complex $\EE_{\cX}$ by the isotropic subcomplex $\EE_f$, which generalizes the reduction $K\uperp/K$ of an orthogonal bundle $E$ by an isotropic subbundle $K$. 


We briefly sketch how we prove Theorem \ref{Thm.VirtualPullbackFormula}. Overall, the proof has a similar structure to the standard arguments of functoriality of various pullbacks \cite{Ful,KKP,Man,Qu,CKL,KPVIT}. Specifically, in our setting, we first introduce a relative version of Oh-Thomas virtual cycles: {\em square root virtual pullbacks} for relative 3-term symmetric obstruction theories. Then the crucial step is to construct a relative 3-term symmetric obstruction theory $\phi_h$ for the composition
\[ h : \cX\times \A^1 \to \cY \times \A^1 \to M^{\circ}_{\cY/\spec(\C)},\]
where $M^{\circ}_{\cY/\spec(\C)}$ denotes the deformation space \cite{Ful,Kresch}, such that the relative intrinsic normal cone $\fC_{\cX\times \A^1/M^{\circ}_{\cY/\spec(\C)}}$ is isotropic. After constructing $\phi_h$, a deformation argument reduces Theorem \ref{Thm.VirtualPullbackFormula} to the functoriality for $\cX \to \cY \to \fC_{\cY}$. Further replacing the cone stack $\fC_{\cY}$ by some cone $C$ that covers $\fC_{\cY}$, the blowup method in \cite{KP} will complete the proof.

The virtual pullback formula in Theorem \ref{Thm.VirtualPullbackFormula} was motivated by the Cao-Kool conjecture \cite{CKp} on the tautological Hilbert scheme invariants and the Cao-Maulik-Toda conjectures \cite{CMT,CMT2} on the genus zero Gopakumar-Vafa type invariants. Now Theorem \ref{Thm.VirtualPullbackFormula} turns out that these conjectures (under some assumptions) are corollaries of more general phenomena: {\em Lefschetz principle} and {\em Pairs/Sheaves correspondence}.

\subsection*{Application I: Lefschetz principle} 

Recall \cite{KKP} that the quantum Lefschetz principle relates the Gromov-Witten invariants of an algebraic variety with the Gromov-Witten invariants of its divisor. The virtual pullback formula in Theorem \ref{Thm.VirtualPullbackFormula} provides an analogous formula in Donaldson-Thomas theory.

Let $X$ be a Calabi-Yau 4-fold, i.e., a smooth projective variety $X$ such that $K_X\cong \O_X$. The Hilbert scheme $I_{n,\beta}(X)$ of curves $C$ with $\ch(\O_C)=(0,0,0,\beta,n)$ carries an Oh-Thomas virtual cycle
\[[I_{n,\beta}(X)]\virt_{OT}  \in A_n (I_{n,\beta}(X))\]
by \cite{OT}, which depends on the choice of an orientation \cite{CGJ}. For a line bundle $L$ on $X$, define the tautological complex as
\[\sL_{n,\beta} := \bR\pi_*(\O_Z \otimes L)\]
where $Z$ denotes the universal family and $\pi: I_{n,\beta}(X) \times X \to I_{n,\beta}(X)$ denotes the projection map. Suppose that there is a smooth divisor $D$ with $\O_X(D)=L$ and let $i:D \hookrightarrow X$ be the inclusion map. Then the Hilbert scheme $I_{n,\beta}(D)$ of curves $C$ on $D$ with $\ch(i_*\O_C)=(0,0,0,\beta,n)$ carries a Behrend-Fantechi virtual cycle
\[[I_{n,\beta}(D)]\virt_{BF} \in A_{-\beta \cdot D} (I_{n,\beta}(D))\]
by \cite{MNOP1,MNOP2,PT}.

\begin{theorem}[Lefschetz principle]\label{Intro.Thm.Lefschetz}
Let $X$ be a Calabi-Yau 4-fold and $L$ be a line bundle. Let $D$ be a smooth divisor with $\O_X(D)=L$. Fix a curve class $\beta \in H_2(X,\Q)$ and an integer $n \in \Z$. If $H^1(C,L)=0$ for all $[C] \in I_{n,\beta}(X)$, then for any orientation on $I_{n,\beta}(X)$, there exist canonical signs $\sigma(e)$ for connected components $I_{n,\beta}(D)^e$ of $I_{n,\beta}(D)$ such that
\beq\label{Intro.Eq.Lefschetz}\sum_{e}(-1)^{\sigma(e)}(j_{e})_*[I_{n,\beta}(D)^e]\virt_{BF} =  e(\sL_{n,\beta}) \cap [I_{n,\beta}(X)]\virt_{OT}\eeq
where $j_e : I_{n,\beta}(D)^e \hookrightarrow I_{n,\beta}(X)$ denotes the inclusion map.
\end{theorem}

The defining equation of the divisor $D$ induces the tautological section $\tau$ of the vector bundle $\sL_{n,\beta}$ over $I_{n,\beta}(X)$ whose zero locus is $I_{n,\beta}(D)$,
\[\xymatrix{  & \sL_{n,\beta} \ar[d] \\   I_{n,\beta}(D) \ar@{^{(}->}[r]^{j} & I_{n,\beta}(X)  \ar@/_.4cm/[u]_{\tau} .}\] 
The key idea of proving Theorem \ref{Intro.Thm.Lefschetz} is to construct a natural {\em 3-term} symmetric obstruction theory on $I_{n,\beta}(D)$ as a {\em non-trivial} extension of the standard 2-term perfect obstruction theory on $I_{n,\beta}(D)$. Then by comparing this 3-term symmetric obstruction theory with the standard 3-term symmetric obstruction theory on $I_{n,\beta}(X)$, we can apply the virtual pullback formula in Theorem \ref{Thm.VirtualPullbackFormula},
\[j^![I_{n,\beta}(X)]\virt_{OT} = [I_{n,\beta}(D)]\virt_{OT} = \sum_e (-1)^{\sigma(e)}[I_{n,\beta}(D)^e]\virt_{BF},\]
which implies the Lefschetz formula \eqref{Intro.Eq.Lefschetz}.

Theorem \ref{Intro.Thm.Lefschetz} also holds for moduli spaces of pairs. For instance, let $P_{n,\beta}(X)$ be the moduli space of Pandharipande-Thomas stable pairs $(F,s)$ with $\ch(F)=(0,0,0,\beta,n)$ \cite{PT,Pot}. Then we have two virtual cycles
\[ [P_{n,\beta}(X)]\virt_{OT} \in A_n (P_{n,\beta}(X)) \and [P_{n,\beta}(D)]\virt_{BF} \in A_{-\beta\cdot D} (P_{n,\beta}(X))\]
where $P_{n,\beta}(D)$ is the moduli space of stable pairs on $D$. If $H^1(X,F)=0$ for all $(F,s) \in P_{n,\beta}(X)$, then we have a Lefschetz formula
\[\sum_{e}(-1)^{\sigma(e)}(j_{e})_*[P_{n,\beta}(D)^e]\virt_{BF} =  e(\sL_{n,\beta}) \cap [P_{n,\beta}(X)]\virt_{OT}\]
where $\sL_{n,\beta}:=\bR\pi_*(\FF\otimes L)$ denotes the tautological complex, $\FF$ denotes the universal family, and all the other notations are given analogously.

There are two immediate corollaries of Theorem \ref{Intro.Thm.Lefschetz}. Firstly, when $\beta=0$, the tautological complex $L\un:=\sL_{n,0}$ on the Hilbert scheme of points $X\un=I_{n,0}(X)$ is always a vector bundle. Thus the Lefschetz principle computes the tautological invariants in terms of the MacMahon funtion.

\begin{corollary}[Tautological Hilbert scheme invariants]\label{Intro.Cor.CaoKoolconj}
Let $X$ be a Calabi-Yau 4-fold and $L$ be a line bundle. If there is a smooth connected divisor $D$ such that $\O_X(D)=L$, then there exists a choice of orientations such that
\beq\label{Intro.CaoKoolConj}
\sum_{n\geq 0}\int_{[X\un]\virt}e(L\un)\cdot q^n =M(-q)^{\int_X{c_3(X)c_1(L)}}
\eeq
where $M(q)=\prod_{n \geq 1}(1-q^n)^{-n}$ denotes the MacMahon function.
\end{corollary}

The above formula \eqref{Intro.CaoKoolConj} was conjectured by Cao-Kool in \cite[Conjecture 1.2]{CKp} for any line bundle $L$. Corollary \ref{Intro.Cor.CaoKoolconj} proves that the Cao-Kool conjecture holds when $L$ has a smooth connected divisor (e.g., when $L$ is very ample).

Secondly, when $\beta \neq 0$, the Lefschetz principle for a Calabi-Yau divisor gives us a tautological DT/PT correspondence.

\begin{corollary}[Tautological DT/PT correspondence]\label{Intro.Cor.DTPTcurves}
Let $X$ be a Calabi-Yau 4-fold and $L$ be a line bundle which has a smooth connected Calabi-Yau divisor. Let $\beta \in H_2(X,\Q)$ be a curve class satisfying the followings:
\begin{enumerate}
\item [A1)] For all pure 1-dimensional closed subschemes $C$ of $X$ with $[C]=\beta$, we have $H^1(C,L)=0$.
\item [A2)] For all $n$, the inclusion maps $I_{n,\beta}(D) \hookrightarrow I_{n,\beta}(X)$ and $P_{n,\beta}(D) \hookrightarrow P_{n,\beta}(X)$ induce injective maps between the sets of connected components.
\end{enumerate}
Then there exists a choice of orientations such that
\beq\label{Intro.TautDT/PT}
\dfrac{\sum_{n\geq 0} \int_{[I_{n,\beta}(X)]\virt} e(\sL_{n,\beta}) \cdot  q^n}{\sum_{n\geq 0} \int_{[X\un]\virt} e(L\un) \cdot  q^n}  = \sum_{n\geq 0} \int_{[P_{n,\beta}(X)]\virt} e(\sL_{n,\beta}) \cdot  q^n\,.
\eeq
\end{corollary}

The 4-fold DT/PT correspondence was first conjectured by Cao-Kool in \cite[Conjecture 0.3]{CKc} for primary insertions. Later, the tautological DT/PT correspondence \eqref{Intro.TautDT/PT} was conjectured by Cao-Kool-Monavari in \cite[Conjecture 0.13]{CKMk} for any $X$, $L$, and $\beta$. Corollary \ref{Intro.Cor.DTPTcurves} shows that the Cao-Kool-Monavari conjecture holds when $L$ has a smooth Calabi-Yau divisor and the tautological complexes are vector bundles.

It was considered to be difficult to check the tautological DT/PT correspondence \eqref{Intro.TautDT/PT} for compact Calabi-Yau 4-folds. The main reason is that the reduction to Behrend-Fantechi method does not work for the DT moduli spaces $I_{n,\beta}(X)$, even in the special geometries with irreducible curve classes due to free-roaming points. In Example \ref{Ex.DT/PTrational}, we will present some simple examples that satisfy the assumptions in Corollary \ref{Intro.Cor.DTPTcurves}. This provides an evidence for the Cao-Kool-Monavari conjecture.

\subsection*{Application II: Pairs/Sheaves correspondence}

In many cases, maps between moduli spaces of sheaves or complexes can be realized as {\em virtual projective bundles}. Since there is a general pushforward formula for virtual projective bundles, a virtual pullback formula for these cases is practically effective for computing invariants. We provide a correspondence between the moduli of stable pairs and the moduli of stable sheaves as an example.

Let us first briefly explain what we call a virtual projective bundle in this paper. Let $\cX$ be any scheme and $\KK$ be a 2-term perfect complex of amplitude $[0,1]$. The virtual projective bundle associated to $\KK$ is a pair of a projective cone
\[p : \PP(\KK):=\mathrm{Proj} \Sym^\bullet (h^0(\KK\dual)) \to \cX\]
and a natural 2-term perfect obstruction theory 
\[\LL\virt_{\PP(\KK)/\cX}:=\cone(\O_{\PP(\KK)} \to p^*\KK(1))\dual \to \LL_{\PP(\KK)/\cX}.\]
For any cycle class $\alpha \in A_*(\cX)$ and a K-theory class $\xi \in K^0(\cX)$, we have a {\em pushforward formula}
\[p_* (c_m(p^*\xi(1)) \cap p^! \alpha) = \sum_{0 \leq i \leq m} \binom{s-i}{m-i} \cdot c_i(\xi) \cap c_{m-i+1-r}(-\KK) \cap \alpha\]
where $r$ is the rank of $\KK$ and $s$ is the rank of $\xi$.

Let $X$ be a Calabi-Yau 4-fold. The moduli space $P_{n,\beta}(X)$ of PT stable pairs $(F,s)$ with $\ch(F)=(0,0,0,\beta,n)$ and the moduli space $M_{n,\beta}(X)$ of stable sheaves $G$ with $\ch(G)=(0,0,0,\beta,n)$ have Oh-Thomas virtual cycles
\[[P_{n,\beta}(X)]\virt_{OT} \in A_n (P_{n,\beta}(X) ) \and [M_{n,\beta}(X)]\virt_{OT} \in A_1 (M_{n,\beta}(X))\]
by \cite{OT}. When the curve class $\beta$ is irreducible, then all pure 1-dimensional sheaves $G$ with $\ch_3(G)=\beta$ are stable so that $M_{n,\beta}(X)$ is proper, and the forgetful map
\[p: P_{n,\beta}(X) \to M_{n,\beta}(X) : (F,s) \mapsto F\]
is well-defined.

\begin{theorem}[Pairs/Sheaves correspondence]\label{Intro.Thm.Pairs/Sheaves}
Let $X$ be a Calabi-Yau 4-fold, $\beta \in H_2(X,\Q)$ be an irreducible curve class, and $n$ be an integer. Assume that there exists a universal family $\GG$ on $M_{n,\beta}(X)$. Then the forgetful map $p:P_{n,\beta}(X)\to M_{n,\beta}(X)$ is the virtual projective bundle of the tautological complex $\bR\pi_*\GG$. Moreover, there exists a choice of orientations such that the following pullback/pushforward formulas hold:
\begin{enumerate}
\item {\em (Pullback formula)} We have
\[[P_{n,\beta}(X)]\virt_{OT} = p^![M_{n,\beta}(X)]_{OT}\virt\]
where $p^!$ denotes the virtual pullback of the virtual projective bundle.
\item {\em (Pushforward formula)} For any vector bundle $E$ on $X$, we have
\[p_*\left(c_{n-1}(\sE_{n,\beta}) \cap [P_{n,\beta}(X)]\virt_{OT}\right) = \tbinom{N }{n-1} \cdot [M_{n,\beta}(X)]\virt_{OT}\]
where $\sE_{n,\beta}:=\bR\pi_*(\FF\otimes E)$ denotes the tautological complex on $P_{n,\beta}(X)$, $[\O_{P_{n,\beta}(X)\times X} \to \FF]$ denotes the universal pair of $P_{n,\beta}(X)$, and $N:=n\cdot \rank(E)+\int_{\beta}c_1(E)$ denotes the rank of $\sE_{n,\beta}$. 
\end{enumerate}
Here both the projection maps $P_{n,\beta}(X) \times X \to P_{n,\beta}(X)$ and $M_{n,\beta}(X) \times X \to M_{n,\beta}(X)$ are denoted by $\pi$.
\end{theorem}


Theorem \ref{Intro.Thm.Pairs/Sheaves} is a 4-fold analog of Pandharipande-Thomas' formula \cite[Theorem 4]{PT3} on Calabi-Yau 3-folds. The main difference is that they used motivic techniques, while we use a virtual cycle approach. Our approach is a generalization of Cao-Maulik-Toda's approach in \cite[Proposition 2.10]{CMT2}, where they considered the case when Oh-Thomas virtual cycles reduce to Behrend-Fantechi virtual cycles.

The correspondence of virtual cycles in Theorem \ref{Intro.Thm.Pairs/Sheaves} induces correspondences of the primary invariants and the tautological invariants. Recall \cite{CMT,CMT2} that the primary invariants for a cohomology class $\gamma \in H^4(X,\Q)$ are defined as
\[P_{n,\beta}(\gamma) := \int_{[P]\virt}  (\pi _* ( \ch_3(\FF) \cup \gamma))^n, \quad M_{n,\beta}(\gamma) := \int_{[M]\virt} \pi _* ( \ch_3(\GG) \cup \gamma).\]
On the other hand, the tautological invariants for a line bundle $L$ (cf. \cite{CKp,CTt}) can be defined as
\[P_{n,\beta}(L) := \int_{[P]\virt} c_n(\bR\pi_*(\FF\otimes L)), \quad M_{n,\beta}(L) := \int_{[M]\virt} c_1(\bR\pi_*(\GG\otimes L)).\]
Here we abbreviated the virtual cycles $[P_{n,\beta}(X)]\virt_{OT}$ and $[M_{n,\beta}(X)]\virt_{OT}$ by $[P]\virt$ and $[M]\virt$, respectively.

\begin{corollary}[Primary PT/GV correspondence]\label{Intro.Cor.PT/Katz}
Let $X$ be a Calabi-Yau 4-fold and $\beta \in H_2(X,\Q)$ be an irreducible curve class. Then there exists a choice of orientations such that \[P_{1,\beta}(\gamma) = M_{1,\beta}(\gamma)\] for any $\gamma \in H^4(X,\Q)$.
\end{corollary}

In \cite[Conjecture 0.1]{CMT2}, Cao-Maulik-Toda conjectured a primary PT/GV correspondence\footnote{Here we replaced the genus zero Gopakumar-Vafa type invariants $n_{0,\beta}(\gamma)$ in \cite{KlemmPandharipande} by the primary stable sheaf invariants $M_{1,\beta}(\gamma)$ based on Cao-Maulik-Toda's Katz/GV conjecture in \cite[Conjecture 0.2]{CMT}.}
\beq\label{Intro.Eq.PT/GV}
P_{1,\beta}(\gamma) = \sum_{\beta=\beta_1+\beta_2} M_{1,\beta_1}(\gamma) \cdot P_{0,\beta_2}
\eeq
for any  $\beta \in H_2(X,\Q)$ and $\gamma \in H^4(X,\Q)$. Corollary \ref{Intro.Cor.PT/Katz} proves that the Cao-Maulik-Toda conjecture \eqref{Intro.Eq.PT/GV} holds for irreducible curve classes.

\begin{corollary}[Tautological PT/GV correspondence]
Let $X$ be a Calabi-Yau 4-fold, $\beta$ be an irreducible curve class, and $n$ be an integer. Assume that there is a universal family $\GG$ of $M_{n,\beta}(X)$. Then there exists a choice of orientations such that
\[P_{n,\beta}(L) = 
\begin{cases}
-M_{n,\beta}(\O_X) & \text{if } n=0 \\
M_{n,\beta}(L) - M_{n,\beta}(\O_X) & \text{if }n\geq1 
\end{cases}\]
for any line bundle $L$ with $\int_{\beta}c_1(L)=0$.
\end{corollary}

Theorem \ref{Intro.Thm.Pairs/Sheaves} can be generalized to {\em reducible} curve classes as follows: Assume that $M_{n,\beta}(X)$ has no strictly semi-stable sheaves. The moduli space of Joyce-Song type stable pairs in \cite{CTd} (cf. \cite{Pot,JS})
\[P^{JS}_{n,\beta}(X) = \{(F,s) : F \in M_{n,\beta}(X) \and s:\O_X \to F \text{ is non-zero}\}\]  
have an Oh-Thomas virtual cycle $[P^{JS}_{n,\beta}(X)]\virt _{OT}\in A_n (P^{JS}_{n,\beta}(X))$ by \cite{CTd,OT}. Now we have a well-defined forgetful map \[ p : P^{JS}_{n,\beta}(X) \to M_{n,\beta}(X) : (F,s) \mapsto F\] for a reducible curve class $\beta$.

\begin{corollary}[JS/GV correspondence]\label{Intro.Cor.JS/GV}
Let $X$ be a Calabi-Yau 4-fold, $\beta \in H_2(X,\Q)$ be a curve class, and $n$ be an integer. Assume that $\int_{\beta} c_1(\O_X(1))$ and $n$ are coprime. Then the forgetful map $p:P_{n,\beta}^{JS}(X) \to M_{n,\beta}(X)$ is the virtual projective bundle of the tautological complex $\bR\pi _*\GG$. Furthermore, there exists a choice of orientations such that the formula
\[[P_{n,\beta}^{JS}(X)]\virt _{OT} = p^! [M_{n,\beta}(X)]\virt _{OT}\]
holds. Consequently, we have a primary JS/GV correspondence
\beq\label{Intro.Eq.PrimaryJS/GV} P^{JS}_{1,\beta}(\gamma) = M_{1,\beta}(\gamma)\eeq
for any $\gamma \in H^4(X,\Q)$, and a tautological JS/GV correspondence
\[P^{JS}_{n,\beta}(L) = 
\begin{cases}
-M_{n,\beta}(\O_X) & \text{if } n=0 \\
M_{n,\beta}(L) - M_{n,\beta}(\O_X) & \text{if }n\geq1
\end{cases}\]
for any line bundle $L$ with $\int_{\beta}c_1(L)=0$. Here the primary invariants $P^{JS}_{n,\beta}(\gamma)$ and the tautological invariants $P^{JS}_{n,\beta}(L)$ are defined analogously.
\end{corollary}

The primary JS/GV correspondence \eqref{Intro.Eq.PrimaryJS/GV} in Corollary \ref{Intro.Cor.JS/GV} was conjectured by Cao-Toda in \cite[Conjecture 0.3]{CTd}.\footnote{In \cite[Conjecture 0.3]{CTd}, Cao-Toda conjectured a primary JS/GV correspondence for any $n$ and $\beta$. Corollary \ref{Intro.Cor.JS/GV} proves that the Cao-Toda conjecture holds when $\int_{\beta} c_1(\O_X(1))$ and $n$ are coprime.}

\subsection*{Generalizations}

We can generalize the virtual pullback formula in Theorem \ref{Thm.VirtualPullbackFormula} to Deligne-Mumford stacks without assuming the quasi-projectivity. The essential ingredient for this generalization is the Kimura sequence for Artin stacks, which will appear in a forthcoming paper \cite{BP} with Younghan Bae.

Chang-Kiem-Li \cite{CKL} discovered that Graber-Pandharipande's torus localization formula \cite{GrPa} for Behrend-Fantechi virtual cycles can be deduced by Manolache's virtual pullback formula. Analogously, Oh-Thomas's torus localization formula \cite[Theorem 7.1]{OT} can be deduced by the virtual pullback formula in Theorem \ref{Thm.VirtualPullbackFormula}. This virtual pullback approach allows us to remove the quasi-projectivity hypothesis in the torus localization formula.

\begin{proposition}[Torus localization]
Consider a separated DM stack $\cX$ equipped with a $T:=\GG_m$-action and a $T$-equivariant 3-term symmetric obstruction theory. Assume that the fixed locus $\cX^T$ has the $T$-equivariant resolution property. Then we have
\[i_*\left(\dfrac{[\cX^T]\virt}{\sqrt{e}(N\virt)}\right) = [\cX]\virt \in A^T_*(\cX) \otimes_{\Q[t]} \Q[t^{\pm1}]\]
where $i: \cX^T \hookrightarrow \cX$ denotes the inclusion map (see Appendix \ref{A.quasi-projectivity}).
\end{proposition}

We can also generalize the virtual pullback formula in Theorem \ref{Thm.VirtualPullbackFormula} to K-theory by replacing Manolache's virtual pullback with the K-theoretic {\em twisted} virtual pullback (cf. \cite{NO,Qu}). Consequently, we also have K-theoretic versions of the Lefschetz principle and Pairs/Sheaves correspondence.

\subsection*{Future works and open problems}

We will apply the virtual pullback formula in Theorem \ref{Thm.VirtualPullbackFormula} to surface counting problems in a forthcoming paper \cite{BKP} with Younghan Bae and Martijn Kool. We will provide a Lefschetz principle for surface counting moduli spaces and introduce various virtual projective bundles that the virtual pullback formula can be applied.

In the Lefschetz principle (Theorem \ref{Intro.Thm.Lefschetz}), it is desirable to compute the {\em signs} $\sigma(e)$ in \eqref{Intro.Eq.Lefschetz}. If the signs $\sigma(e)$ all coincide, then we will have a simpler Lefschetz formula
\[j_*[I_{n,\beta}(D)]\virt_{BF} =  e(\sL_{n,\beta}) \cap [I_{n,\beta}(X)]\virt_{OT}\]
without the signs. This will allow us to remove the assumption A2 in the tautological DT/PT correspondence (Corollary \ref{Intro.Cor.DTPTcurves}).

It would be interesting to know whether there is a derived algebraic geometry interpretation of the virtual pullback formula in Theorem \ref{Thm.VirtualPullbackFormula}. A naive approach is to consider a morphism $\mathbf{f}: \bX \to \bY$ of (-2)-shifted symplectic derived schemes, but this is almost never quasi-smooth. The compatibility condition \eqref{Compatibility.OT} suggests that we should consider a third derived scheme $\bX' \subseteq \bX$ with a quasi-smooth map $\mathbf{f} : \bX' \to \bY$ such that  $(\bX')_{\mathrm{cl}}=(\bX)_{\mathrm{cl}}$ and $\mathbf{f}^*(\omega_{\bY})=\omega_{\bX}|_{\bX'}$. However, it is not obvious to the author what the geometric meaning of these conditions is and why the virtual pullback formula \eqref{VPF.OT} should hold in this setting from the derived algebraic geometry perspective.

In \cite{GJT}, Gross-Joyce-Tanaka conjectured a powerful wall-crossing formula for DT4 invariants, as an analog of the motivic wall-crossing formula for DT3 invariants \cite{JS}. We hope that the virtual pullback formula in Theorem \ref{Thm.VirtualPullbackFormula} can be used to prove the wall-crossing conjecture.

We expect that the virtual pullback formula in Theorem \ref{Thm.VirtualPullbackFormula} can be generalized to the reduced virtual cycles and the cosection-localized virtual cycles introduced in \cite{KP} under some natural compatibility conditions.

\subsection*{Related works}

In \cite{CQ}, Cao and Qu also presented Corollary \ref{Intro.Cor.CaoKoolconj} by an independent method. They developed another version of a virtual pullback formula that applies to a different setting. Basically, they considered a morphism $f:\cX\to\cY$ of schemes such that $\cX$ and $f$ are equipped with 2-term perfect obstruction theories, and $\cY$ is equipped with a 3-term symmetric obstruction theory whose pullback to $\cX$ splits by a 2-term perfect obstruction theory. Roughly speaking, their formula is an intermediate version of Manolache's virtual pullback formula and the virtual pullback formula (Theorem \ref{Thm.VirtualPullbackFormula}) in this paper. However, Cao-Qu's formula and the virtual pullback formula in Theorem \ref{Thm.VirtualPullbackFormula} are not directly related.

In \cite{Bojko}, Bojko proved the Cao-Kool conjecture \cite[Conjecture 1.2]{CKp} for all line bundles based on the results for very ample line bundles (Corollary \ref{Intro.Cor.CaoKoolconj}) and Gross-Joyce-Tanaka's conjectural wall-crossing formula \cite{GJT}.

\subsection*{Outline}

In \S\ref{S.sqrtVP}, we introduce the notion of square root virtual pullbacks as a relative version of Oh-Thomas virtual cycles. We also explore some basic properties of the isotropic condition of 3-term symmetric obstruction theories. In \S\ref{S.Functoriality}, we prove functoriality of square root virtual pullbacks, which is the relative version of the virtual pullback formula in Theorem \ref{Thm.VirtualPullbackFormula}. In \S\ref{S.Lefschetz}, we prove the Lefschetz principle and its corollaries. In \S\ref{S.Pairs/Sheaves}, we introduce virtual projective bundles and prove Pairs/Sheaves correspondence. In Appendix \ref{A.quasi-projectivity}, we generalize square root virtual pullbacks to algebraic stacks and prove the torus localization formula. In Appendix \ref{A.K-theory}, we generalize square root virtual pullbacks to K-theory. In Appendix \ref{A.Reduction}, we prove some elementary properties of the reduction operation of symmetric complexes.

\subsection*{Notations and conventions}
\begin{itemize}
\item All schemes and algebraic stacks are assumed to be of finite type over the field of complex numbers $\C$.
\item For a morphism $f:\cX\to \cY$ of algebraic stacks, we denote $\LL_f$ the full cotangent complex \cite{Illusie} (cf. \cite{Olsson}) and denote $\bL_f = \trunc \LL_f$ the truncated cotangent complex (cf. \cite{HT}).
\item For any algebraic stack $\cX$, we denote $A_*(\cX)$ the Chow group of Kresch \cite{Kresch} with $\Q$ coefficients. 
\item For any object $\EE \in D^{\leq0}_{\mathrm{coh}}(\cX)$ on an algebraic stack $\cX$, we denote $\fC(\EE):=h^1/h^0((\trunc\EE)\dual)$ the associated abelian cone stack. For any coherent sheaf $Q$, we denote $C(Q):=\spec S^\bullet Q$ the associated abelian cone. 
\item In this paper, all perfect obstruction theories are assumed to be of tor-amplitude $[-1,0]$, and all symmetric obstruction theories are assumed to be of tor-amplitude $[-2,0]$ and oriented.
\end{itemize}

\subsection*{Acknowledgements} The author would like to thank his advisor Young-Hoon Kiem for numerous valuable discussions. We are grateful to Martijn Kool for suggestions on the Lefschetz principle and the tautological DT/PT correspondence. We are thankful to Tasuki Kinjo for various explanations on derived algebraic geometry and shifted symplectic structures. We also thank Younghan Bae, Arkadij Bojko, Yalong Cao, Dominic Joyce, Adeel Khan, Sanghyeon Lee, Cristina Manolache, Jeongseok Oh, Feng Qu, Richard Thomas, and Yukinobu Toda for many useful discussions.

\section{Square root virtual pullback}\label{S.sqrtVP} In this section, we construct a {\em square root virtual pullback} for a three-term symmetric obstruction theory and prove its basic properties.


\subsection{Square root Euler class}

In this preliminary subsection, we briefly review the square root Euler class $\sqe(E)$ of a special orthogonal bundle $E$ and its localization $\sqe(E,s)$ by an isotropic section $s$ from \cite{EG,OT,KP}. Roughly speaking, the localized square root Euler classes $\sqe(E,s)$ are local models of the square root virtual pullbacks.

Let $E$ be a special orthogonal bundle of even rank $2n$ over a scheme $\cX$. In \cite{EG}, Edidin-Graham constructed an algebraic characteristic class
\[\sqe(E) \in A^n(\cX),\]
called the {\em square root Euler class} of $E$. An important property of the square root Euler class is the {\em reduction formula}
\beq\label{Eq3} \sqe(E)=e(K)\circ\sqe(K\uperp/K)\eeq
for an isotropic subbundle $K$ of $E$.


Suppose the special orthogonal bundle $E$ has given an isotropic section $s \in \Gamma(\cX,E)$. In \cite{OT}, Oh-Thomas proved that the square root Euler class $\sqe(E)$ can be localized to the zero locus $\cX(s)$ of the isotropic section $s$. More precisely, they constructed a bivariant class
\[\sqe(E,s) \in A^n_{\cX(s)}(\cX)\]
satisfying  $\imath_* \circ \sqe(E,s) = \sqe(E)$ for the inclusion map $\imath : \cX(s) \hookrightarrow \cX$. This was achieved by combining Edidin-Graham's flag variety \cite{EG}, Fulton-MacPherson's deformation to the normal cone \cite{Ful}, and Kiem-Li's cosection localization \cite{KL}.

An alternative construction of the localized square root Euler class $\sqe(E,s)$ was introduced in \cite{KP} via a blowup method. This blowup method provided some additional functorial properties of $\sqe(E,s)$. In particular, this gives us a refined version of the reduction formula in \eqref{Eq3}.
\begin{proposition}[{\cite[Lemma 4.5]{KP}}]\label{Prop.LocFunct}
Let $K$ be an isotropic subbundle of $E$ such that $s \cdot K=0$. Let $s_1 \in \Gamma(\cX,K^{\perp}/K)$ be the induced isotropic section and let $s_2=s|_{\cX(s_1)} \in \Gamma(\cX(s_1),K|_{\cX(s_1)})$ be the restriction. Then we have
\begin{equation}\label{Eq.LocFunct}
\sqrt{e}(E,s) = e(K|_{\cX(s_1)},s_2) \circ \sqrt{e}(K^{\perp}/K,s_1) \in A^n_{\cX(s)}(\cX).
\end{equation}
\end{proposition}

Roughly speaking, the formula \eqref{Eq.LocFunct} is a local model of the {\em functoriality} of square root virtual pullbacks.

\subsection{Symmetric complex}

In DT4 theory \cite{OT,BJ,CL}, three-term symmetric complexes play the role of two-term perfect complexes in virtual intersection theory \cite{BeFa,Man}. In this subsection, we study basic properties of these three-term symmetric complexes and the abelian cone stacks associated to them.

Let us first fix the notion of {\em symmetric complexes}.

\begin{definition}[Symmetric complex]\label{Def.SymCplx}
A {\em symmetric complex} $\EE$ on a scheme $\cX$ consists of the following data:
\begin{enumerate}
\item A perfect complex $\EE$ of tor-amplitude $[-2,0]$ on $\cX$.
\item A {\em non-degenerate symmetric form} $\theta$ on $\EE$, i.e., a morphism
\[\theta:\O_\cX \to (\EE\otimes \EE) [-2]\]
in the derived category of $\cX$ such that
\begin{enumerate}
\item  $\theta = \sigma \circ \theta$, where $\sigma :\EE\otimes \EE \to \EE\otimes \EE$ is the transition map, and
\item  the induced map $\iota_{\theta}:\EE\dual  \to \EE[-2]$ is an isomorphism.
\end{enumerate}
\item An {\em orientation} $o$ of $\EE$, i.e., an isomorphism $o:\O_\cX \to \mathrm{det}(\EE )$ of line bundles such that $\mathrm{det}(\iota_{\theta})=o\circ o\dual$.
\end{enumerate}
\end{definition}

In this paper, all symmetric complexes are assumed to be of amplitude $[-2,0]$ and oriented, unless stated otherwise.

\begin{proposition}[Symmetric resolution] \label{Prop.SymRes} 
Any symmetric complex $\EE$ on a quasi-projective scheme $\cX$ has a {\em symmetric resolution}, i.e., an isomorphism 
\begin{equation}
\label{Eq.SymRes} \left[B \to E\dual \to B\dual\right] \mapright{\cong} \EE 
\end{equation} 
in the derived category of $\cX$ satisfying the following properties: 
\begin{enumerate} 
\item $E$ is a special orthogonal bundle and $B$ is a vector bundle. 
\item Under the isomorphism \eqref{Eq.SymRes}, the symmetric form is represented by the chain map 
\[\xymatrix{
\EE\dual \ar[d]^{\iota_{\theta}} & B \ar[r]^d \ar@{=}[d] & E \ar[r]^{d\dual\circ q} \ar[d]^q & B\dual \ar@{=}[d] \\ 
\EE[-2] & B \ar[r]^{q \circ d} & E\dual \ar[r]^{d\dual} & B\dual}\]
where $q:E \to E\dual$ is the quadratic form of $E$. 
\item The orientation of $E$ is given by the canonical isomorphism between the determinant line bundles $\det(E) \cong \det(\EE)$ induced by \eqref{Eq.SymRes}.
\end{enumerate}

Moreover, given a map $\delta:\EE \to \KK$ from a symmetric complex $\EE$ to a perfect complex $\KK$ of tor-amplitude $[-1,0]$, if $h^0(\delta)$ is surjective, then there exist a symmetric resolution \eqref{Eq.SymRes} of $\EE$ and a resolution of $\KK$ by a complex of vector bundles such that the map $\delta$ can be represented by a degreewise surjective chain map 
\[\xymatrix{\EE \ar[d]^{\delta} & B \ar[r] \ar[d] & E\dual  \ar[r] \ar[d] & B\dual \ar[d] \\  \KK & 0 \ar[r]& K\dual \ar[r] & D\dual }\] 
of chain complexes.
\end{proposition}

\begin{proof}
We omit the proofs since the statements follow directly from the proof of \cite[Proposition 4.1]{OT}.
\end{proof}

Recall that an important operation on a special orthogonal bundle $E$ is the reduction $K\uperp/K$ by an isotropic subbundle $K$. A similar operation exists for symmetric complexes. We first introduce a notion of an {\em isotropic subcomplex} of a symmetric complex.

\begin{definition}[Isotropic subcomplex]\label{Def.IsotropicSubcomplex}
Let $\EE$ be a symmetric complex on a scheme $\cX$. An {\em isotropic subcomplex} of $\EE$ is a perfect complex $\KK$ equipped with a map $\delta :\EE \to \KK$ satisfying the following properties:
\begin{enumerate}
\item $\KK$ is perfect of tor-amplitude $[-1,0]$; 
\item $h^0(\delta ):h^0(\EE ) \to h^0(\KK )$ is surjective;
\item $\KK$ is {\em isotropic}, i.e., the induced symmetric form 
\[\delta_*(\theta):=(\delta\otimes\delta)\circ\theta :\O_\cX \to (\EE\otimes\EE)[-2] \to (\KK\otimes\KK)[-2]\]
on $\KK$ is zero.
\end{enumerate}
\end{definition}

We can form a {\em reduction} of a symmetric complex by an isotropic subcomplex.

\begin{proposition}[Reduction]\label{Prop.Reduction}
Let $\cX$ be a quasi-projective scheme. For any symmetric complex $\EE$ on $\cX$ and an isotropic subcomplex $\KK$ with $\delta : \EE \to \KK$, there exists a {\em reduction}
\beq\label{Eq.R2}
\left[\KK\dual[2] \mapright{\delta\dual} \EE \mapright{\delta} \KK\right]
\eeq
of $\EE$ by $\KK$, i.e., a symmetric complex $\GG$ on $\cX$, denoted by \eqref{Eq.R2}, satisfying the following properties:
\begin{enumerate}
\item There exists a morphism of distinguished triangles
\beq\label{Eq.Reduction}
\xymatrix{
\DD\dual[2] \ar[r]^{\alpha\dual} \ar[d]^{\beta\dual} & \EE \ar[r]^{\delta} \ar[d]^{\alpha} & \KK  \ar@{=}[d] \ar[r] &\\
\GG \ar[r]^{\beta} & \DD \ar[r] & \KK \ar[r] &}
\eeq
for some $\DD$, $\alpha$, $\beta$, where $\alpha\dual$, $\beta\dual$ are the duals of $\alpha$, $\beta$ with respect to the identifications $\EE\dual[2]\cong\EE$ and $\GG\dual[2]\cong\GG$.
\item The orientation of $\GG$ is induced from the orientation of $\EE$ under the canonical isomorphism $\det(\GG)\cong\det(\EE)$ given by \eqref{Eq.Reduction}.
\end{enumerate}
\end{proposition}

\begin{proof}
By Proposition \ref{Prop.SymRes}, we can choose a symmetric resolution $\EE\cong[B\to E\dual \to B\dual]$ and a resolution $\KK\cong[K\dual\to D\dual]$ such that $\delta:\EE \to \KK$ can be represented by a surjective chain map. Then $D$ is a subbundle of $B$ and $K$ is an isotropic subbundle of $E$. Thus we can define the reduction as
\[\GG = \left[(B/D) \to (K\uperp/K) \to (B/D)\dual\right],\]
which fits into the diagram \eqref{Eq.Reduction}. The distinguished triangles in the diagram \eqref{Eq.Reduction} give us isomorphisms
\[\det(\EE)\cong\det(\DD)\otimes\det(\KK\dual)\cong\det(\GG)\otimes\det(\KK)\otimes\det(\KK\dual)\cong\det(\GG)\]
between the determinant line bundles. It completes the proof.
\end{proof}

\begin{remark}
In Appendix \ref{A.Reduction}, we will prove that the reduction in Proposition \ref{Prop.Reduction} is unique as a symmetric complex.
\end{remark}








Recall that any object $\EE$ in the derived category of a scheme $\cX$, concentrated in non-positive degrees, defines an abelian cone stack by \cite[Proposition 2.4]{BeFa}. In other words, there is a contravariant functor
\[\fC : D_{\mathrm{coh}}^{\leq0}(\cX) \to \left\{\text{abelian cone stacks over }\cX \right\} : \EE  \mapsto h^1/h^0((\tau^{\geq-1}\EE)_{{\mathrm{fl}}}\dual).\]
We will also denote $\fC(\EE)$ by $\fC_\EE$.



For a symmetric complex $\EE$, the associated abelian cone stack $\fC_\EE$ has a {\em quadratic function} $\fq_\EE:\fC_\EE \to \A^1_\cX$ induced from the symmetric form $\theta$ of $\EE$.  This quadratic function $\fq_\EE$ will be used to define the isotropic condition of symmetric obstruction theories in \S\ref{ss.SOT}.

\begin{proposition}[Quadratic function]\label{Prop.QuadraticFunction}
For each symmetric complex $\EE$ on a scheme $\cX$, there exists a function
\[\fq_{\EE} : \fC_\EE \to \A^1_\cX\]
on the associated abelian cone stack $\fC_\EE$ satisfying the following properties:
\begin{enumerate}
\item If $\EE=E[1]$ for a special orthogonal bundle $E$, then 
\[\fq_\EE = \fq_E : \fC_\EE = E \to \A^1_\cX,\]
is the quadratic function on $E=\spec S^\bullet E\dual$ defined by the quadratic form $q_E \in \Gamma(\cX, S^2E\dual)$ of $E$.
\item For any morphism $f: \cY \to \cX$ of schemes, we have
\[f^*\fq_\EE = \fq_{f^*\EE} : f^*\fC_\EE = \fC_{f^*\EE} \to \A^1_\cY.\]
\item If $\GG$ is the reduction of $\EE$ by an isotropic subcomplex $\KK$, then the diagram
\beq\label{Eq.ReductionFormulaforQuadraticFunction}
\xymatrix{
\fC_{\DD} \ar[r]^{\fC_{\alpha}} \ar[d]_{\fC_{\beta}} & \fC_{\EE} \ar[d]^{\fq_{\EE}} \\
\fC_{\GG} \ar[r]_{\fq_{\GG}} & \A^1_\cX
}\eeq
commutes, where $\alpha$, $\beta$, $\DD$ are given as in Proposition \ref{Prop.Reduction}.
\end{enumerate}
Moreover, the function $\fq_\EE$ is uniquely determined by the above properties. We call $\fq_\EE:\fC_\EE \to \A^1_{\cX}$ the {\em quadratic function}.
\end{proposition}

\begin{proof}
We first construct the quadratic function $\fq_\EE$ when $\cX$ is a quasi-projective scheme. Choose a symmetric resolution $[B\to E\dual \to B\dual] \cong \EE$ of the symmetric complex $\EE$ as in Proposition \ref{Prop.SymRes}. Let $Q=\coker(B\to E\dual)$ be the cokernel and $C_Q=\spec S^\bullet Q$ be the abelian cone associated to $Q$. Then we have $\fC_\EE=[C_Q/B]$ by definition. Consider the function
\beq\label{4}\fq_E|_{C_Q} : C_Q \hookrightarrow E \to \A^1_\cX,\eeq
where $\fq_E$ is the quadratic function of the special orthogonal bundle $E$.

We claim that the function $\fq_E|_{C_Q}$ in \eqref{4} is invariant under the action of $B$ on $C_Q$ so that it descends to a function
\beq\label{q5}\fq_{\EE} : \fC_\EE=[C_Q/B] \to \A^1_\cX\eeq
on the abelian cone stack $\fC_\EE$. Indeed, it suffices to show that the two maps
\[\xymatrix{
B \times C_Q  \ar@<.5ex>[r]^<<<<<{\sigma} \ar@<-.5ex>[r]_<<<<<{p_2} & C_Q \ar@{^{(}->}[r] & E \ar[r]^{\fq_E} & \A^1_\cX
}\]
coincide, where $\sigma$ is the action and $p_2$ is the projection. Equivalently, it suffices to show that the two maps 
\beq\label{q2}
\xymatrix{
\O_\cX \ar[r]^<<<<<{\fq_E} &  S^2(E\dual) \ar[r] & S^2(Q) \ar@<.5ex>[r]^<<<<{S^2(1,c)} \ar@<-.5ex>[r]_<<<<{S^2(1,0)} & S^2(Q\oplus B\dual)
}\eeq
coincide, where $c:Q \to B\dual$ is the canonical map induced by $d\dual :E\dual \to B\dual$.
Note that there is a canonical decomposition
\[S^2(Q\oplus B\dual) = (S^2Q) \oplus (Q \otimes B\dual) \oplus (S^2B\dual).\]
Then we can deduce that the two maps in \eqref{q2} coincide since the two maps
\[B \mapright{d} E \mapright{\fq_E} E\dual \to Q \and B \mapright{d} E \mapright{\fq_E} E\dual \mapright{d\dual} B\dual\]
are zero. It proves the claim.

We now claim that the map $\fq_\EE :\fC_\EE \to \A^1_\cX$ in \eqref{q5} is independent of the choice of the symmetric resolution. Indeed, consider two symmetric resolutions
\[[B_1 \to E_1\dual \to B_1\dual] \cong \EE \cong [B_2 \to E_2\dual \to B_2\dual]\]
of $\EE$. Let $Q_1=\coker (B_1\to E_1\dual)$ and $Q_2=\coker(B_2\to E_2\dual)$ be the cokernels. 
Let
\[\fq_1, \fq_2 : \fC_\EE \rightrightarrows \A^1_\cX\]
be the two functions induced by the functions $\fq_{E_1}|_{C_{Q_1}}:C_{Q_1}\to \A^1$ and $\fq_{E_2}|_{C_{Q_2}}:C_{Q_2}\to\A^1$, respectively. We will show that $\fq_1=\fq_2$. Since the statement is local on $\cX$, we may assume that $\cX$ is an affine scheme. Then the identity map $\id:\EE\dual \to \EE\dual$ can be represented by a chain map
\[\xymatrix{
\EE\dual \ar@{=}[d] & B_2 \ar[r]^{d_2} \ar[d]^u & E_2 \ar[r]^{d_2\dual q_2} \ar[d]^v& B_2\dual \ar[d]^{w\dual}\\
\EE\dual & B_1 \ar[r]^{d_1} & E_1 \ar[r]^{d_1\dual q_1} & B_1\dual }\]
for some $u$, $v$, $w$, where $q_1:E_1 \cong E_1\dual$ and $q_2 :E_2 \cong E_2\dual$ are the symmetric forms. Hence, we obtain a commutative diagram
\[\xymatrix{
C_{Q_2}\ar[r] \ar@{.>}[d]^r & E_2 \ar[d]^v \\
C_{Q_1} \ar[r] & E_1
}\]
for a unique dotted arrow $r:C_{Q_2} \to C_{Q_1}$. Consider a diagram
\[\xymatrix{
\EE\dual \ar[d]^{\iota_\theta} & B_2 \ar[r]^{d_2}  \ar@<-.5ex>[d]_{1} \ar@<.5ex>[d]^{w u} & E_2 \ar[r]^{d_2\dual q_2} \ar@<-.5ex>[d]_{q_2} \ar@<.5ex>[d]^{v\dual q_1 v} 
& B_2\dual \ar@<-.5ex>[d]_{1} \ar@<.5ex>[d]^{u\dual w\dual} 
\\
\EE[-2] &B_2 \ar[r]_{q_2 d_2} & E_2 \dual \ar[r]_{d_2\dual} & B_2\dual
}\]
of two chain maps representing the same map $\iota_\theta :\EE\dual \to \EE[-2]$ in the derived category. Since $\cX$ is affine, there exist maps $h:E_2 \to B_2$ and $k\dual : B_2\dual \to E_2\dual$ such that
\[q_2 - v\dual  q_1 v = q_2 d_2  h + k\dual  d_2\dual  q_2 : E_2 \to E_2\dual.\]
Equivalently, if we let $q_1 \in \Gamma(\cX, E_1\dual\otimes E_1\dual)$, $q_2 \in \Gamma(\cX, E_2\dual\otimes E_2\dual)$ by abuse of notation, then we have
\beq\label{10}
q_2 = (v\dual\otimes v\dual)(q_1) + (h\dual d_2\dual \otimes 1)(q_2)+(1\otimes k\dual d_2\dual)(q_2) \eeq
as global sections in $\Gamma(\cX,E_2\dual\otimes E_2\dual)$. If we compose \eqref{10} with the canonical map $E_2\dual \otimes E_2\dual \to Q_2\otimes Q_2$, then we obtain a commutative diagram
\[\xymatrix{
C_{Q_2} \ar[rd]^{\fq_{E_2}} \ar[d]^r & \\
C_{Q_1} \ar[r]_{\fq_{E_1}} & \A^1_\cX
}\]
since the composition
\[\O_\cX \mapright{q_2} E_2\dual\otimes E_2\dual \mapright{} B_2\dual \otimes Q_2\]
is zero. Since the projection map $C_{Q_2} \to \fC_\EE=[C_{Q_2}/B_2]$ is smooth and surjective, we have $\fq_1=\fq_2$. It proves the claim.


We have shown that there exists a well-defined quadratic function $\fq_\EE:\fC_\EE\to \A^1_\cX$ for a symmetric complex $\EE$ on a quasi-projective scheme $\cX$. This can be generalized to an arbitrary scheme $\cX$ since we can always construct the quadratic function locally, and then glue them globally. The result in the previous paragraph assures that this is possible. Moreover, (1) and (2) follow directly from the construction.

We will now prove \eqref{Eq.ReductionFormulaforQuadraticFunction} for an isotropic subcomple $\KK$ of $\EE$. By Proposition \ref{Prop.SymRes}, we can choose a symmetric resolution $[B\to E\dual \to B\dual] \cong \EE$ and a resolution $[0 \to K\dual \to D\dual] \cong \KK$ such that $\delta:\EE\to\KK$ can be represented by a surjective chain map. Then we have induced resolutions
\[\DD \cong [(B/D) \to (K\uperp)\dual \to B\dual] \and \GG \cong[(B/D) \to (K\uperp/K) \dual \to (B/D)\dual]\]
where the second one is a symmetric resolution. Let $Q=\coker(B \to E\dual)$ as before, and let
\[R = \coker((B/D)  \to (K\uperp)\dual) \and S= \coker((B/D) \to (K\uperp/K)\dual)\]
be the cokernels. We claim that the two squares
\beq\label{q7}
\xymatrix{ K\uperp \ar[r] \ar[d] & E \ar[d]^{\fq_E} \\ K\uperp/K \ar[r]^-{\fq_{K\uperp/K}} & \A^1_\cX }\qquad
\xymatrix{
C_R \ar[r] \ar[d] & C_Q \ar[d]^{\fq_E} \\
C_S \ar[r]^-{\fq_{K\uperp/K}} \ar[r] & \A^1_\cX
}\eeq
commute. Indeed, the left square in \eqref{q7} commutes since the quadratic form on the reduction $K\uperp/K$ is induced from the quadratic form of $E$. The right square in \eqref{q7} commutes since it is a restriction of the left square in \eqref{q7}. Since the projection map $C_R \to \fC_{\DD} = [C_R/ B]$ is smooth and surjecitve, the right square in \eqref{q7} implies \eqref{Eq.ReductionFormulaforQuadraticFunction}. 

Note that for any symmetric resolution $[B \to E \to B\dual] \cong \EE$, the symmetric complex $\EE$ is the reduction of $E[1]$ by $B[1]$. The uniqueness of the quadratic function $\fq_\EE$ follows from this observation.
\end{proof}

Using the language of derived algebraic geometry, the quadratic function $\fq_\EE:\fC_\EE \to \A^1_{\cX}$ on the abelian cone stack $\fC_\EE$ of a symmetric complex $\EE$ has the following simple description:

\begin{remark}
Consider the total space  $\mathrm{Tot}(\EE\dual[1])$ of the perfect complex $\EE\dual[1]$ of tor-amplitude $[-1,1]$, which is a derived Artin stack. Then we have a weak homotopy equivalence
\[\mathrm{Map}_{\mathrm{dSt}_{\cX}}(\mathrm{Tot}(\EE\dual[1]), \A^1_{\cX}) \cong \mathrm{Map}_{L_{\mathrm{qcoh}}(\cX)}(\O_\cX,S^{\bullet}(\EE[-1]))\] 
between the mapping spaces. Thus the symmetric form $\theta \in \Gamma(\cX, S^2 (\EE[-1]))$ defines a function
\beq\label{dq1}
\mathrm{Tot}(\EE\dual[1]) \to \A^1_{\cX}.
\eeq
The quadratic function $\fq_{\EE}:\fC_{\EE} \to \A^1_{\cX}$ in Proposition \ref{Prop.QuadraticFunction} is the restriction of the function \eqref{dq1} to its classical truncation $\mathrm{Tot}(\EE\dual[1])_{\mathrm{cl}} =\fC_\EE$.
\end{remark}

\subsection{Symmetric obstruction theory}\label{ss.SOT}

To construct a square root virtual pullback for a symmetric obstruction theory, we need an additional assumption called the {\em isotropic condition}. In this subsection. we define the isotropic condition and explore its basic properties.

We first fix the definition of {\em symmetric obstruction theories}.

\begin{definition}[Symmetric obstruction theory]\label{Def.SOT}
A {\em symmetric obstruction theory} for a morphism $f:\cX\to \cY$ of schemes is a morphism $\phi:\EE\to \bL_f$ in the derived category of $\cX$ such that
\begin{enumerate}
\item $\EE$ is a symmetric complex in the sense of Definition \ref{Def.SymCplx},
\item $\phi$ is an obstruction theory in the sense of Behrend-Fantechi \cite{BeFa}, i.e., $h^0(\phi)$ is bijective and $h^{-1}(\phi)$ is surjective,
\end{enumerate}
where $\bL_f=\tau^{\geq-1}\LL_f$ is the truncated cotangent complex.
\end{definition}


Let $f:\cX\to \cY$ be a morphism of schemes equipped with a symmetric obstruction theory $\phi:\EE\to\bL_f$. The obstruction theory $\phi$ induces a closed embedding
\[\fC_f \hookrightarrow  \fC_{\EE}\]
of the intrinsic normal cone $\fC_f$ to the {\em virtual normal cone} $\fC_\EE$ by \cite[Proposition 2.6]{BeFa}.
Since $\EE$ is a symmetric complex, we have a quadratic function
\[\fq_\EE :\fC_\EE \to \A^1_\cX\]
on the virtual normal cone $\fC_\EE$ by Proposition \ref{Prop.QuadraticFunction}.

\begin{definition}[Isotropic condition]\label{Def.Isotropic}
We say that a symmetric obstruction theory $\phi:\EE \to \bL_f$ satisfies the {\em isotropic condition} if the intrinsic normal cone $\fC_f$ is {\em isotropic} in the virtual normal cone $\fC_\EE$, i.e., the restriction
\[\fq_{\EE}|_{\fC_f} : \fC_f \hookrightarrow \fC_{\EE} \to \A^1\]
of the quadratic function $\fq_{\EE}$ on $\fC_\EE$ to $\fC_f$ vanishes.
\end{definition}


We now compare the definition of the isotropic condition in Definition \ref{Def.Isotropic} and the isotropic condition of Oh-Thomas in \cite[Proposition 4.3]{OT}.

\begin{lemma}[Comparison of isotropic conditions]\label{Lem.ComparisonofIsotropic}
Assume that there exists a symmetric resolution $\EE\cong [B\to E\dual \to B\dual]$. Let $\FF=[E\dual \to B\dual]$ be the stupid truncation. Consider a fiber diagram
\beq\label{l1}
\xymatrix{ C \ar[r] \ar[d]^p   & C_Q \ar[r] \ar[d] & E \ar[d]\\ \fC_f \ar[r] & \fC_\EE \ar[r] & [E/B]}\eeq
where $Q=\coker(B \to E\dual)$, and the bottom horizontal arrows are the closed embeddings induced by the maps
\[\FF=[0 \to E\dual \to B\dual] \to \EE=[B\to E\dual \to B\dual] \to \bL_f.\]
Then $\phi$ satisfies the isotropic condition if and only if the subcone $C$ of $E$ is isotropic. 
\end{lemma}

\begin{proof}
From the commutative diagram
\[\xymatrix
{\FF \ar[r] \ar[d] & E\dual[1] \ar[r] \ar[d] & B\dual[1] \ar@{=}[d]\\
\EE \ar[r] & \FF\dual[2] \ar[r] & B[1]
}\]
we deduce that $\EE$ is the reduction of $E[1]$ by $B[1]$.
Hence by Proposition \ref{Prop.QuadraticFunction}(3), we obtain a commutative diagram
\[\xymatrix{
C_Q \ar[r] \ar[d] & E \ar[d]^{\fq_E}\\
\fC_{\EE} \ar[r]^{\fq_{\EE}} & \A^1
}\]
where $\fC_{\FF\dual[2]} = C_Q$. From the diagram \eqref{l1}, we obtain
\[p^*(\fq_{\EE}|_{\fC_f}) = \fq_E|_C.\]
Since the projection map $p:C \to \fC_f$ is smooth and surjective, $\fq_{\EE}|_{\fC_f}$ vanishes if and only if  $\fq_E|_C$ vanishes. It completes the proof.
\end{proof}


A more refined version of the above lemma is the following criterion:

\begin{proposition}[Criterion for isotropic condition]\label{Prop.Isotropic}
Consider a factorization
\[\xymatrix{
& \widetilde{\cY} \ar[d]^{\overline{f}}\\
\cX \ar@{^{(}->}[ru]^{\tf} \ar[r]_{f} & \cY
}\]
of $f$ by a closed immersion $\tf$ and a smooth morphism $\overline{f}$. Assume that there is a symmetric resolution \eqref{Eq.SymRes} of $\EE$ such that $\phi:\EE\to \bL_f$ is represented by a chain map
\[\xymatrix{
\EE \ar[d]^{\phi} & B \ar[r]^{q\circ d} \ar[d] & E\dual \ar[r]^{d \dual} \ar[d]^{\tphi} & B\dual \ar[d]\\
\bL_f & 0  \ar[r] & \I/\I^2 \ar[r] & \Omega_{\overline{f}}|_\cX 
}\]
where $\I=\I_{\tf}$ is the ideal sheaf.  Then $\phi$ satisfies the isotropic condition if and only if the composition
\[C_{\cX/\widetilde{\cY}} \hookrightarrow N_{\cX/\widetilde{\cY}} \mapright{C({\tphi})} E \mapright{\fq_E} \A^1\]
vanishes.
\end{proposition}

\begin{proof}
As in Lemma \ref{Lem.ComparisonofIsotropic}, we have a commutative diagram
\[\xymatrix{
C_{\tf} \ar[r] \ar[rd]_{r} & C \ar[r] \ar[d] & C_Q \ar[r] \ar[d]& E \ar[d] \\
& \fC_f \ar[r] & \fC_{\EE} \ar[r] & [E/B].
}\]
Since $\EE$ is the reduction of $E[1]$ by $B[1]$, we have
\[r^*(\fq_{\EE}|_{\fC_f}) = \fq_{E}|_{C_{\tf}}\]
by Proposition \ref{Prop.QuadraticFunction}(3).
Since the projection map $r:C_{\tf}\to \fC_f$ is smooth and surjective, the two vanishing conditions are equivalent.
\end{proof}




\subsection{Square root virtual pullback}

Let $f:\cX \to \cY$ be a morphism of quasi-projective schemes equipped with a symmetric obstruction theory $\phi:\EE\to\bL_f$ satisfying the isotropic condition. By Proposition \ref{Prop.SymRes}, we can choose a symmetric resolution $[B \to E\dual \to B\dual] \cong \EE$. The stupid truncation $[E\dual \to B\dual] =\FF$ gives us a fiber diagram \[\xymatrix{ C \ar[r] \ar[d]^p & C_Q \ar[r] \ar[d] & E \ar[d]\\ \fC_f \ar[r] & \fC_\EE \ar[r] & [E/B]}\] where $C$ is an isotropic subcone of $E$ by Lemma \ref{Lem.ComparisonofIsotropic}. Let $\tau \in \Gamma(C,E|_C)$ be the tautological section. Then the zero section $0 :X \hookrightarrow C$ is the zero locus of the tautological section $\tau$.

\begin{definition}[Square root virtual pullback]\label{Def.sqrtVP}
The {\em square root virtual pullback} is defined as the composition:
\[\sqrt{f^!} : A_* (\cY) \xrightarrow{\mathrm{sp}_f} A_* (\fC_f) \xrightarrow{p^*} A_{*+\rank(B)} (C) \xrightarrow{\sqrt{e}(E|_C,\tau)} A_{*+\frac12\rank{\EE}} (\cX),\]
where $\mathrm{sp}_f:A_*(\cY)\to A_*(\fC_f)$ denotes the specialization map \cite{Kresch,Man}, and $\sqe(E|_C,\tau)$ denotes the localized square root Euler class \cite{OT}.
\end{definition}


\begin{lemma}[cf. \cite{OT}]
The square root virtual pullback is independent of the choice of the symmetric resolution.
\end{lemma}

\begin{proof}
The proof is analogous to that in \cite[pp. 28-31]{OT}. We just need to replace the fundamental classes of cones in \cite{OT} with the specialization maps. The functorial properties in \cite[Lemma 4.4]{KP} and \cite[Corollary 4.7]{KP} assures that this is possible. We omit the details.
\end{proof}

Consider a fiber diagram
\[\xymatrix{
\cX' \ar[r]^{f'} \ar[d]^{g'} & \cY' \ar[d]^g\\
\cX \ar[r]^f & \cY
}\]
of quasi-projective schemes. Let $a_g:\fC_{f'} \hookrightarrow g'^*\fC_f$ denote the canonical closed embedding \cite[Proposition 2.26]{Man}. The composition
\beq\label{Eq.BaseChange} \phi':g'^*\EE \xrightarrow{g'^*\phi} g'^*\bL_{f} \to\bL_{f'} \eeq
is also a symmetric obstruction theory for $f':\cX'\to \cY'$, and satisfies the isotropic condition by Proposition \ref{Prop.QuadraticFunction}(2). Thus we have an induced square root virtual pullback
\[\sqrt{(f')^!} : A_* (\cY') \to A_*(\cX').\]
By abuse of notation, we also denote $\sqrt{(f')^!}$ by $\sqrt{f^!}$.

\begin{proposition}[Bivariance]\label{Prop.Bivar}
The square root virtual pullback is a bivariant class, i.e.,
\begin{enumerate}
\item If $g:\cY'\to \cY$ is projective, 
$\sqrt{f^!}\circ g_* = g'_* \circ \sqrt{f^!}.$
\item If $g:\cY' \to \cY$ is flat and equi-dimensional, 
$\sqrt{f^!} \circ g^* = (g')^*\circ \sqrt{f^!}.$
\item If $g:\cY' \to \cY$ is a regular immersion, 
$\sqrt{f^!} \circ g^! = g!\circ \sqrt{f^!}.$
\end{enumerate}
\end{proposition}

\begin{proof}
By the proof of \cite[Theorem 4.1]{Man} and \cite[Theorem 4.3]{Man} (cf. \cite[Lemma 4.10]{KPVIT}), the map
\[A_*(\cY') \xrightarrow{\sp_{f'}} A_*(\fC_{f'})\xrightarrow{(a_g)_*} A_* (\fC_f|_{\cX'})\]
commutes with projective pushforwards, flat pullbacks, and Gysin pullbacks of regular immersions. Since the localized square root Euler class $\sqe(E|_C,\tau)$ is a bivariant class by \cite[Lemma 4.4]{KP}, the square root virtual pullback $\sqrt{f^!}=\sqe(E|_C,\tau)\circ p^* \circ \sp_f$ is also a bivariant class.
\end{proof}


\begin{definition}[Oh-Thomas virtual cycle]
Let $\cX$ be a quasi-projective scheme equipped with a symmetric obstruction theory satisfying the isotropic condition. The {\em Oh-Thomas virtual cycle} is defined as
\[[\cX]\virt := \sqrt{p^!}[\spec(\C)] \in A_*(\cX)\]
where $p: \cX\to \spec(\C)$ denotes the projection map.
\end{definition}

\begin{example}[-2-shifted symplectic derived scheme]
Let $\bX$ be a derived scheme equipped with a -2-shifted symplectic form and an orientation. Then the classical truncation $\cX=\bX_{\mathrm{cl}}$ carries a symmetric obstruction theory 
$\LL_{\bX}|_{\cX} \to \LL_{\cX} \to \bL_{\cX}.$
The Darboux theorem \cite{BBJ} assures that this obstruction theory satisfies the isotropic condition. Indeed, this can be shown directly from MacPherson's graph construction \cite[Remark 5.1.1]{Ful} and the criterion in Proposition \ref{Prop.Isotropic}. Thus $\cX$ carries an Oh-Thomas virtual cycle $[\cX]\virt \in A_*(\cX)$ when $\cX$ is quasi-projective.
\end{example}

\subsection{Reduction formula}


When a symmetric obstruction theory can be factored by a perfect obstruction theory, we can decompose the square root virtual pullback into Manolache's virtual pullback \cite{Man} and Edidin-Graham's square root Euler class \cite{EG}.


\begin{proposition}[Reduction formula]\label{Prop.ReductionFormula}
Let $f:\cX \to \cY$ be a morphism of quasi-projective schemes equipped with a symmetric obstruction theory $\phi:\EE \to \bL_f$. 
Let $\KK$ be an isotropic subcomplex of $\EE$ with respect to $\delta:\EE\to\KK$ such that $h^0(\delta)$ is bijective and $h^{-1}(\delta)$ is surjective. Assume that there is a map $\psi : \KK \to \bL_f$ such that $\phi = \psi \circ \delta$.
Then the reduction of $\EE$ by $\KK$ is $G[1]$ for some special orthogonal bundle $G$, $\phi$ satisfies the isotropic condition, $\psi$ is a perfect obstruction theory, and we have a {\em reduction formula}
\[\sqrt{f^!_{\phi}}=\sqe(G) \circ f^!_{\psi}\]
where $\sqrt{f^!_{\phi}}$ denotes the square root virtual pullback for $\phi:\EE\to\bL_f$ and $f^!_{\psi}$ denotes the virtual pullback for $\psi:\KK\to\bL_f$.
\end{proposition}

\begin{proof}
By Proposition \ref{Prop.SymRes}, we can choose a symmetric resolution $[B\to E\dual \to B\dual] \cong \EE$ and a resolution $[0 \to K\dual \to D\dual] \cong \KK$ such that the map $\delta : \EE \to \KK$ is represented by a surjective chain map. Since $h^0(\delta)$ is bijecitve and $h^{-1}(\delta)$ is surjective, the map $E\dual \to B\dual \times_{D\dual}K\dual$ is surjective. Replacing $[K\dual \to D\dual]$ by $[B\dual \times_{D\dual}K\dual \to B\dual]$, we may assume that $D=B$. 

Let $G=K\uperp/K$ be the reduction of $E$ by $K$. Then the reduction of $\EE$ by $\KK$ is $G[1]$. Form a fiber diagram
\[\xymatrix{
C \ar[r] \ar[d]^p & K \ar[r] \ar[d] & C_Q \ar[r] \ar[d] & E \ar[d]\\
\fC_f \ar[r] & \fC_\KK \ar[r] & \fC_\EE \ar[r] & \fC_\FF
}\]
where $\FF=[E\dual \to B\dual]$ is the stupid truncation of $[B\to E\dual \to B\dual]$. Since $K$ is an isotropic subbundle of $E$, the subcone $C$ is also isotropic. By Lemma \ref{Lem.ComparisonofIsotropic}, the symmetric obstruction theory $\phi:\EE\to\bL_f$ satisfies the isotropic condition. Recall that
\[f^!_{\psi} = e(K|_C,\tau) \circ p^* \circ \sp_f  \and \sqrt{f^!_\phi} = \sqrt{e}(E|_C,\tau) \circ p^* \circ \sp_f\]
where $\sp_f:A_*(Y)\to A_*(\fC_f)$ denotes the specialization map and $\tau \in \Gamma(C,K|_C)$ denotes the tautological section. The formula \eqref{Eq.LocFunct} in Proposition \ref{Prop.LocFunct} completes the proof.
\end{proof}

\begin{corollary}[Local complete intersection]\label{Cor.sqrtVPforlci}
Assume that $f:\cX\to \cY$ is a local complete intersection morphism equipped with a symmetric obstruction theory $\phi:\EE \to \bL_f$. Then the reduction of $\EE_f$ by $\LL_f$ is $G[1]$ for some special orthogonal bundle $G$ over $\cX$, the symmetric obstruction theory $\phi_f$ is isotropic, and we have $\sqrt{f^!} = \sqe(G) \circ f^!$.
\end{corollary}

\section{Functoriality}\label{S.Functoriality}

In this section, we prove {\em functoriality} of square root virtual pullbacks.

\subsection{Main result}
Consider a commutative diagram
\beq\label{C.XYZ} \xymatrix{\cX \ar[r]^f \ar@/_0.4cm/[rr]_{g\circ f} & \cY \ar[r]^g & \cZ }\eeq
of quasi-projective schemes. The canonical distinguished triangle of cotangent complexes \[\xymatrix{ f^*\LL_g \ar[r] & \LL_{g\circ f} \ar[r] & \LL_f \ar[r] & }\] induces a distinguished triangle
\beq\label{T.CotanCplx.XYZ} \xymatrix{ \trunc f^*\bL_g \ar[r]^-{a} & \bL_{g\circ f} \ar[r]^-{b} & \bL_f' \ar[r] & }\eeq
where $\bL_g:=\trunc \LL_g$ and $\bL_{g\circ f}:=\trunc \LL_{g \circ f}$ are the truncated cotangent complexes, and $\bL_f'$ is the cone of $a$. Let $r: \bL_f' \to \trunc \bL_f' \cong  \bL_f$ denote the canonical map.


\begin{definition}[Compatibility condition]\label{Def.Compatibility}
Given a commutative diagram \eqref{C.XYZ}, we say that a triple $(\phi_f,\phi_g,\phi_{g\circ f})$ of symmetric obstruction theories $\phi_{g}:\EE_g \to \bL_g$, $\phi_{g\circ f}:\EE_{g\circ f} \to \bL_{g\circ f}$, and a perfect obstruction theory $\phi_f:\EE_f \to \bL_f$ is {\em compatible} if
\begin{enumerate}
\item there exist two morphisms of distinguished triangles
\beq\label{Eq.Compatibility}
\xymatrix{
\DD\dual[2] \ar[r]^{\alpha\dual} \ar[d]^{\beta\dual} & \EE_{g \circ f} \ar[r]^{\delta} \ar[d]^{\alpha} & \EE_f  \ar@{=}[d] \ar[r] & \\
f^*\EE_g \ar[r]^{\beta} \ar[d]^{f^*\phi_g} & \DD \ar[r]^{\gamma} \ar[d]^{\phi_{g\circ f}'} & \EE_f  \ar[d]^{\phi_f'} \ar[r] & \\
\tau^{\ge-1}f^*\bL_g \ar[r]^a  & \bL_{g\circ f} \ar[r]^b & \bL_f' \ar[r] & 
}\eeq
such that $\phi_{g\circ f}=\phi'_{g\circ f} \circ \alpha$ and $\phi_f = r \circ \phi_f'$;
\item the orientation of $\EE_{g\circ f}$ is given by the orientation of $\EE_g$ via the canonical isomorphism $\det(\EE_{g\circ f}) \cong \det(f^*\EE_g)$ induced by \eqref{Eq.Compatibility}.
\end{enumerate}
\end{definition}


The compatibility condition in Definition \ref{Def.Compatibility} is slightly general than that in Theorem \ref{Thm.VirtualPullbackFormula} since the truncated cotangent complexes are used instead of the full cotangent complexes. This generality will be needed in the applications in \S\ref{S.Lefschetz} and \S\ref{S.Pairs/Sheaves}.
 
\begin{theorem}[Functoriality]\label{Thm.Functoriality}
Consider a commutative diagram \eqref{C.XYZ} of quasi-projective schemes equipped with a compatible triple $(\phi_f,\phi_g,\phi_{g\circ f})$ of obstruction theories. Assume that $\phi_g$ and $\phi_{g\circ f}$ satisfy the isotropic condition in Definition \ref{Def.Isotropic}. Then we have 
\begin{equation}\label{Eq.Functoriality}
\sqrt{(g\circ f)^!} = f^! \circ \sqrt{g^!}.
\end{equation}

In particular, if $\cZ=\spec(\C)$, then we have a {\em virtual pullback formula}
\[[\cX]\virt = f^! [\cY]\virt\]
between the Oh-Thomas virtual cycles.
\end{theorem}

The isotropic condition for $\phi_{g\circ f}$ is redundant in Theorem \ref{Thm.Functoriality}.

\begin{lemma}\label{Lem.IsotropicConditioninFunctoriality}
Given a compatible triple $(\phi_f,\phi_g,\phi_{g\circ f})$ of obstruction theories for \eqref{C.XYZ}, if $\phi_g$ satisfies the isotropic condition, then so does $\phi_{g\circ f}$.
\end{lemma}

\begin{proof}
By Proposition \ref{Prop.QuadraticFunction}(3), the diagram
\[\xymatrix{
\fC_{g\circ f} \ar[r] \ar[d] & \fC(\DD) \ar[r] \ar[d] & \fC(\EE_{g\circ f}) \ar[d]^{\fq(\EE_{g\circ f})}\\
f^*\fC_{g} \ar[r] & \fC(f^*\EE_g) \ar[r]^-{\fq(f^*\EE_g)} & \A^1_{\cX}
}\]
commutes. By Proposition \ref{Prop.QuadraticFunction}(2), we have $\fq(f^*\EE_g)=f^*\fq(\EE_g)$. Hence the isotropic condition for $\phi_g$ implies the isotropic condition for $\phi_{g\circ f}$.
\end{proof}

The rest of this section is devoted to the proof of Theorem \ref{Thm.Functoriality}. Basically, the structure of the proof is similar to the standard arguments of functoriality in \cite{Ful,KKP,Man,CKL,Qu,KPVIT}. The additional ingredients for Theorem \ref{Thm.Functoriality} are the followings:
\begin{enumerate}
\item Blowup method introduced in \cite{KP} (cf. \cite{KL});
\item Reduction operation of symmetric complexes in Proposition \ref{Prop.Reduction};
\item Criterion for isotropic condition in Proposition \ref{Prop.Isotropic}.
\end{enumerate}

\subsection{Special case via blowup}

We first prove Theorem \ref{Thm.Functoriality} when $g$ is a closed immersion using the blowup method as in \cite[Lemma 4.5]{KP}.

\begin{lemma}\label{Lem.specialcaseviablowup}
Theorem \ref{Thm.Functoriality} holds if $g: \cY \to \cZ$ is a closed immersion.
\end{lemma}

\begin{proof}
Let $\widetilde{\cZ}=\mathrm{Bl}_{\cY} \cZ$ be the blowup of $\cZ$ along $\cY$. Form fiber diagrams
\[\xymatrix{
\widetilde{\cX} \ar[r]^{\tf} \ar[d] & \widetilde{\cY} \ar[r]^{\tg} \ar[d] & \widetilde{\cZ} \ar[d]^{\rho}\\
\cX \ar[r]^f & \cY \ar[r]^g & \cZ
} \qquad \xymatrix{
\cX \ar[r]^{f} \ar[d]^{\id} & \cY \ar[r]^{\id} \ar[d]^{\id} & \cY \ar[d]^{g}\\
\cX \ar[r]^f & \cY \ar[r]^g & \cZ .}\]
By the blowup sequence \cite[Example 1.8.1]{Ful}, we have a surjective map
\[(\rho_*,g_*):A_* (\widetilde{\cZ}) \oplus A_*(\cY) \to A_*(\cZ).\]
Since the both sides of \eqref{Eq.Functoriality} commute with projective pushforwards by Proposition \ref{Prop.Bivar}(1), and the compatibility condition in Definition \ref{Def.Compatibility} is stable under the base change of $\cZ$, it suffice to prove \eqref{Eq.Functoriality} for 
\[\xymatrix{\widetilde{\cX} \ar[r]^{\tf} \ar@/_0.4cm/[rr]_-{\tg\circ \tf} & \widetilde{\cY} \ar[r]^-{\tg} & \widetilde{\cZ} } \and
\xymatrix{\cX \ar[r]^f \ar@/_0.4cm/[rr]_-{f} & \cY \ar[r]^-{\id} & \cY } \,.\]
Since the reduction formula in Proposition \ref{Prop.ReductionFormula} proves the latter case, it remains to prove the former case. After replacing $\widetilde{\cX}$, $\widetilde{\cY}$, $\widetilde{\cZ}$ by $\cX$, $\cY$, $\cZ$, we may assume that the closed immersion $g:\cY \to \cZ$ is a local complete intersection morphism.

Since $g:\cY \to \cZ$ is a regular immersion, the symmetric obstruction theory $\phi_g$ can be expressed as
\[\phi_g : \EE_{g} = G \dual[1] \longrightarrow \bL_g=\LL_g =N\dual[1].\] 
for some special orthogonal bundle $G$ over $\cY$, where $N=N_{\cY/\cZ}$ is the normal bundle of $\cY$ in $\cZ$. Since $\phi_g$ satisfies the isotropic condition, $N$ is an isotropic subbundle of $G$.

Form a morphism of distinguished triangles 
\beq\label{Eq25}\xymatrix{
f^*\LL_{g} \ar[r]^{} \ar@{=}[d] & \KK \ar[r]^-{\eta} \ar@{.>}[d]^{\psi}& \EE_{f}  \ar[d]^{\phi_f'} \ar[r] & \\
f^*\LL_{g} \ar[r]^{} & \bL_{g \circ f} \ar[r]^{}  & \bL_f  '\ar[r]& 
}\eeq
for some perfect complex $\KK$ and maps $\psi$ and $\eta$. Then $\psi:\KK \to \bL_{g\circ f}$ is a perfect obstruction theory in the sense of Behrend-Fantechi. By Manolache's virtual pullback formula \cite[Theorem 4.8]{Man},\footnote{Here the compatibility condition \eqref{Eq25} is slightly general that than in Manolache \cite{Man} since we are considering the truncated cotangent complexes. However, exactly the same proof works for this generalized compatibility condition.} the diagram \eqref{Eq25} gives us
\begin{equation}\label{Eq18}
(g\circ f)^!_{\psi} = f^! \circ g^!,    
\end{equation}
where $(g\circ f)^!_{\psi}$ denotes the virtual pullback for $\psi$ and $g^!$ denotes the ordinary Gysin pullback.

Form a morphism of distinguished triangles
\[\xymatrix{
f^*\EE_g \ar[r]^{\beta} \ar[d]^{f^*\phi_g } & \DD \ar[r]^{\gamma} \ar@{.>}[d]^{\zeta} & \EE_{f} \ar@{=}[d] \ar[r]&  \\
f^*\LL_{g} \ar[r]^{} & \KK \ar[r]^-{\eta} & \EE_{f} \ar[r] & }\]
for some map $\zeta$. Then $\KK$ is an isotropic subcomplex of $\EE_{g\circ f}$ with respect to $\zeta\circ \alpha : \EE_{g\circ f} \to \KK$ since $\phi_g$ satisfies the isotropic condition. Note that the reduction of $\EE_{g\circ f}$ by $\EE_f$ is $f^*\EE_g$, and the reduction of $f^*\EE_g=f^*G[1]$ by $f^*\LL_g=f^*N\dual[1]$ is $f^*(N\uperp/N)[1]$. We claim that 
\beq\label{F1}[\KK\dual[2] \to \EE_{g\circ f} \to \KK ] = f^*(N\uperp/N)[1].\eeq
Indeed, choose a symmetric resolution of $\EE_{g\circ f}$ and resolutions of $\KK$, $\EE_f$ such that the two maps $\EE_{g\circ f} \mapright{\zeta\circ \alpha} \KK \mapright{\eta} \EE_f$ are represented by surjective chain maps. As in the proof of Proposition \ref{Prop.ReductionFormula}, we may assume that the degree 0 terms of the resolutions of $\EE_{g\circ f}$, $\KK$, $\EE_f$ are isomorphic since $h^0(\zeta\circ\alpha)$, $h^0(\eta)$ are bijective and $h^{-1}(\zeta\circ\alpha)$, $h^{-1}(\eta)$ are surjective. Then \eqref{F1} follows immediately.

By the reduction formula in Proposition \ref{Prop.ReductionFormula}, we have
\beq\label{Eq29}\sqrt{(g\circ f)^!_{\psi\circ \zeta\circ \alpha}} = \sqe(N\uperp/N)\circ (g\circ f)^!_{\psi}\eeq
where the left-hand side of \eqref{Eq29} is the square root virtual pullback  for the induced symmetric obstruction theory $\psi \circ \zeta \circ \alpha : \EE_{g\circ f} \to \bL_{g\circ f}$. Combining the two equations \eqref{Eq18} and \eqref{Eq29}, we obtain
\[\sqrt{(g\circ f)^!_{\psi\circ \zeta\circ \alpha}} = f^! \circ \sqrt{g^!}.\]
since $\sqrt{g^!} =\sqe(N\uperp/N)\circ g^!$ by Corollary \ref{Cor.sqrtVPforlci}.

Finally, consider the two morphisms
\[\xymatrix{
f^*\EE_g \ar[r]^{\beta} \ar[d]^{f^*\phi_g} & \DD \ar[r]^{\gamma} \ar@<-.5ex>[d]_{\phi_{g\circ f}'} \ar@<.5ex>[d]^{\psi\circ\zeta} & \EE_f  \ar[d]^{\phi_f'} \ar[r] & \\
\tau^{\ge-1}f^*\bL_g \ar[r]^a  & \bL_{g\circ f} \ar[r] & \bL_f' \ar[r] &
}\]
of distinguished triangles. By Lemma \ref{Lem.IsotropicConditioninFunctoriality}, the maps
\[\left((1-t)\cdot (\phi'_{g\circ f}) + (t) \cdot (\psi \circ \zeta) \right) \circ \alpha:\EE_{g\circ f} \to \bL_{g\circ f}\]
for all $t \in \A^1$ give us a family of symmetric obstruction theories satisfying the isotropic condition. Since square root virtual pullbacks are bivariant classes by Proposition \ref{Prop.Bivar}, we can deduce
\[\sqrt{(g\circ f)^!} = \sqrt{(g\circ f)^!_{\psi \circ \zeta \circ \alpha}}\]
by a deformation argument. It completes the proof.
\end{proof}

\subsection{Cone stack case}

We then consider the case when $\cZ=\fC$ is a cone stack over $\cY$. After choosing a presentation $C \to \fC$ of $\fC$ by a cone $C$, we will reduce the situation as
\beq\label{Eq.ConeStacktoCone} 
\xymatrix{\cX \ar[r]^f \ar@/_0.4cm/[rr]_-{k\circ f} & \cY \ar[r]^-{k} & \fC } \leadsto
\xymatrix{\cX \ar[r]^f \ar@/_0.4cm/[rr]_-{\tk\circ f} & \cY \ar[r]^-{\tk} & C } \eeq
by lifting the obstruction theories. Since the zero section $\tk : \cY \to C$ is a closed embedding, we can apply Lemma \ref{Lem.specialcaseviablowup} to the latter case in \eqref{Eq.ConeStacktoCone} and deduce Theorem \ref{Thm.Functoriality} for the former case in \eqref{Eq.ConeStacktoCone}.

\begin{remark}
Note that everything in \S\ref{S.sqrtVP} can be generalized to a morphism $f:\cX\to \cY$ from a quasi-projective scheme $\cX$ to an algebraic stack $\cY$ in a straightforward manner. Indeed, we can define a symmetric obstruction theory, the isotropic condition, and the square root virtual pullback as in Definition \ref{Def.SOT}, Definition \ref{Def.Isotropic}, and Definition \ref{Def.sqrtVP}. The criterion in Proposition \ref{Prop.Isotropic} holds if we further assume that $\overline{f}$ is quasi-projective. Proposition \ref{Prop.Bivar} also holds if we further assume that $g$ is quasi-projective. 
\end{remark}

\begin{lemma}\label{Lem.conestackcase}
Consider a commutative diagram 
\[\xymatrix{\cX \ar[r]^f \ar@/_0.4cm/[rr]_{k\circ f} & \cY \ar[r]^k & \fC }\]
of quasi-projective schemes $\cX$, $\cY$ and a cone stack $\fC$ over $\cY$, where $k: \cY \to \fC$ is the zero section. Consider a symmetric obstruction theory $\phi_k : \EE_k \to \bL_k$ satisfying the isotropic condition and a perfect obstruction theory $\phi_f : \EE_f \to \bL_f$. Note that $\bL_{k\circ f} \cong \trunc f^*\bL_k \oplus \bL_f$ by \cite[Lemma 2.10]{Qu}. Then
\[\phi_{k\circ f}=\begin{pmatrix}f^*\phi_k & \xi & 0\\ 0 & \phi_f & 0\end{pmatrix} : f^*\EE_k \oplus (\EE_f \oplus \EE_f\dual[2]) \to \tau^{\geq-1} f^*\bL_k \oplus \bL_f =\bL_{k \circ f}\]
is a symmetric obstruction theory satisfying the isotropic condition for any map $\xi$ and gives us
\beq\label{Eq.ConeStack}
\sqrt{(k\circ f)^!} = f^! \circ \sqrt{k^!}.\eeq 
\end{lemma}


\begin{proof}
It is easy to show that $(\phi_f,\phi_k,\phi_{k\circ f})$ form a compatible triple of obstruction theories with obvious distinguished triangles. Hence $\phi_{k \circ f}$ is a symmetric obstruction theory satisfying the isotropic condition by Lemma \ref{Lem.IsotropicConditioninFunctoriality}. It remains to prove the formula \eqref{Eq.ConeStack}.

Note that replacing $\xi$ by $t \cdot \xi$ for any $t \in \C$ gives us a symmetric obstruction theory. Hence, by a deformation argument, we may assume that $\xi=0$.

By Proposition \ref{Prop.SymRes}, we can choose a symmetric resolution $\EE_{k} \cong [B\to E \to B\dual]$. Let $C=\fC \times_{[E/B]}E$ and consider a commutative diagram
\[\xymatrix{
&& C \ar[d]^{p}\\
\cX \ar[r]_f & \cY \ar[r]_k \ar[ru]^{\tk} & \fC 
}\]
where $\tk:\cY \to C$ denotes the zero section. The inclusion $C \subseteq E$ defines a symmetric obstruction theory  $\phi_{\tk} : \EE_{\tk}=E[1] \to \LL_{\tk},$ which is isotropic by Lemma \ref{Lem.ComparisonofIsotropic}. 
Let
\[\phi_{\tk \circ f}=(f^*\phi_{\tk} \oplus \phi_f \oplus 0)  : f^*\EE_{\tk} \oplus (\EE_f \oplus \EE_f\dual[2])  \to \trunc f^*\bL_{\tk} \oplus\bL_f  = \bL_{\tk \circ f} .\]
Then $(\phi_f, \phi_{\tk}, \phi_{\tk \circ f})$ is also a compatible triple of obstruction theories and thus $\phi_{\tk \circ f}$ is isotropic by Lemma \ref{Lem.IsotropicConditioninFunctoriality}. Since $\tk:Y \to C$ is a closed immersion, Lemma \ref{Lem.specialcaseviablowup} gives us
\beq\label{n2}
\sqrt{(\tk \circ f)^!} = f^! \circ \sqrt{(\tk)^!}.
\eeq

Choose a resolution $\EE_f\cong [K\dual \to D \dual]$ and form a fiber diagram
\[\xymatrix{
C' \ar[rr] \ar[d]^{s} & &E \times K \times K\dual \ar[d]\\
\fC_{\tk \circ f} \ar[r] \ar[d]^{r} &  f^*N_{\tk} \times \fN_f \ar[r] \ar[d] & 
E\times [K/D]\times K\dual\ar[d]\\
\fC_{k \circ f} \ar[r] & f^*\fN_{k} \times \fN_f \ar[r] & 
[E/B] \times [K/D] \times K\dual
}\]
where the horizontal arrows are closed embedding and the vertical arrows are smooth morphisms. Since $M^\circ_{\tk \circ f} \to M^\circ_{k \circ f}$ is smooth, we can easily show
\[r^* \circ \sp_{k \circ f} = \sp_{\tk \circ f} \circ p^*.\]
Therefore, we have
\begin{align}\label{n3}
\sqrt{(k\circ f)^!} &=\sqe(E \oplus (K\oplus K\dual),\tau) \circ s^* \circ r^* \circ\sp_{k \circ f}\\
&=\sqe(E \oplus (K\oplus K\dual),\tau) \circ s^* \circ\sp_{\tk \circ f}\circ p^* = \sqrt{(\tk \circ f)^!}\circ p^* \nonumber
\end{align}
where $\tau \in \Gamma(C',E \oplus (K \oplus K\dual))$ denotes the tautological section. On the other hand, we have
\beq\label{n4}
\sqrt{k^!} =  \sqrt{e}(E|_{C},\tau)\circ p^* =  \sqrt{(\tk)^!}\circ p^*.\eeq
since $\sp_k=\id$ (cf. \cite[Lemma 2.8]{Qu}). The three equations \eqref{n2}, \eqref{n3} and \eqref{n4} completes the proof.
\end{proof}


\subsection{Deformation to the normal cone}

As in \cite{Ful,KKP,Man,Qu,KPVIT}, we will use a deformation argument to replace the commutative diagram \eqref{C.XYZ} as
\[\xymatrix{\cX \ar[r]^f \ar@/_0.4cm/[rr]_-{g\circ f} & \cY \ar[r]^-{g} & \cZ } \leadsto
\xymatrix{\cX \ar[r]^f \ar@/_0.4cm/[rr]_-{k\circ f} & \cY \ar[r]^-{k} & \fC_{g} }\] where $\fC_g$ denotes the intrinsic normal cone of $g$. Then Lemma \ref{Lem.conestackcase} will complete the proof of Theorem \ref{Thm.Functoriality}.

Given a commutative diagram \eqref{C.XYZ} of quasi-projective schemes, let $M^\circ_g$ be the deformation space of $g:\cY \to \cZ$. Consider the canonical cartesian diagram
\[\xymatrix{
\cX \ar[r]^{g\circ f} \ar[d] & \cZ \ar[d] \ar[r] & \{1\} \ar[d]^{i_1} \\
\cX\times \A^1 \ar[r]^h & {(M^{\circ}_g)'} \ar[r] & \A^1\\
\cX \ar[r]^{k\circ f} \ar[u] & \fC_{g} \ar[u] \ar[r] & \{0\} \ar[u]_{i_0} 
}\]
where {$(M^{\circ}_g)'$} denotes the open substack of $M^\circ_g$ consists of the fibers of $M^\circ_{g} \to \PP^1$ over $\A^1=\PP^1\setminus\{0\}$, and $k : \cY \to \fC_g$ denotes the zero section. Note that we have a canonical isomorphism $\bL_k=\bL_g$ between the truncated cotangent complexes.

\begin{proposition}\label{Prop.DeformationtoNormalCone}
There exists a symmetric obstruction theory
\[\phi_h : \EE_h \to \bL_h\]
satisfying the isotropic condition such that the fibers of $\phi_h$ over $\lambda \in \A^1$ (see \eqref{Eq.BaseChange}) are given as follows:
\begin{enumerate}
\item If $\lambda\neq 0$, then
\beq\label{D.Eq27}
(\phi_h)_\lambda=\phi_{g\circ f} : \EE_{g\circ f} \to \bL_{g\circ f}.\eeq
\item If $\lambda=0$, then 
\beq\label{D.Eq28}(\phi_h)_0=\begin{pmatrix}f^*\phi_g & \xi & 0\\ 0 & r \circ \eta & 0\end{pmatrix} : f^*\EE_g \oplus \EE_f \oplus \EE_f\dual[2] \to \tau^{\geq-1} f^*\bL_g \oplus \bL_f =\bL_{k \circ f}
\eeq
for some maps $\xi$ and $\eta:\EE_f \to \bL_f'$ such that $\eta$ fits into a morphism
\beq\label{D1}
\xymatrix{
f^*\EE_g \ar[r]^{\beta} \ar[d]^{\phi_g} & \DD \ar[r]^{\gamma} \ar[d]^{\phi_{g\circ f}'} & \EE_f \ar@{.>}[d]^{\eta} \ar[r] & \\
\tau^{\geq-1} f^*\bL_g \ar[r]^a & \bL_{g \circ f} \ar[r] & \bL_f' \ar[r] & 
}\eeq
of distinguished triangles over $\cX$.
\end{enumerate}
\end{proposition}

Let us first assume that Proposition \ref{Prop.DeformationtoNormalCone} holds. Then Theorem \ref{Thm.Functoriality} follows immediately from the following simple argument.

\begin{proof}[Proof of Theorem \ref{Thm.Functoriality}]
By Proposition \ref{Prop.DeformationtoNormalCone} and the bivariance of square root virtual pullbacks in Proposition \ref{Prop.Bivar}, we have
\beq\label{D2}
\sqrt{(g\circ f)^!} = \sqrt{(k \circ f)^!} \circ \sp_g
\eeq
where $\sp_g : A_* (\cZ) \to A_* (\fC_g)$ denotes the specialization map. By Lemma \ref{Lem.conestackcase}, we have
\beq\label{D3}
\sqrt{(k\circ f)^!} = f^!_{\eta} \circ \sqrt{k^!}
\eeq
where $f^!_{\eta}$ denotes the virtual pullback associated to the perfect obstruction theory $\EE_f \mapright{\eta} \bL_f' \mapright{r} \bL_f$. Since the maps
\[(1-t)\cdot \phi_f' + t \cdot \eta:\EE_f \to \bL_f'\]
for all $t \in \CC$ fit into the diagram \eqref{D1} as the dotted arrow, a deformation argument proves
\beq\label{D4}
f^!= f^!_\eta. 
\eeq
Since $\sqrt{g^!} = \sqrt{k^!} \circ \sp_g$ by Definition \ref{Def.sqrtVP}, the three equations \eqref{D2}, \eqref{D3}, and \eqref{D4} proves the functoriality \eqref{Eq.Functoriality}. 
\end{proof}

The rest of this subsection is devoted to the proof of Proposition \ref{Prop.DeformationtoNormalCone}.

\subsubsection{Review on $M^\circ _g$ and $\LL_h$} 

We first review the basic facts on the deformation space $M^\circ _g$ and the cotangent complex $\LL_h$ from \cite{Ful,KKP}.

Consider a factorization of the diagram \eqref{C.XYZ} as 
\beq\label{Fact.XYZ}
\xymatrix{
\cX \ar@{^{(}->}[r]^{\tf} \ar[rd]_f & \widetilde{\cY} \ar@{^{(}->}[r]^{\tg} \ar[d] & \widetilde{\cZ} \ar[d]^{\overline{f}}\\
 & \cY \ar@{^{(}->}[r] \ar[rd]_g & \cZ' \ar[d]^{\overline{g}}\\
 & & \cZ
}\eeq
where the square is cartesian, the horizontal arrows are closed embeddings, and the vertical arrows are smooth. 
Let $I=I_{\widetilde{\cY}/\widetilde{\cZ}}$ and $J=I_{\cX/\widetilde{\cZ}}$ be the ideal sheaves. 
Then the canonical maps between the truncated cotangent complexes 
induced by the distinguished triangle \eqref{T.CotanCplx.XYZ} can be represented by a canonical right exact sequence
\beq\label{Res.CotanCplx.XYZ}
\xymatrix{
& \trunc f^*\bL_g  \ar[r]  \ar@{}[d]|{\rotatebox[origin=c]{90}{=}} & \bL_{g\circ f} \ar[r] \ar@{}[d]|{\rotatebox[origin=c]{90}{=}}& \bL_f \ar@{}[d]|{\rotatebox[origin=c]{90}{=}} &  \\
& {I}/{IJ} \ar[r]^{}\ar[d]  & {J}/{J^2} \ar[r]\ar[d] & {J}/{(J^2+I)} \ar[r] \ar[d] & 0 \\
0 \ar[r] &\Omega_{\overline{g}}|_{\cX} \ar[r]
 & \Omega_{\overline{g}\circ \overline{f}}|_{\cX} \ar[r] 
& \Omega_{\overline{f}}|_{\cX} \ar[r]
 & 0
}\eeq
of chain complexes. 

The factorization \eqref{Fact.XYZ} induces a factorization of the map $h$ as
\beq\label{Fact.DeformSpace}
\xymatrix{
\cX \times \A^1 \ar@{^{(}->}[r]^{\widetilde{h}} \ar[rd]_h & (M^{\circ} _{\widetilde{g}})' \ar[d]^{\overline{h}} \\
& (M^{\circ} _g )'}
\eeq
where $(M^\circ_{\tg})'$ is the open substack of $M^\circ_{\tg}$ consists of the fibers of $M^\circ _{\tg} \to \PP^1$ over $\A^1=\PP^1\setminus\{0\}$. The horizontal arrow $\th$ in \eqref{Fact.DeformSpace} is a closed embedding and the vertical arrow $\overline{h}$ in \eqref{Fact.DeformSpace} is a smooth morphism. By \cite[\S5]{Ful}, the deformation space $(M^\circ_{\tg})'$ is
\[(M^{\circ} _{\tg})' = \spec(\cdots\oplus I^2  T^{-2}\oplus{I}  T^{-1}\oplus \O_{\widetilde{\cZ}}  T^0\oplus \O_{\widetilde{\cZ}}  T^1\oplus \O_{\widetilde{\cZ}}  T^2\oplus\cdots)\]
over $\widetilde{\cZ} \times \A^1$, and the ideal sheaf of the closed subscheme $\cX \times \A^1$ in $(M^\circ_{\tg})'$ is
\[\I= \cdots\oplus{I}^2 T^{-2}\oplus{I} T^{-1}\oplus{J} T^0\oplus{J} T^1\oplus{J} T^2\oplus\cdots .\]

\begin{lemma}[cf. {\cite[Proposition 1]{KKP}}]\label{Lem.KKP}
There is a distinguished triangle
\beq\label{T.KKP}
\xymatrix{ \trunc f^* \bL_g \ar[r]^-{(T,a)} & \trunc f^* \bL_g \oplus \bL_{g\circ f} \ar[r] & \bL_h '  \ar[r] &}
\eeq
for some $\bL_h '$ such that $\trunc \bL_h ' \cong \bL_h$.

Moreover, the induced maps between the truncated cotangent complexes can be represented by a canonical right exact sequence
\beq\label{Res.KKP}
\xymatrix{
& \trunc f^*\bL_g \ar[r]  \ar@{}[d]|{\rotatebox[origin=c]{90}{=}} & \trunc f^*\bL_g \oplus \bL_{g\circ f} \ar[r] \ar@{}[d]|{\rotatebox[origin=c]{90}{=}}& \bL_h  \ar@{}[d]|{\rotatebox[origin=c]{90}{=}} & \\
& \dfrac{I}{IJ}[T] \ar[r]^-{(T,1)} \ar[d] & \left(\dfrac{I}{IJ}\oplus \dfrac{J}{J^2}\right)[T] \ar[r]^-{(T^{-1},-1)} \ar[d] & \dfrac{\I}{\I^2} \ar[r] \ar[d] & 0\\
0 \ar[r] & \left(\Omega_{\overline{g}}|_{\cX}\right)[T] \ar[r] & \left(\Omega_{\overline{g}}|_{\cX} \oplus \Omega_{\overline{g}\circ \overline{f}}|_{\cX}\right)[T] \ar[r] & \Omega_{\overline{h}}|_{\cX\times \A^1} \ar[r] & 0
}\eeq
of chain complexes, where
$\frac{\I}{\I^2} = \frac{I}{IJ}T^{-1} \oplus \frac{J}{J^2} T^0 \oplus \frac{J}{J^2} T^1 \oplus \cdots.$
\end{lemma}

\begin{proof}
The statements follow from the proof of \cite[Proposition 1]{KKP}. Indeed, the upper right exact sequence in \eqref{Res.KKP} is given in Case 1 in the proof of \cite[Proposition 1]{KKP}. On the other hand, in Case 3 of the proof of \cite[Proposition 1]{KKP}, it is explained that the lower exact sequence in \eqref{Res.KKP} can be deduced by the conormal sequences and the result in Case 1. Finally, if we form the distinguished triangle \eqref{T.KKP}, then we have $\trunc \bL_h ' \cong \bL_h$ by the right exact sequence \eqref{Res.KKP}. 
\end{proof}


We will need the following technical lemma for proving the isotropic condition in Lemma \ref{Lem.Deform.Isotropic}.

\begin{lemma}\label{Lem.D1}
There exists a dotted arrow that fits into the diagram
\beq\label{D11}
\xymatrix{
\fC_h \ar@{.>}[r] \ar@{^{(}->}[d] & \fC_g|_{\cX\times \A^1} \ar@{^{(}->}[d]\\
\fN_h \ar[r] & \fN_g|_{\cX\times\A^1}
}\eeq
where the map $\fN_h \to \fN_g|_{\cX\times \A^1}$ of the intrinsic normal sheaves is induced by the map $\trunc f^*\bL_g \to \bL_h$ of the truncated cotangent complexes in \eqref{Res.KKP}.
\end{lemma}

\begin{proof}
Note that the normal cone $C_{\th}$ of $\cX\times \A^1$ in $(M_{\tg} ^\circ)'$ is
\[C_{\th}=\spec\left(\bigoplus_{n\geq0} \dfrac{\I^n}{\I^{n+1}}\right) \quad \text{where}\quad \dfrac{\I^n}{\I^{n+1}}= \bigoplus_{1\leq i \leq n}\dfrac{I^i J^{n-i}}{I^i J^{n-i+1}}\cdot T^{-i}
 \oplus \bigoplus_{j\geq 0}\dfrac{{J}^n}{{J}^{n+1}}\cdot T^j .\]
Hence the map of graded $\O_{\cX}[T]$-algebras
\beq\label{i2}
\bigoplus_{n \geq 0}T^{-n}\cdot : \bigoplus_{n\geq 0}\left(\dfrac{I^n}{I^nJ}[T]\right) \to \bigoplus_{n\geq 0}\left( \bigoplus_{1\leq i \leq n}\dfrac{I^i J^{n-i}}{I^i J^{n-i+1}}\cdot T^{-i}
\oplus \bigoplus_{j\geq 0}\dfrac{{J}^n}{{J}^{n+1}}\cdot T^j \right)
\eeq
gives us a commutative diagram
\beq\label{D42}
\xymatrix{
C_{\th} \ar@{.>}[r] \ar@{^{(}->}[d] & C_{\tg}|_{\cX\times \A^1} \ar@{^{(}->}[d]\\
N_{\th} \ar[r] & N_{\tg}|_{\cX\times\A^1}
}\eeq
where 
the (-1)-term of the map \eqref{i2}
is equal to the map $({I}/{IJ})[T] \to {\I}/{\I^2}$ in \eqref{Res.KKP}. The commutative diagram \eqref{D11} follows from \eqref{D42}.
\end{proof}

\subsubsection{Construction of $\phi_h$}

We now construct a symmetric obstruction theory $\phi_h$ for $h$ using the reduction operation in Proposition \ref{Prop.Reduction}.


By Proposition \ref{Prop.Reduction}, we can form a reduction
\[\EE_h := \left[ \EE_f\dual[2] \xrightarrow{(\delta\dual,0,T)}\EE_{g\circ f} \oplus \left(\EE_f \oplus \EE_f\dual[2]\right) \xrightarrow{(\delta,T,0)} \EE_f\right]\]
where $T \in \Gamma(\A^1,\sO_{\A^1})$ is the coordinate function. Then the fibers of the symmetric complex $\EE_h$ over $\lambda \in \A^1$ can be expressed as follows:
\[(\EE_h)_\lambda =
\begin{cases}
\EE_{g\circ f} & \text{ if }\lambda\neq0 \\
f^*\EE_g \oplus \left( \EE_f \oplus \EE_f\dual[2]\right) & \text{ if }\lambda=0\,.
\end{cases}\]


\begin{lemma}\label{Lem.D3}
There exists a symmetric obstruction theory $\phi_h : \EE_h \to \bL_h$ such that the fibers are given as in \eqref{D.Eq27} and \eqref{D.Eq28}.
\end{lemma}

\begin{proof}
Form a morphism of distinguished triangles
\beq\label{Eq1_Lem.Deformation}
\xymatrix@C+1pc{
\EE_f\dual[2] \ar[r]^<<<<<<<{(T,\gamma\dual)} \ar@{=}[d] &\EE_f\dual[2]\oplus \DD\dual[2]  \ar[r]^>>>>>>>>{s_1} \ar[d]^{1\oplus\alpha\dual} & \BB_1 \ar[r] \ar@{.>}[d]^u & \EE_f\dual[3] \ar@{=}[d]\\
\EE_f\dual[2] \ar[r]^<<<<<<<{(T,\delta\dual)} & \EE_f\dual[2]\oplus\EE_{g \circ f}  \ar[r]^>>>>>>>>{s_2} & \BB_2 \ar[r] & \EE_f\dual[3]
}\eeq
for some perfect complexes $\BB_1$, $\BB_2$ and maps $s_1$, $s_2$, $u$. Applying the octahedral axiom to \eqref{Eq1_Lem.Deformation}, we obtain a distinguished triangle
\beq\label{T.B1B2Ef}
\xymatrix{\BB_1 \ar[r]^u & \BB_2 \ar[r]^v & \EE_f \ar[r]^w & \BB_1[1]}\eeq
such that $v \circ s_2 = (0,\delta)$.

We claim that there is a commutative diagram
\beq\label{Eq2_Lem.Deformation}
\xymatrix{
\EE_f\dual[2]\oplus \DD\dual[2]\ar[r]^<<<<<{1\oplus\alpha\dual} \ar[d]^{s_1} \ar@/_0.6cm/[dd]_{(0,\phi_g\circ\beta\dual)} & \EE_f\dual[2]\oplus\EE_{g \circ f}  \ar[d]_{s_2} \ar@/^0.6cm/[dd]^{(0,\phi_{g \circ f})} \\
\BB_1 \ar[r]^u \ar@{.>}[d]^{\psi_1} & \BB_2 \ar@{.>}[d]_{\psi_2} \\
\tau^{\geq-1} f^*\bL_g \ar[r]^a & \bL_{g \circ f}
}\eeq
induced by \eqref{Eq1_Lem.Deformation}. Indeed, the distinguished triangles in \eqref{Eq1_Lem.Deformation} show that there exist the dotted arrows  $\psi_1$ and $\psi_2$ which factor the curved arrows $(0,\phi_g\circ  \beta\dual)$ and $(0,\phi_{g \circ f})$. Then the commutativity of the total square implies the commutativity of the bottom square since $\Hom(\EE_f\dual[3], \bL_{g\circ f})=0$.

Form a morphism of distinguished triangles
\beq\label{Eq3_Lem.Deformation}
\xymatrix{
\BB_1 \ar[r]^-{(T,u)} \ar[d]^{\psi_1} & \BB_1 \oplus \BB_2 \ar[r] \ar[d]^{\psi_1\oplus\psi_2} & \EE_h  \ar[r] \ar@{.>}[d]^{\phi_h'}& \\
\tau^{\geq-1} f^*\bL_{g} \ar[r]^-{(T,a)} & \tau^{\geq-1} f^*\bL_{g} \oplus  \bL_{g\circ f} \ar[r] & \bL_h' \ar[r] & }\eeq
for some map $\phi_h '$, where the upper distinguished triangle is given by the distinguished triangle \eqref{T.B1B2Ef} (see Lemma \ref{Lem.D2} below) and the lower distinguished triangle is \eqref{T.KKP}. Let
\[\phi_h : \EE_h \xrightarrow{\phi_h'}\bL_h ' \to \trunc \bL_h'\cong\bL_h\]
be the composition. The diagram \eqref{Eq2_Lem.Deformation} shows that $h^0(\psi_1)$, $h^0(\psi_2)$ are bijective and $h^{-1}(\psi_1)$, $h^{-1}(\psi_2)$ are surjective.
The long exact sequence associated to \eqref{Eq3_Lem.Deformation} shows that $\phi_h$ is an obstruction theory.

Clearly, the restriction of the diagram \eqref{Eq3_Lem.Deformation} over $\lambda\neq0 \in \A^1$ gives us the obstruction theory for $g\circ f$,
\[(\phi_h)_\lambda = \phi_{g \circ f} : \EE_{g \circ f} \to  \bL_{g \circ f}.\]
On the other hand, the restriction of the diagram \eqref{Eq3_Lem.Deformation} over $\lambda=0 \in \A^1$ gives us a morphism
\[\xymatrix@C-0.3pc{
\ar[r] & (\EE_f\dual[2] \oplus f^*\EE_{g}) \oplus (\EE_f\dual[2] \oplus \DD) \ar[r]^-{1\oplus(0,\gamma)} \ar[d]^{(0,\phi_g)\oplus(0,\phi_{g\circ f}')} & (\EE_f\dual[2] \oplus f^*\EE_g) \oplus \EE_f \ar[d]^{(\phi_h')_0}\ar[r] & \\
\ar[r] & \tau^{\geq-1} f^*\bL_{g} \oplus \bL_{g\circ f} \ar[r]^-{1\oplus b} & \tau^{\geq-1} f^*\bL_g \oplus \bL'_f  \ar[r] & }\]
of distinguished triangles. Therefore we have \eqref{D.Eq28} for some maps $\eta:\EE_f \to \bL'_f$ and $\xi:\EE_f \to \tau^{\geq-1} f^*\bL_g$ such that $\eta$ fits into \eqref{D1}.
\end{proof}

We need the following lemma to complete the proof of Lemma \ref{Lem.D3}.

\begin{lemma}\label{Lem.D2}
There is a distinguished triangle
\beq\label{D21}
\xymatrix{
\BB_1 \ar[r]^-{(T,u)}& \BB_1 \oplus \BB_2 \ar[r] & \EE_h  \ar[r] & \BB_1[1] \\
}\eeq
associated to the short exact sequence \eqref{SES.B1B1B2Eh} below.
\end{lemma}

\begin{proof}
We can simply deduce an isomorphism
\[\cone(\BB_1 \xrightarrow{(T,u)} \BB_1 \oplus \BB_2)\cong \cone(\BB_2 \oplus \EE_f \xrightarrow{(v,T)} \EE_f)[-1] = \EE_h\]
from the distinguished triangle \eqref{T.B1B2Ef} using the octahedral axiom twice. 
However, we will construct an explicit short exact sequence using resolutions of complexes to prove the isotropic condition of $\phi_h$ in Lemma \ref{Lem.Deform.Isotropic} below. We will find a short exact sequence representing \eqref{T.B1B2Ef}, then we will have an isomorphism 
\[\coker(\BB_1 \xrightarrow{(T,u)} \BB_1 \oplus \BB_2)\cong \ker(\BB_2 \oplus \EE_f \xrightarrow{(v,T)} \EE_f) = \EE_h\]
of chain complexes.

Choose a symmetric resolution and a resolution
\beq\label{Res.Egf,Ef}
 [B\to E\to B\dual] \cong \EE_{g\circ f} \and [K\dual \to D\dual] \cong \EE_f,
\eeq
respectively, such that $\delta : \EE_{g\circ f} \to \EE_f$ can be represented by a surjective chain map. The resolutions \eqref{Res.Egf,Ef} induce a symmetric resolution and a resolution
\beq\label{Res.Eg,D} [\tfrac{B}{D} \to \tfrac{K\uperp}{K} \to \left(\tfrac{B}{D}\right)\dual] \cong f^*\EE_g \and [\tfrac{B}{D} \to \tfrac{E}{K} \to B\dual] \cong \DD \eeq
of $f^*\EE_g$ and $\DD$, respectively. Consider the two subbundles
\[D_1=(1,T)\cdot D \subseteq B \oplus D \and K_1 = (1,T)\cdot K \subseteq E \oplus K\]
on $\cX \times \A^1$, given by the images of the embeddings $(1,T) : D \to B \oplus D$ and $(1,T) : K \to E \oplus K$. Let
\[K_1\uperp \subseteq E \oplus (K \oplus K\dual) \] 
be the orthogonal complement of $K_1=K_1\oplus 0$ in the special orthogonal bundle $E \oplus (K\oplus K\dual)$. The two perfect complexes $\BB_1$ and $\BB_2$ in \eqref{Eq1_Lem.Deformation} have resolutions
\beq\label{Res.B1B2}
\left[\tfrac{B\oplus D}{D_1} \to \tfrac{K\uperp \oplus K}{K_1} \to \left(\tfrac BD\right)\dual\right] \cong \BB_1 \and
\left[\tfrac{B\oplus D}{D_1} \to \tfrac{E \oplus K}{K_1} \to B\dual\right] \cong \BB_2,\eeq
respectively, and the symmetric complex $\EE_h$ has a symmetric resolution
\beq\label{Res.Eh}
\left[\tfrac{B\oplus D}{D_1} \to \tfrac{K_1\uperp}{K_1} \to \left(\tfrac{B\oplus D}{D_1}\right)\dual\right] \cong \EE_h. \eeq
The canonical short exact sequence of chain complexes
\beq\label{SES.B1B1B2Eh}
\xymatrix{
& \BB_1 \ar[r]^-{(T,u)}  \ar@{}[d]|-{\rotatebox[origin=c]{90}{=}} & \BB_1\oplus \BB_2\ar[r] \ar@{}[d]|-{\rotatebox[origin=c]{90}{=}}& \EE_h \ar@{}[d]|-{\rotatebox[origin=c]{90}{=}} &  \\
0 \ar[r] & \tfrac{B\oplus D}{D_1} \ar[r]^-{(T,1)}\ar[d] & \tfrac{B\oplus D}{D_1}\oplus \tfrac{B\oplus D}{D_1} \ar[r]^-{(1,-T)} \ar[d] & \tfrac{B\oplus D}{D_1} \ar[d] \ar[r] & 0   \\
0 \ar[r] & \tfrac{K\uperp \oplus K}{K_1} \ar[r]^-{(T,1)} \ar[d] & \tfrac{K\uperp \oplus K}{K_1} \oplus \tfrac{E \oplus K}{K_1} \ar[r]^-{\kappa} \ar[d] & \tfrac{K_1\uperp}{K_1} \ar[d] \ar[r] & 0\\
0 \ar[r] & \left(\tfrac BD\right)\dual \ar[r]^-{(T,1)} & \left(\tfrac BD\right)\dual \oplus B\dual \ar[r]^-{}  & \left(\tfrac{B\oplus D}{D_1}\right)\dual \ar[r] & 0
}\eeq
gives us the desired distinguished triangle \eqref{D21}, where the map $\kappa$ is
\begin{align*}
\left(\tfrac{K\uperp \oplus K}{K_1}\right) \oplus \left(\tfrac{E \oplus K}{K_1}\right) & \mapright{\kappa} \tfrac{K_1\uperp}{K_1} \subseteq \tfrac{E \oplus K \oplus K\dual}{K_1} \\ (\overline{(x,y)},\overline{(z,w)})& \mapsto \overline{(x-Tz,y-Tw,\overline{z})}
\end{align*}
for $x \in K\uperp$, $y \in K$, $z \in E$, and $w \in K$.
\end{proof}

\subsubsection{Isotropic condition for $\phi_h$}

We finally prove the isotropic condition of $\phi_h$ in Lemma \ref{Lem.D3} using the criterion in Proposition \ref{Prop.Isotropic}.

\begin{lemma}\label{Lem.Deform.Isotropic}
$\phi_h$ satisfies the isotropic condition.
\end{lemma}

\begin{proof}
We use the notations as above. More precisely, we will use the factorizations \eqref{Fact.XYZ}, \eqref{Fact.DeformSpace}, the resolutions of cotangent complexes in \eqref{Res.CotanCplx.XYZ}, \eqref{Res.KKP}, and the resolutions of perfect complexes and symmetric complexes in \eqref{Res.Egf,Ef}, \eqref{Res.Eg,D}, \eqref{Res.B1B2}, \eqref{Res.Eh}, and the short exact sequence \eqref{SES.B1B1B2Eh}. Since the statement is local, we may assume that all schemes in \eqref{Fact.XYZ} are affine.

Since $\cX$ is affine, there is a chain map
\beq\label{NI1}
\xymatrix{
f^*\EE_g  \ar[d]_{f^*\phi_g} \ar@{}[rd]|{=} & {B}/{D} \ar[r] \ar[d] & {K\uperp}/{K} \ar[r] \ar[d]^{\Phi_g} & \left({B}/{D}\right)\dual \ar[d]\\
\trunc f^* \bL_g    & 0 \ar[r] & {I}/{IJ} \ar[r] & \Omega_{\overline{g}}|_{\cX}
}\eeq
representing $f^*\phi_g$. By Proposition \ref{Prop.Isotropic}, we have a commutative diagram
\[\xymatrix{
C_{\tg}|_{\cX} \ar@{^{(}->}[r] \ar[d] & N_{\tg}|_{\cX} \ar[r] \ar@/^0.5cm/[rr]^-{C(\Phi_g)} \ar[d] & C(Q) \ar[r] \ar[d]& K\uperp/K \ar[d]^{\fq_{K\uperp/K}}\\
\fC_g|_{\cX} \ar@{^{(}->}[r] & \fN_g|_{\cX} \ar@{^{(}->}[r]^-{\fC(\phi_g)} & \fC(f^*\EE_g) \ar[r]^-{f^*\fq(\EE_g)} & \A^1 _{\cX}
}\]
where $Q=\coker(B/D \to K\uperp/K)$. Hence the composition
\beq\label{NI2}C_{\tg}|_{\cX} \hookrightarrow N_{\tg}|_{\cX} \xrightarrow{C(\Phi_g)} K\uperp/K \xrightarrow{\fq_{K\uperp/K}} \A^1\eeq
vanishes by the isotropic condition of $\phi_g$.

The chain map \eqref{NI1} induces a chain map representing $\psi_1$ 
\beq\label{NI3}
\xymatrix{
\BB_1  \ar[d]_{\psi_1} \ar@{}[rd]|{=} & \dfrac{B\oplus D}{D_1} \ar[r] \ar[d] & \dfrac{K\uperp\oplus K}{K_1} \ar[r] \ar[d]^{\Psi_1} & \left(\dfrac{B}{D}\right)\dual \ar[d]\\
\trunc f^* \bL_g  & 0 \ar[r] & {I}/{IJ} \ar[r] & \Omega_{\overline{g}}|_{\cX}
}\eeq
since $\Hom(\EE_f\dual[3],\trunc f^*\bL_g)=0$. Then the (-1)-term $\Psi_1$ is the composition
\[\Psi_1 : {(K\uperp\oplus K)}/{K_1} \xrightarrow{(1,0)} {K\uperp}/{K} \xrightarrow{\Phi_g} {I}/{IJ}.\]

Choose any chain map
\beq\label{NI6}
\xymatrix{
\EE_h  \ar[d]^{\phi_h} & \dfrac{B\oplus D}{D_1} \ar[r] \ar[d] & \dfrac{K_1\uperp}{K_1} \ar[r] \ar[d]^-{\Phi_h} & \left(\dfrac{B\oplus D}{D_1}\right)\dual \ar[d]\\
\bL_h  & 0 \ar[r] & {\I}/{\I^2} \ar[r] & \Omega_{\overline{h}}|_{\cX}
}\eeq
representing $\phi_h$. By the criterion of isotropic condition in Proposition \ref{Prop.Isotropic}, it suffices to show that the composition
\[C_{\th} \hookrightarrow  N_{\th} \xrightarrow{C(\Phi_h)} K_1\uperp/K_1 \xrightarrow{\fq_{K_1\uperp/K_1}} \A^1\]
vanishes. 

Consider the commutative diagram 
\beq\label{NI5} 
\xymatrix{
\BB_1 \ar[r] \ar[d]^{\psi_1} & \EE_h \ar[d]^{\phi_h}\\
\trunc f^* \bL_g  \ar[r] & \bL_h 
}\eeq
in the derived category of $\cX\times \A^1$, induced by \eqref{Eq3_Lem.Deformation}. The chain maps in \eqref{SES.B1B1B2Eh}, \eqref{Res.KKP}, \eqref{NI3}, and \eqref{NI6} represent the maps in the diagram \eqref{NI5}. Since $\cX\times \A^1$ is affine, these chain maps representing the diagram \eqref{NI5} commute up to a homotopy. This homotopy should be given by a map $x: (B/D)\dual \to \I/\I^2$. Considering the (-1)-terms, there is a diagram
\beq\label{NI9}
\xymatrix{
\dfrac{K\uperp\oplus K}{K_1} \ar[rd]|-{1\oplus 1\oplus 0} \ar[dd]^{\Psi_1} \ar[rr]& & \left(\dfrac{B}{D}\right)\dual \ar[dd]^x \\
& \dfrac{K_1\uperp}{K_1}\ar[rd]|{\Phi_h} &\\
\left({{I}}/{{I}{J}}\right)[T] \ar[rr]^-{T^{-1}} && {\I}/{\I^2}
}\eeq
where the diagonal composition in \eqref{NI9} is the sum of the other two compositions in \eqref{NI9}. 

Let $K_2$ be the isotropic subbundle of $K_1\uperp/K_1$ given by the embedding
\[(0,1,0) : K \hookrightarrow \dfrac{K_1\uperp}{K_1} \subseteq \dfrac{E \oplus K \oplus K\dual}{K_1}.\]
Then $(K\uperp\oplus K)/K_1 = K_2\uperp $ as subbundles of $K_1\uperp/K_1$ and $K_2\uperp/K_2 \cong K\uperp/K$ as special orthogonal bundles. The composition
\[K_2 \hookrightarrow K_1\uperp/K_1 \mapright{\Phi_h} \I/\I^2\]
vanishes since the two maps 
\[K_2 \to K_2\uperp \mapright{\Psi_1} (I/IJ)[T] \and K_2 \to K_2\uperp \to (B/D)\dual\]
vanish. Let
\[\Phi_h ' : K_2\uperp/K_2 \to (K_1\uperp/K_1)/K_2 \to \I/\I^2.\]
be the induced map. From the commutative diagram
\[\xymatrix{
N_{\th} \ar[r] \ar@/^0.5cm/[rr]^{C(\Phi_h)} \ar[rd]_{C(\Phi_h')} & K_2\uperp \ar[r] \ar[d] & K_1\uperp/K_1 \ar[d]^{\fq_{K_1\uperp/K_1}} &  \\
& K\uperp/K \ar[r]^-{\fq_{K\uperp/K}} & \A^1
}\]
it suffices to show that the composition
\[C_{\th} \hookrightarrow  N_{\th} \xrightarrow{C(\Phi_h ')} K\uperp/K \xrightarrow{\fq_{K\uperp/K}} \A^1\]
vanishes. 

The diagram \eqref{NI9} induces a diagram
\[\xymatrix{
{K\uperp}/{K} \ar[rd]|{\Phi_h '} \ar[d]_{\Phi_g} \ar[r]^{} & \left({B}/{D}\right)\dual \ar[d]^x \\
\left({{I}}/{{I}{J}}\right)[T] \ar[r]^-{T^{-1}} & {\I}/{\I^2}
}\]
where the diagonal arrow is the sum of the two other compositions. Since the two compositions
\[B/D \to K\uperp/K \mapright{\Phi_g} I/IJ \and B/D \to K\uperp/K \to (B/D)\dual\]
vanishes by \eqref{NI1}, it suffices to show that the composition
\[C_{\th} \hookrightarrow N_{\th} \to N_{\tg}|_{\cX\times \A^1} \xrightarrow{C(\phi_g)}  K\uperp/K \xrightarrow{\fq_{K\uperp/K}}  \A^1\]
vanishes. The commutative square \eqref{D42} in Lemma \ref{Lem.D1} and the vanishing of the composition \eqref{NI2} completes the proof.
\end{proof}

\section{Lefschetz principle}\label{S.Lefschetz}

In this section, we prove the {\em Lefschetz principle} in Donaldson-Thomas theory.

\subsection{Main result}\label{ss.Lefschetz.1}

Let $X$ be a smooth projective Calabi-Yau 4-fold. Fix a curve class $\beta \in H_2(X,\Q)$ and an integer $n \in \Z$. Let
\begin{align*}
I_{n,\beta}(X) &= \{\text{closed subschemes }Z \text{ of } X \text{ with } \ch(\O_Z) = (0,0,0,\beta,n)\},\\
P_{n,\beta}(X) &= \{\text{stable pairs }(F,s) \text{ on } X \text{ with } \ch(F) = (0,0,0,\beta,n)\}
\end{align*}
be the Hilbert scheme of curves and the moduli space of stable pairs. In particular, when $\beta=0$, we let $I_{n,0}(X)=X\un$ be the Hilbert scheme of points and $P_{n,0}(X)=\emptyset$.


Let $P(X)$ denote one of the above two moduli spaces. Then there is a universal family 
\[\II = [\O_{P(X)\times X} \mapright{} \FF]\]
of pairs on $X$. When $P(X)=I_{n,\beta}(X)$ is the Hilbert scheme, $\FF$ is the structure sheaf of the universal family. The perfect complex $\II$ defines an open embedding 
\[P(X) \hookrightarrow \Perf(X)^\spl_{\O_X}\]
to the moduli space of simple perfect complexes on $X$ with fixed trivial determinant \cite{Inaba,Lieblich,ToVa,STV}. Hence $P(X)$ carries a (-2)-shifted symplective derived enhancement \cite{PTVV} and an orientation \cite{CGJ}, which induces an Oh-Thomas virtual cycle 
\[[P(X)]\virt_{OT} \in A_n(P(X))\]
by \cite{OT}. Let $L$ be an line bundle on $X$. We will study the {\em tautological invariants}
\[\int_{[P(X)]\virt} c_n(\bR\pi_*(\FF\otimes L))\]
in terms of the divisors of $L$. Here $\pi : P(X) \times X \to P(X)$ denotes the projection map.

Let $D$ be a smooth divisor of $X$ such that $\O_X(D)=L$. Define the moduli space $P(D)$ on $D$ as
\[P(D) = \begin{cases}
\bigsqcup_{\beta'} I_{n,\beta'}(D) & \text{ if } P(X) = I_{n,\beta}(X)\\
\bigsqcup_{\beta'} P_{n,\beta'}(D) & \text{ if } P(X) = P_{n,\beta}(X)
\end{cases}\]
where the disjoint union is taken over all $\beta' \in H_2(D,\Q)$ such that $i_*\beta'=\beta$, and $i:D\hookrightarrow X$ denotes the inclusion map. Then $P(D)$ carries a Behrend-Fantechi virtual cycle
\[[P(D)]\virt_{BF} \in A_{-\beta\cdot D} (P(D))\]
associated to the standard perfect obstruction theory.

\begin{theorem}[Lefschetz principle]\label{Thm.Lefschetz}
Let $X$ be a Calabi-Yau 4-fold and $D$ be a smooth divisor with $\O_X(D)=L$. Fix a curve class $\beta \in H_2(X,\Q)$ and an integer $n$. Let $P(X)$ be one of the two moduli spaces $I_{n,\beta}(X)$ and $P_{n,\beta}(X)$. Assume that the tautological complex $\bR\pi_*(\FF\otimes L)$ is a vector bundle concentrated in degree $0$. Then for any orientation on $P(X)$, there exist canonical signs $\sigma(e)$ for connected components $P(D)^e$ of $P(D)$ such that
\beq\label{Eq.Lefschetz} 
\sum_{e}(-1)^{\sigma(e)}(j_{e})_*[P(D)^e]\virt_{BF} =  e\left(\bR\pi_*(\FF\otimes L)\right) \cap [P(X)]\virt_{OT},
\eeq
where $j_e : P(D)^e \hookrightarrow P(X)$ denotes the inclusion map.
\end{theorem}


In principle, the signs $\sigma(e)$ are uniquely determined by natural triangles, but the author does not know how to compute them. If we have an affirmative answer to Question \ref{Question.Signs} below, then we can remove the signs $\sigma(e)$ in \eqref{Eq.Lefschetz} and have the following simpler formula
\[j_*[P(D)]\virt_{BF} = e\left( \bR\pi_*(\FF\otimes L) \right) \cap [P(X)]\virt_{OT}\]
where $j: P(D) \hookrightarrow P(X)$ denotes the inclusion map.

\begin{question}\label{Question.Signs}
Given $X,D,P(X)$ as in Theorem \ref{Thm.Lefschetz}, is there a choice of orientations on connected components of $P(X)$ such that the signs $\sigma(e)$ all coincide?
\end{question}


Before we prove our main theorem (Theorem \ref{Thm.Lefschetz}) in this section, we present two immediate corollaries. 

\begin{corollary}[Tautological Hilbert scheme invariants]
Let $X$ be a Calabi-Yau 4-fold and $L$ be a line bundle. If there is a smooth connected divisor $D$ such that $\O_X(D)=L$, then there exists a choice of orientations such that
\[\sum_{n\geq 0}\int_{[X\un]\virt}e(L\un)\cdot q^n =M(-q)^{\int_X{c_3(X)c_1(L)}}\]
where $M(q)=\prod_{n \geq 1}(1-q^n)^{-n}$ denotes the MacMahon function.
\end{corollary}

\begin{proof}
Since $D\un$ is connected, the Lefschetz principle gives us
\[\int_{[X\un]\virt}e(L\un) =\int_{[D\un]\virt}1.\]
By \cite{LePa,Li}, the generating series of the degree zero MNOP invariants \cite{MNOP1,MNOP2} of a smooth projective 3-fold $D$ can be expressed as
\[ \sum_{n\geq 0}\int_{[D\un]\virt}1 \cdot q^n  =M(-q)^{\int_D c_3(\TT_D\otimes K_D)}.\]
By an elementary argument, we can deduce
\[\int_D c_3(\TT_D\otimes K_D) = \int_X c_3(\TT_X)c_1(L)\]
(cf. \cite[(2.5)]{CKp}). It completes the proof.
\end{proof}


\begin{corollary}[Tautological DT/PT correspondence]\label{Cor.DTPTcurves}
Let $X$ be a Calabi-Yau 4-fold and $L$ be a line bundle. Assume that there is a smooth connected divisor $D$ of $X$ which is a Calabi-Yau 3-fold and $\O_X(D)=L$. Let  $\beta \in H_2(X,\Q)$ be a curve class satisfying the following two conditions for both $I_{n,\beta}(X)$ and $P_{n,\beta}(X)$ and for all $n$:
\begin{enumerate}
\item [A1)] The tautological complex $\bR\pi_*(\FF\otimes L)$ is a vector bundle.
\item [A2)] The inclusion map $j :P(D) \hookrightarrow P(X)$ induces an injective function between the sets of connected components of $P(D)$ and $P(X)$.
\end{enumerate}
Then there exists a choice of orientations such that
\[\dfrac{\sum_{n\geq 0} \int_{[I_{n,\beta}(X)]\virt} c_n(\bR\pi_*(\FF\otimes L)) \cdot  q^n}{\sum_{n\geq 0} \int_{[X\un]\virt} c_n(L\un) \cdot  q^n}  = \sum_{n\geq 0} \int_{[P_{n,\beta}(X)]\virt} c_n(\bR\pi_*(\FF\otimes L)) \cdot  q^n\,.\]
\end{corollary}

\begin{proof}
Applying the Lefschetz principle to the three moduli space $I_{n,\beta}(X)$, $P_{n,\beta}(X)$, and $I_{n,0}(X)=X\un$, the 3-fold DT/PT correspondence \cite{Bridgeland,Toda}
\[\dfrac{\sum_{n\geq 0} \int_{[I_{n,\beta}(D)]\virt}1 \cdot  q^n} {\sum_{n\geq 0}\int_{[D\un]\virt}1\cdot q^n}
= \sum_{n\geq 0}\int_{[P_{n,\beta}(D)]\virt} 1 \cdot q^n\]
completes the proof.
\end{proof}

We present an example that satisfies the assumptions in Corollary \ref{Cor.DTPTcurves}.

\begin{example}\label{Ex.DT/PTrational}
Let $X=Y\times E$ be the product of a Calabi-Yau 3-fold $Y$ and an elliptic curve $E$. Let $L=\O_X(Y\times \{pt\})$. Let $\beta \in H_2(Y)\subseteq H_2(X)$ be a curve class. Assume that for any pure 1-dimensional closed subscheme $C$ of $Y$ with $[C]=\beta$, there is a collection of rigid smooth rational curves $C_1,C_2,\cdots,C_r$ on $Y$ such that
\[C=\bigcup_{1\leq i\leq r} C_i \and \#\left(C_i \cap (\bigcup_{j<i}C_j)\right)\leq 1 \text{ for all }i.\]
Then the assumptions in Corollary \ref{Cor.DTPTcurves} are satisfied and the tautological DT/PT correspondence holds.
\end{example}

The rest of this section is devoted to the proof of Theorem \ref{Thm.Lefschetz}. We briefly sketch the structure of the proof.
\begin{enumerate}
\item In \S\ref{ss.ComparisonofModuli}, we prove that $P(D)$ is the zero locus of the tautological section $\tau$ of the tautological bundle $\bR\pi_*(\FF\otimes L)$ in $P(X)$. 
\item In \S\ref{ss.ConstructionofSOT}, we construct a natural {\em 3-term} symmetric obstruction theory on $P(D)$ which gives us the same virtual cycle $[P(D)]\virt_{OT}=[P(D)]\virt_{BF}$ (up to signs). 
\item In \S\ref{ss.ComparisonofOT}, we show the compatibility of the 3-term symmetric obstruction theories for $P(X)$ and $P(D)$. \item In \S\ref{ss.ComparisonofVirtualCycles}, we apply the virtual pullback formula in \S\ref{S.Functoriality} and finish the proof of Theorem \ref{Thm.Lefschetz}.
\end{enumerate}

\subsection{Comparison of moduli spaces}\label{ss.ComparisonofModuli}


Under the notations in Theorem \ref{Thm.Lefschetz}, define the {\em tautological section} as the composition
\[\tau : \O_{P(X)} \mapright{s} \bR\pi_*(\FF)  \mapright{f_D} \bR\pi_*(\FF\otimes L)\]
where $s:\O_{P(X)\times X} \to \FF$ is the universal section and $f_D \in \Gamma(X,L)$ is the defining equation of the divisor $D$. Consider a diagram
\[\xymatrix{
& \bR\pi_* (\FF\otimes L) \ar[d] \\
P(D) \ar@{^{(}->}[r]^{j} & P(X) \ar@/_.4cm/[u]_{\tau} .
}\]

\begin{proposition}\label{Prop.ComparisonofModuli}
P(D) is the zero locus of $\tau$ in $P(X)$.
\end{proposition}

Proposition \ref{Prop.ComparisonofModuli} is a geometric version of the Lefschetz principle in Theorem \ref{Thm.Lefschetz}. If we prove Proposition \ref{Prop.ComparisonofModuli}, then we will have a Gysin pullback
\beq\label{L.GysinPullback}
j^! : A_* (P(X)) \to A_* (P(D))
\eeq
such that  $j_*\circ j^! = e(\bR\pi_*(\FF\otimes L)).$


\begin{proof}[Proof of Proposition \ref{Prop.ComparisonofModuli}]
Let $P(X)=I_{n,\beta}(X)$ be the Hilbert scheme. Then it is easy to show that $P(D)$ is the zero locus of $\tau$ (cf. \cite[Proposition 2.4]{CKp}). Indeed, consider a morphism $T\to P(X)$ from a scheme $T$ that corresponds to a closed subscheme $Z \subseteq T\times X$. Then the pullback $\tau|_T$ of the tautological section corresponds to the composition
\beq\label{L43}\O_T \boxtimes L\dual \mapright{f_D} \O_{T\times X} \to \O_Z\eeq
under the adjunction $\pi^* \dashv \pi_*$. Since \eqref{L43} vanishes if and only if $Z$ is contained in $T \times D$, the zero locus of $\tau$ is exactly $P(D)$.

Now assume that $P(X)=P_{n,\beta}(X)$ is the moduli space of stable pairs. Let $V$ be the zero locus of $\tau$ in $P(X)$. Obviously, we have $P(D) \subseteq V$ since the pullback $\tau|_{P(D)}$ is zero. Conversely, we will show that $V \subseteq P(D)$. Choose any morphism $T \to V$ from a scheme $T$. Then the composition $T \to V \to P(X)$ corresponds to a family $s:\O_{T\times X} \to F$ of stable pairs. Since the pullback $\tau|_T$ of the tautological section vanishes, the composition
\[ \O_T \boxtimes L\dual \mapright{f_D} \O_{T\times X} \mapright{s} F\]
also vanishes by the adjunction. To show that the map $T \to V$ factors through $P(D)$, we need to show that $F$ is scheme-theoretically supported in $T\times D$. Equivalently, it suffices to show that the map
\beq\label{L42}
f_D\otimes 1_{F} : L\dual \otimes F\to F\eeq
vanishes. Let $Q=\coker(s : \O_{T\times X} \to F)$ be the cokernel of the section $s$ and $G=L\dual \cdot F \subseteq F$ be the image of the map \eqref{L42}. Since the map \eqref{L42} vanishes away from the support of $Q$, the reduced support of $G$ is contained in the support of $Q$. Let $s_t:\O_X \to F_t$ be the fiber of $s:\O_{T\times X} \to F$ over $t \in T$. Then $Q_t=\coker(s_t : \O_X \to F_t)$ is a zero-dimensional sheaf, and $F_t$ is a pure 1-dimensional sheaf, by the stability of the pair $(F_t,s_t)$. Hence the fiber $G_t$ of $G$ over $t\in T$ is also a 0-dimensional sheaf, and thus $\Hom_{X}(G_t,F_t) = 0$. By Lemma \ref{Lem.Fiberwisezero} below, we have
\[\Hom_{T\times X}(G, F) = 0.\]
Hence $G=0$ and the map \eqref{L42} is also zero.
\end{proof}

We need the following lemma to complete the proof of Proposition \ref{Prop.ComparisonofModuli}.

\begin{lemma}\label{Lem.Fiberwisezero}
Let $\pi : \sX \to T$ be a smooth projective morphism of schemes. Let $\sX_t=\sX\times_T \{t\}$ be the fiber over $t \in T$. Consider coherent sheaves $\sF$ and $\sG$ on $\sX$. Assume that $\sF$ is flat over $T$. Then
\beq\label{L41}\Hom_{\sX_t}(\sG_t,\sF_t) =0  \text{ for all } t\in T \implies \Hom_{\sX}(\sG,\sF)=0, \eeq
where $\sG_t $ and $\sF_t$ are the pullbacks of the sheaves $\sG$ and $\sF$ to $\sX_t$.
\end{lemma}

\begin{proof}
Since the statement is local, we may assume that $T$ is a quasi-projective scheme. Let $i_t : \sX_t \hookrightarrow \sX$ be the inclusion map. We claim that
\[\Hom_{\sX_t} (\bL i_t^* \sG,\sF_t) = \Hom_{\sX_t}(\sG_t,\sF_t).\]
Indeed, we have a distinguished triangle
\[\xymatrix{ R_1 \ar[r] & \bL i_t^*\sG \ar[r] & \sG_t \ar[r] & R_1[1]}\]
on $\sX_t$ such that $R_1$ is concentrated in degree $\leq -1$. Hence
\[\Hom_{\sX_t}(R_1,\sF_t) = \Hom_{\sX_t}(R_1[1],\sF_t)=0,\]
which proves the claim.

Since $\sX$ is quasi-projective, we can find an exact sequence
\[0 \to R_2 \to E^{-1} \to E^0 \to \sG \to 0\]
for some vector bundles $E^0$ and $E^{-1}$ on $\sX$. Letting $\EE = [E^{-1} \to E^0]$, we can form a distinguished triangle
\[\xymatrix{R_2[1] \ar[r] & \EE \ar[r] & \sG \ar[r] & R_2[2] }\]
on $\sX$. As above, we have
\[\Hom_{\sX_t}(\bL i_t^*\sG,\sF_t) = \Hom_{\sX_t}(\EE_t,\sF_t) \and \Hom_{\sX}(\EE,\sF)=\Hom_{\sX}(\sG,\sF)\]
for all $t \in T$, where $\EE_t = \bL i_t ^* \EE$. Hence it suffices to show \eqref{L41} for $\EE$. 

Consider the perfect complex $\bR\hom_{\pi}(\EE,\sF)$ on $T$. By the base change theorem, we have
\[\bR\hom_{\pi}(\EE,\sF)|_{t \in T} = \bR\Hom_{\sX_t}(\EE_t,\sF_t).\]
Hence if $\Hom_{\sX_t}(\EE_t,\sF_t)=0$ for all $t \in T$, then  $\bR\hom_{\pi}(\EE,\sF)$ has tor-amplitude $\geq 1$ so that
\[\Hom_{\sX}(\EE,\sF) = H^0 (\bR \Gamma \circ \bR\hom_{\pi}(\EE,\sF) )= 0.\]
It completes the proof.
\end{proof}

\begin{remark}
There is an alternative proof of Proposition \ref{Prop.ComparisonofModuli} by comparing the {\em pair} obstruction theories\footnote{These obstruction theories are different with the obstruction theories used to define the virtual cycles in \S\ref{ss.Lefschetz.1}.} of $P(D)$ and $P(X)$ using derived algebraic geometry. This approach allows us to generalize Proposition \ref{Prop.ComparisonofModuli} and Theorem \ref{Thm.Lefschetz} to other moduli spaces of pairs (e.g., moduli space $P_{n,\beta}^{t}(X)$ of $Z_t$-stable pairs \cite{CTd}). We sketch the proof here.

Let $\bP(X)$ be the derived enhancement of $P(X)$ as a derived moduli space of pairs. Then the tangent complex at a $\C$-point $(F,s) \in \bP(X)$ can be expressed as
\[\TT_{\bP(X),(F,s)} = \bR\Hom_{X}(I,F)\]
where $I=[\O_X \xrightarrow{s} F]$. The tautological complex $\bR\pi_*(\FF\otimes L)$ and the tautological section $\tau$ on $P(X)$ extend to $\bP(X)$ naturally. Let $\bV$ be the derived zero locus of the extended tautological section.
Clearly, the classical truncation of $\bV$ is $V$. The canonical distinguished triangle of tangent complexes for the map $\bV \to \bP(X)$ at a $\C$-point $(F,s) \in \bV$ is 
\[\xymatrix{\bR\Hom_{X}(J,F) \ar[r] \ar@{=}[d]& \bR\Hom_{X}(I,F) \ar[r] \ar@{=}[d] & \bR\Hom_{X}(L\dual,F) \ar@{=}[d]\\ 
\TT_{\bV,(F,s)} \ar[r] & \TT_{\bP(X),(F,s)} \ar[r] & \TT_{\bV/\bP(X),(F,s)}[1] } \]
where $I=[\O_X \xrightarrow{s} F]$ and $J=[\O_D \xrightarrow{s} F]$. 

Let $\bP(D)$ be the derived enhancement of $P(D)$ as a derived moduli space of pairs. Then there exists a canonical map $\bP(D) \to \bV$. The distinguished triangle of tangent complexes for $\bP(D) \to \bV$ at a $\C$-point $(F,s) \in \bP(D)$ is
\[\xymatrix{
\bR\Hom_D(J,F) \ar[r] \ar@{=}[d] & \bR\Hom_X(J,F) \ar[r] \ar@{=}[d] & \bR\Hom_D(J,F\otimes L)[-1] \ar@{=}[d] \\
\TT_{\bP(D),(F,s)} \ar[r] & \TT_{\bV,(F,s)} \ar[r] & \TT_{\bP(D)/\bV,(F,s)}[1]
}\]
where $J=[\O_D \xrightarrow{s} F]$. Hence the relative cotangent complex $\LL_{\bP(D)/\bV}$ is concentrated in degrees $\leq -2$. Therefore the induced map $P(D) \to V$ between the classical truncations is an \'etale morphism. Using the stability conditions, it is easy to show that $P(D) \to V$ is bijective. Hence $P(D) \to V$ is an isomorphism.
\end{remark}

\begin{remark}
Proposition \ref{Prop.ComparisonofModuli} holds even when the tautological complex $\bR\pi_*(\FF\otimes L)$ is not a vector bundle. In this case, we have a cartesian diagram
\beq\label{L21}
\xymatrix{
P(D) \ar[r] \ar[d] & P(X) \ar[d]^{\tau}\\
P(X) \ar[r]^>>>>>0 & C(h^0(\bR\pi_*(\FF\otimes L)\dual))
}\eeq
where $C(h^0(\bR\pi_*(\FF\otimes L)\dual)$ is the abelian cone associated to the coherent sheaf $h^0(\bR\pi_*(\FF\otimes L)\dual)$.

However, the derived enhancement of the cartesian diagram \eqref{L21} 
\[\xymatrix{ \bP(D) \ar[r] \ar[d] & \bP(X) \ar[d] \\ \bP(X) \ar[r]^-{0} & \bR\pi_*(\bF\otimes L) }\] 
is {\em not} homotopy-cartesian, where $\bF$ denotes the extended universal family on $\bP(X)$.
\end{remark}

\subsection{Symmetric obstruction theory on $P(D)$}\label{ss.ConstructionofSOT}

Given $X,L,D,P(X)$ as in Theorem \ref{Thm.Lefschetz}, consider the commutative diagram
\[\xymatrix{ P(D)\times D \ar@{^{(}->}[r]^d \ar[rd]_{\pi_{D}} & P(D) \times X  \ar[d]^{\pi_{X}}  \\ &P(D)  }\]
of schemes. Recall that the Behrend-Fantechi virtual cycle $[P(D)]\virt_{BF}$ is constructed from the 2-term perfect obstruction theory 
\[\psi_D : \bR\hom_{\pi_D}(\cJ,\cJ\otimes L)_0[2] \xrightarrow{At(\cJ)} \LL_{P(D)} \]
induced by the Atiyah class $At(\cJ) : \cJ \to \cJ \otimes \LL_{P(D) \times D}[1]$ of the universal complex  $\cJ=[\O_{P(D)\times D}\to \cF]$.


\begin{proposition}\label{Prop.ConstructionofSOT}
There exists a canonical 3-term symmetric complex 
\[\bR\hom_{\pi_X}(d_*\cJ,d_*\cJ)_\#[3]\]
on $P(D)$ with a distinguished triangle
\beq\label{N.Eq.JJs} 
\xymatrix@-1.2pc{ 
\bR\hom_{\pi_D}(\cJ,\cJ)_0[3] \ar[r]^-{\epsilon\dual} & \bR\hom_{\pi_X}(d_*\cJ, d_*\cJ)_\#[3] \ar[r]^-{\epsilon} & \bR\hom_{\pi_D}(\cJ,\cJ\otimes L)_0[2] }
\eeq 
for some map $\epsilon$. Here we identified
\begin{align*}
(\bR\hom_{\pi_X}(d_*\cJ,d_*\cJ)_\#[3])\dual[2] &= \bR\hom_{\pi_X}(d_*\cJ,d_*\cJ)_\#[3]\\ 
(\bR\hom_{\pi_D}(\cJ,\cJ\otimes L)_0[2])\dual[2] &= \bR\hom_{\pi_D}(\cJ,\cJ)_0[3]
\end{align*}
via the symmetric form of $ \bR\hom_{\pi_X}(d_*\cJ,d_*\cJ)_\#[3]$ and the relative Serre duality.
Moreover, the symmetric form of $\bR\hom_{\pi_X}(d_*\cJ,d_*\cJ)_\#[3]$ is induced from the standard symmetric form of $\bR\hom_{\pi_X}(d_*\cJ,d_*\cJ)[3]$ in the sense of Lemma \ref{Lem.JJs} below.
\end{proposition}

Let us first assume that Proposition \ref{Prop.ConstructionofSOT} holds. Then we can define a 3-term symmetric obstruction theory on $P(D)$ as the composition
\beq\label{Eq.P(D)SOT}
\phi_D : \bR\hom_{\pi_X}(d_*\cJ, d_*\cJ)_\#[3] \mapright{\epsilon} \bR\hom_{\pi_D}(\cJ,\cJ\otimes L)_0[2] \mapright{\psi_D} \LL_{P(D)}.
\eeq
By the reduction formula in Proposition \ref{Prop.ReductionFormula}, the associated Oh-Thomas virtual cycle is identical to the Behrend-Fantech virtual cycle
\beq\label{L32}
[P(D)]\virt_{OT} = \sum_e (-1)^{\sigma(e)}[P(D)^e]\virt_{BF}
\eeq
up to canonical signs $(-1)^{\sigma(e)}$ on connected components $P(D)^e$ of $P(D)$.


\medskip

The rest of this subsection is devoted to the proof of Proposition \ref{Prop.ConstructionofSOT}. For simplicity, we will use the following abbreviations:
\[\sD= P(D) \times D, \quad \sX=P(D) \times X, \and \sP=P(D) \,.\]
We start with the following well-known lemma.

\begin{lemma}
There is a canonical distinguished triangle
\beq\label{N.Eq.JJ} \xymatrix@C-1pc{ \bR\hom_{d}(\cJ,\cJ) \ar[r]^-{\xi\dual} & \bR\hom_{\sX}(d_*\cJ,d_*\cJ) \ar[r]^-{\xi} & \bR \hom_{d}(\cJ,\cJ\otimes L)[-1] }\eeq
on $\sX$, where $\bR\hom_d:=\bR d_*\bR\hom_{\sD}$. Here we identified
\[\bR\hom_{\sX}(d_*\cJ,d_*\cJ) = \bR\hom_{\sX}(d_*\cJ,d_*\cJ) \dual\] \[\bR\hom_{d}(\cJ,\cJ) = (\bR\hom_d(\cJ,\cJ\otimes L)[-1])\dual\] 
via the trace map and the Grothendieck-Serre duality.
\end{lemma}

\begin{proof}
Since the self-dual distinguished triangle \eqref{N.Eq.JJ} is well-known, we will only sketch how \eqref{N.Eq.JJ} is constructed. Note that there is a canonical distinguished triangle 
\beq\label{N.Eq.BO} \xymatrix{ \cJ \otimes L\dual[1] \ar[r]& d^* d_* \cJ \ar[r]& \cJ \ar[r] & \cJ \otimes L\dual[2]} \eeq 
on $\sD$ by \cite[Lemma 3.3]{BO} (see also \cite[Corollary 11.4]{Huybrechts}).\footnote{In \cite{BO,Huybrechts}, they only considered smooth schemes but their proofs also work for arbitrary schemes.} Here the first two arrows are given by the adjunctions 
\[d_*\dashv d^* \otimes L[-1] \and d^* \dashv d_*\,.\]
Applying $\bR\hom_{\sX}(-,\cJ)$ to \eqref{N.Eq.BO}, we obtain the desired distinguished triangle \eqref{N.Eq.JJ}. The self-dualness of the distinguished triangle \eqref{N.Eq.JJ} follows from the naturality of the duality maps and the adjunction maps. 
\end{proof}

We will construct a symmetric complex $\bR\hom_{\sX}(d_*\cJ,d_*\cJ)_\#$ as the "reduction" of the "symmetric complex" $\bR\hom_{\sX}(d_*\cJ,d_*\cJ)$ by the "isotropic subcomplex"
\[\bR\hom_{\sX}(d_*\cJ,d_*\cJ) \xrightarrow{tr \circ \xi} d_*(\sO_{\sP}\boxtimes L|_D)[-1].\]
Strictly speaking, $\bR\hom_{\sX}(d_*\cJ,d_*\cJ)$ is not a symmetric complex in the sense of Definition \ref{Def.SymCplx}. However, Lemma \ref{Lem.Gen.Reduction} in Appendix \ref{A.Reduction} shows that we can still form a reduction.


\begin{lemma}\label{Lem.JJs}
There is a perfect complex 
$\bR\hom_{\sX}(d_*\cJ,d_*\cJ)_\#$
on $\sX$ and a self-dual isomorphism
\beq\label{Eq.JJs.sym} \theta:\bR\hom_{\sX}(d_*\cJ,d_*\cJ)_{\#}\cong (\bR\hom_{\sX}(d_*\cJ,d_*\cJ)_{\#})\dual,\eeq i.e., $\theta\dual=\theta$, satisfying the following two properties:
\begin{enumerate}
\item There is a morphism of distinguished triangles
\beq\label{Eq.JJs.reduction}
\xymatrix{
\bR\hom_{\sX}(d_*\cJ,d_*\cJ)_R \ar[r]^-{\mu\dual} \ar[d]^{\nu\dual} & \bR\hom_{\sX}(d_*\cJ,d_*\cJ) \ar[r]^-{tr \circ \xi} \ar[d]^{\mu} & d_*(\O_{\sP} \boxtimes L|_{D})[-1] \ar@{=}[d] \\ 
\bR\hom_{\sX}(d_*\cJ,d_*\cJ)_{\#} \ar[r]^-{\nu} & \bR\hom_{\sX}(d_*\cJ,d_*\cJ)_L  \ar[r] & d_*(\O_{\sP} \boxtimes L|_{D})[-1] 
}\eeq
for some maps $\mu$ and $\nu$, and perfect complexes $\bR\hom_{\sX}(d_*\cJ,d_*\cJ)_L$ and $\bR\hom_{\sX}(d_*\cJ,d_*\cJ)_R := (\bR\hom_{\sX}(d_*\cJ,d_*\cJ)_L)\dual.$
\item There exists a distinguished triangle  
\beq\label{Eq.JJs} 
\xymatrix{ \bR\hom_{d}(\cJ,\cJ)_0 \ar[r]^-{\epsilon\dual} &\bR\hom_{\sX}(d_*\cJ, d_*\cJ)_\# \ar[r]^-{\epsilon} & \bR\hom_{d}(\cJ,\cJ\otimes L)_0[-1] }
\eeq 
for some map $\epsilon$
such that the diagram
\beq\label{Eq.JJs.compatibility}
\xymatrix{
\bR\hom_{\sX}(d_*\cJ,d_*\cJ)_R \ar[r]^-{\mu\dual} \ar[d]^{\nu\dual} & \bR\hom_{\sX}(d_* \cJ,d_*\cJ) \ar[d]^{\xi}\\
\bR\hom_{\sX}(d_*\cJ,d_*\cJ)_{\#} \ar[r]^-{(\epsilon,0)} & \bR\hom_d(\cJ,\cJ\otimes L)[-1]
}\eeq
commutes. Here we identified 
\[\bR\hom_d(\cJ,\cJ\otimes L)=\bR\hom_d(\cJ,\cJ\otimes L)_0 \oplus d_*(\O_{\sP}\boxtimes L|_{D}).\]
\end{enumerate}
Moreover, the pair of the perfect complex $\bR\hom_{\sX}(d_*\cJ,d_*\cJ)_\#$ and the self-dual isomorphism $\theta$ is uniquely determined by the diagram \eqref{Eq.JJs.reduction}.
\end{lemma}

\begin{proof}
We first simplify the notations as follows\footnote{The notations $A,B,C,...,b,c,d,...,p,q,r,...$ introduced in the proof of Lemma \ref{Lem.JJs} will not be used in the rest of this paper.}: Let
\begin{align*}
A & := \bR\hom_{\sX}(d_*\cJ,d_*\cJ)\\
B &:=\bR\hom_{d}(\cJ,\cJ\otimes L)_0[-1], \quad C :=d_*(\O_{\sP}\boxtimes L|_D)[-1] 
\end{align*}
be the abbreviations of the perfect complexes on $\sX$. Note that
\[\Hom_{\sX}(C\dual,(B\oplus C)[-1]) = \Hom_{\sX}(d_*\O_{\sD},\bR\hom_{d}(\cJ,\cJ\otimes L)[-2])=0.\]
Hence we have
\beq\label{N.Van}
\Hom_{\sX}(C\dual,B[-1]) = \Hom_{\sX}(C\dual, C[-1])=0.
\eeq
By Lemma \ref{Lem.Gen.Reduction}, we can form a perfect complex $\bR\hom_{\cX}(d_*\cJ,d_*\cJ)_\#$ and a self-dual isomorphism \eqref{Eq.JJs.sym} with the diagram \eqref{Eq.JJs.reduction}. The uniqueness also follows from Lemma \ref{Lem.Gen.Reduction}. Let
\[D  := \bR\hom_{\sX}(d_*\cJ,d_*\cJ)_L \and E := \bR\hom_{\sX}(d_*\cJ,d_*\cJ)_\#\]
be the abbreviations of the perfect complexes on $\sX$.

Now we will form the distinguished triangle \eqref{Eq.JJs}. By the octahedral axiom, we can form a commutative diagram of four distinguished triangles
\beq\label{N.Oct.1}
\xymatrix{
B\dual \oplus C\dual \ar@{=}[r] \ar@{.>}[d]^{(p,q)}& B\dual \oplus C\dual \ar[d]^{(b\dual,c\dual)} & \\
F \ar@{.>}[r]^{f} \ar@{.>}[d]^{c \circ f} & A \ar[r]^{b} \ar[d]^{(b,c)} & B \ar@{=}[d] & \\
C \ar[r]^-{(0,1)} & B \oplus C \ar[r]^-{(1,0)} & B
}\eeq
for some perfect complex $F$ and maps $p$, $q$, $f$, where the middle vertical distinguished triangle is \eqref{N.Eq.JJ} and $(b,c):=\xi$. 
Again by the octahedral axiom, we have a commutative diagram of four distinguished triangles
\beq\label{N.Oct.2}
\xymatrix{
C\dual \ar@{=}[r] \ar[d]^{(0,1)}& C\dual \ar[d]^{q} & \\
B\dual \oplus C\dual \ar[r]^-{(p,q)} \ar[d]^{(1,0)} & F \ar[r]^-{c \circ f} \ar@{.>}[d]^{g} & C \ar@{=}[d] & \\
B\dual \ar@{.>}[r]^-{g \circ p} & G \ar@{.>}[r]^-{r} & C
}\eeq
for some perfect complex $G$ and maps $g$, $r$, where the middle horizontal distinguished triangle is the left vertical distinguished triangle in \eqref{N.Oct.1}. Note that we have $f \circ q = c\dual$ by the left upper square in \eqref{N.Oct.1}. By the octahedral axiom, we have a commutative diagram of four distinguished triangles
\beq\label{N.Oct.3}
\xymatrix{
C\dual \ar@{=}[r] \ar[d]^{q}& C\dual \ar[d]^{c\dual} & \\
F \ar[r]^-{f} \ar[d]^{g} & A \ar[r]^-{b} \ar[d]^{\mu} & B \ar@{=}[d] & \\
G \ar@{.>}[r]^-{s} & D \ar@{.>}[r]^-{t} & B
}\eeq
for some map $s$ and $t$, where 
the middle vertical distinguished triangle is the dual of the top horizontal distinguished triangle in \eqref{Eq.JJs.reduction}.

We claim that we can form a commutative diagram of four distinguished triangles
\beq\label{N.Oct.4}
\xymatrix{
B \dual \ar[r]^{g \circ p} \ar@{.>}[d]^{e} & G \ar[r]^{r} \ar[d]^{s} & C \ar@{=}[d] \\
E \ar[r]^{\nu} \ar@{.>}[d]^{\epsilon} & D \ar[r]^{d} \ar[d]^{t} & C \\
B \ar@{=}[r] & B
}\eeq
for some maps $\epsilon$ and $e$, where the map $d$ and the middle horizontal distinguished triangle in \eqref{N.Oct.4} are given by \eqref{Eq.JJs.reduction}.
To deduce the claim from the octahedral axiom, we need to show that the right upper square in \eqref{N.Oct.4} commutes, $r=d \circ s.$ It suffices to show that
\[r \circ g = d\circ s \circ g : F \to G \rightrightarrows C\]
since the middle vertical sequence in \eqref{N.Oct.2} is a distinguished triangle, and $\Hom_{\sX}(C\dual[1],C)=0$ by \eqref{N.Van}. We have $r \circ g = c \circ f$ by \eqref{N.Oct.2} and $d\circ s \circ g = d\circ \mu \circ f =c \circ f$ by \eqref{N.Oct.3} and \eqref{Eq.JJs.reduction}. It proves the claim.

We claim that
\beq\label{N.eq3}e =\epsilon\dual : B\dual \rightrightarrows E.\eeq
By the distinguished triangle in \eqref{Eq.JJs.reduction}, it suffices to show that \[\nu \circ e  =\nu \circ \epsilon\dual : B\dual \rightrightarrows E \to D\]
since $\Hom_{\sX}(B\dual,C[-1])=0$ by \eqref{N.Van}. Consider the commutative diagram
\[\xymatrix{
D\dual \ar[r]^{\mu\dual}\ar[d]^{\nu\dual} & A \ar[rd]^{b} \ar[d]^{\mu} & \\
E \ar[r]^{\nu} \ar@/_0.4cm/[rr]_{\epsilon}& D \ar[r]^{t} & B
}\]
induced by the commutative diagrams \eqref{Eq.JJs.reduction}, \eqref{N.Oct.3}, and \eqref{N.Oct.4}. By taking dual, we obtain
\beq\label{N.eq1}\nu \circ \epsilon\dual = \mu \circ b\dual.\eeq
On the other hand, consider the commutative diagram
\[\xymatrix{
B\dual \oplus C\dual \ar[r]_-{(p,q)} \ar[d]^{(1,0)} \ar@/^0.4cm/[rr]^-{(b\dual,c\dual)} & F \ar[r]_{f} \ar[d]^{g} & A \ar[d]^{\mu}\\
B\dual \ar[r]^-{g \circ p} \ar@/_0.4cm/[rr]_{\nu \circ e}& G \ar[r]^{s} & D }\]
induced by the commutative diagrams \eqref{N.Oct.1}, \eqref{N.Oct.2}, \eqref{N.Oct.3}, and \eqref{N.Oct.4}. Hence we have
\beq\label{N.eq2}
\nu \circ e = \mu \circ b\dual.\eeq
The two equations \eqref{N.eq1} and \eqref{N.eq2} proves the claim \eqref{N.eq3}.

By the equation \eqref{N.eq3}, the distinguished triangle
\[\xymatrix{
B\dual \ar[r]^-{e=\epsilon\dual} & E \ar[r]^{\epsilon} & B
}\]
in \eqref{N.Oct.4} gives us the desired distinguished triangle \eqref{Eq.JJs}. The dual of the equation \eqref{N.eq1}
\[\epsilon \circ \nu\dual = b \circ \mu \dual \]
induces the desired commutative diagram \eqref{Eq.JJs.compatibility} since $c \circ \mu \dual =0$ and $\xi=(b,c)$. It completes the proof.
\end{proof}

Now Proposition \ref{Prop.ConstructionofSOT} follows directly from Lemma \ref{Lem.JJs}.

\begin{proof}[Proof of Proposition \ref{Prop.ConstructionofSOT}]
Let
\[\bR\hom_{\pi_{X}}(d_*\cJ,d_*\cJ)_\# := \bR(\pi_X)_*\bR\hom_{\sX}(d_*\cJ,d_*\cJ)_\#.\]
Then the self-dual isomorphism \eqref{Eq.JJs.sym} induces a self-dual isomorphism
\[(\bR\hom_{\pi_{X}}(d_*\cJ,d_*\cJ)_\#)\dual \cong \bR\hom_{\pi_{X}}(d_*\cJ,d_*\cJ)_\#[4]\]
by the naturality of the relative Serre duality map. Similarly, the distinguished triangle \eqref{Eq.JJs} in Lemma \ref{Lem.JJs}(2) gives us the desired distinguished triangle \eqref{N.Eq.JJs}.
\end{proof}

\subsection{Comparison of obstruction theories}\label{ss.ComparisonofOT}


By Proposition \ref{Prop.ComparisonofModuli} and Proposition \ref{ss.ConstructionofSOT}, it suffices to show that \beq\label{L7} j^![P(X)]\virt_{OT} = [P(D)]\virt_{OT} \eeq holds to prove the Lefschetz principle in Theorem \ref{Thm.Lefschetz}. We would like to deduce \eqref{L7} from the virtual pullback formula in Theorem \ref{Thm.Functoriality}. To do this, we will compare the obstruction theories of $P(X)$ and $P(D)$.

\medskip

Given $X,L,D,P(X)$ as in Theorem \ref{Thm.Lefschetz}, let
\[\xymatrix{ P(D)\times D \ar[r]^d \ar[rd]_{\pi_{D}} & P(D) \times X  \ar[d]^{\pi_{X}} \ar[r] & P(X)\times X \ar[d]^{\pi}  \\ &P(D) \ar[r]^j & P(X) }\]
be the canonical commutative diagram. Recall that the Oh-Thomas virtual cycle $[P(X)]\virt_{OT}$ is constructed from the 3-term symmetric obstruction theory
\beq\label{L2} \phi_X:\bR\hom_{\pi}(\II,\II)_0[3] \xrightarrow{At(\II)} \LL_{P(X)} \eeq 
induced by the Atiyah class $At(\II) : \II \to \II \otimes \LL_{P(X) \times X}[1]$ of the universal complex $\II=[\O_{P(X)\times X} \to \FF]$. Let
\[\cI=\II|_{P(D)\times X} = [\O_{P(D)\times X} \to d_*\cF]\]
be the restriction, where $\cJ = [\O_{P(D)\times D} \to \cF]$ is the universal complex of $P(D)$.

\begin{proposition}\label{Prop.L.Comp.OT}
The two symmetric obstruction theory $\phi_X$ in \eqref{L2} and $\phi_D$ in \eqref{Eq.P(D)SOT} are compatible as follows: 
\begin{enumerate}
\item There is a {\em reduction diagram} (see Proposition \ref{Prop.Reduction}), i.e. a morphism of distinguished triangles 
\beq\label{Eq.JJstoII0}\xymatrix@C-1.2pc{
\bR\hom_{\pi_{X}}(d_*\cJ,\cI)_\#[3] \ar[r]^-{\alpha\dual} \ar[d]^{\beta\dual} & \bR\hom_{\pi_{X}}(d_*\cJ, d_*\cJ)_\#[3] \ar[r]^-{\delta} \ar[d]^{\alpha} & \bR\hom_{\pi_{X}}(d_*\cF,L\dual)[5] \ar@{=}[d]\\
\bR\hom_{\pi_{X}}(\cI,\cI)_0[3] \ar[r]^-{\beta} & \bR\hom_{\pi_{X}}(\cI,d_*\cJ)_\#[3] \ar[r]^-{\gamma} & \bR\hom_{\pi_{X}}(d_*\cF,L\dual)[5]
}\eeq
for some maps $\alpha,\beta,\gamma,\delta$ and perfect complexes $\bR\hom_{\pi_{X}}(\cI,d_*\cJ)_\#$ and $\bR\hom_{\pi_{X}}(d_*\cJ,\cI)_\#:=(\bR\hom_{\pi_{X}}(\cI,d_*\cJ)_\#)\dual[-4]$ 
on $P(D)$.
\item The Atiyah classes of $\II$ and $\cJ$ are compatible, i.e., the diagram
\beq\label{Eq.L.Atiyah}\xymatrix{
\bR\hom_{\pi_{X}}(d_*\cJ,\cI)_\#[3] \ar[d]^{\beta\dual}\ar[r]^-{\alpha\dual} & \bR\hom_{\pi_{X}}(d_*\cJ,d_*\cJ)_\#[3] \ar[d]^{\epsilon} \\
\bR\hom_{\pi_{X}}(\cI,\cI)_0[3] \ar[d]^{At(\II)} \ar[rd]|-{At(\cI)} & \bR\hom_{\pi_{D}}(\cJ,\cJ\otimes L)_0[2] \ar[d]^{At(\cJ)}\\
\LL_{P(X)}|_{P(D)} \ar[r] & \LL_{P(D)}
}\eeq
commutes.
\end{enumerate}
\end{proposition}

In the rest of this subsection, we will prove Proposition \ref{Prop.L.Comp.OT} through several steps. We first fix the notations. For simplicity, we will use the following abbreviations
\[\sD= P(D) \times D, \quad \sX=P(D) \times X, \and \sP=P(D) \]
of schemes as in \S\ref{ss.ConstructionofSOT}. Let
\beq\label{Eq.JOF} \xymatrix{\cJ \ar[r]^-{i} & \O_{\sD} \ar[r]^-{s} & \cF \ar[r]^-{e} & \cJ[1]} \eeq 
\beq\label{Eq.IOF} \xymatrix{\cI \ar[r] & \O_{\sX} \ar[r]^{} & d_*\cF \ar[r]^-{f} & \cI[1]} \eeq
be the two canonical distinguished triangles on $\sD$ and $\sX$, respectively. By applying the octahedral axiom to the two distinguished triangles \eqref{Eq.JOF} and \eqref{Eq.IOF}, we obtain a third distinguished triangle
\beq\label{Eq.LdIJ} \xymatrix{\cL\dual \ar[r]^-{u}& \cI \ar[r]^-{v} & d_*\cJ \ar[r]^-{w}& \cL\dual [1]} \eeq 
on $\sX$ such that $d_*e=v \circ f$. Here we denoted $\cL=\O_{P(D)}\boxtimes L$.

\subsubsection{Prototype}

We start with a {\em prototype} of the reduction diagram \eqref{Eq.JJstoII0} in Proposition \ref{Prop.L.Comp.OT} where we use
\[\bR\hom_{\sX}(d_*\cJ,d_*\cJ) \and \bR\hom_{\sX}(\cI,d_*\cJ)\]
instead of 
$\bR\hom_{\sX}(d_*\cJ,d_*\cJ)_\#$ and $\bR\hom_{\sX}(\cI,d_*\cJ)_\#.$

\begin{lemma}\label{Lem.L.Prototype}
There is a (not necessarily commutative) diagram of two distinguished triangles
\beq\label{Eq.JJtoII0}
\xymatrix{
\bR\hom_{\sX}(d_*\cJ,\cI) \ar[r]^-{l_v} \ar[d]^{\rho \circ r_v} & \bR\hom_{\sX}(d_*\cJ,d_*\cJ) \ar[r]^-{l_w} \ar[d]^{r_v} & \bR\hom_{\sX}(d_*\cJ,\cL\dual)[1] \ar@{=}[d]\\
\bR\hom_{\sX}(\cI,\cI)_0 \ar[r]^-{l_v\circ \iota} & \bR\hom_{\sX}(\cI,d_*\cJ) \ar[r]^-{\eta} & \bR\hom_{\sX}(d_*\cJ,\cL\dual)[1]
}\eeq
for some $\eta$ such that $r_v \circ \eta = l_w$. Here $l_v$ denotes the left composition with $v$ and $r_v$ denotes the right composition with $v$. Also $\iota$ and $\rho$ are the canonical maps in the direct sum diagram 
\beq\label{Eq.D.II}
\xymatrix{\bR\hom_{\sX}(\cI,\cI)_0 \ar@<.5ex>[r]^-{\iota} & \bR\hom_{\sX}(\cI,\cI) \ar@<.5ex>[r]^-{tr} \ar@<.5ex>[l]^-{\rho} & \O_{\sX} \ar@<.5ex>[l]^-{1}} 
\eeq 
where $tr$ denotes the trace map.
\end{lemma}

\begin{proof} 
Applying $\bR\hom_{\sX}(d_*\cJ,-)$ to \eqref{Eq.LdIJ}, we can form the upper distinguished triangle in \eqref{Eq.JJtoII0}. The lower distinguished triangle can be obtained by applying the octahedral axiom to the commutative diagram \[\xymatrix{ \bR\hom_{\sX}(\cI,\cL\dual) \ar[r]^-{l_u} \ar[d]^{r_u} & \bR\hom_{\sX}(\cI,\cI) \ar[d]^{tr} \\ \bR\hom_{\sX}(\cL\dual,\cL\dual) \ar@{=}[r] & \O_{\sX} }\] 
where the three distinguished triangles are the obvious ones given by \eqref{Eq.LdIJ} and \eqref{Eq.D.II}. The formula $r_v \circ \eta = l_w$ follows from the octahedral axiom.
\end{proof}

\subsubsection{} 

We will replace the complexes in \eqref{Eq.JJtoII0} as \begin{align*} \bR\hom_{\sX}(d_*\cJ,d_*\cJ) & \leadsto \bR\hom_{\sX}(d_*\cJ,d_*\cJ)_\# \\ \bR\hom_{\sX}(\cI,d_*\cJ) & \leadsto \bR\hom_{\sX}(\cI,d_*\cJ)_\# \\ \bR\hom_{\sX}(d_*\cJ,\cL\dual)[1] & \leadsto \bR\hom_{\sX}(d_*\cF,\cL\dual)[2] \end{align*} for some perfect complex $\bR\hom_{\sX}(\cI,d_*\cJ)_\#$. Then the (non-commutative) diagram \eqref{Eq.JJtoII0} will become commutative and induce the desired commutative diagram \eqref{Eq.JJstoII0}. 

Define the perfect complex $\bR\hom_{\sX}(\cI,d_*\cJ)_\#$ as the cone of the composition $r_v\circ \xi\dual \circ \id_{\cJ}$ (see \eqref{N.Eq.JJ}). Then it fits into a distinguished triangle
\beq\label{N.Eq.IJs}
\xymatrix@+1pc{d_*\O_{\sD} \ar[r]^-{r_v\circ \xi\dual \circ \id_{\cJ}} & \bR\hom_{\sX}(\cI,d_*\cJ) \ar[r]^-{\lambda} & \bR\hom_{\sX}(\cI,d_*\cJ)_\# \ar[r] & d_*\O_{\sD}[1]}
\eeq
for some map $\lambda$. Let $\bR\hom_{\sX}(d_*\cJ,\cI)_\# := (\bR\hom_{\sX}(\cI,d_*\cJ)_\#)\dual$ be the dual.

By the octahedral axiom, we obtain a commutative diagram of four distinguished triangles
\beq\label{Eq.IJs}
\xymatrix{
 & \bR\hom_{\sX}(\cL\dual, d_*\cJ)[-1] \ar@{=}[r] \ar[d]^{r_w} & \bR\hom_{\sX}(\cL\dual,d_*\cJ)[-1] \ar@{.>}[d]^{\mu\circ r_w} \\
d_*\O_{\sD} \ar[r]^-{\xi\dual \circ \id_{\cJ}} \ar@{=}[d] & \bR\hom_{\sX}(d_*\cJ,d_*\cJ) \ar[r]^-{\mu} \ar[d]^{r_v} & \bR\hom_{\sX}(d_*\cJ,d_*\cJ)_L \ar@{.>}[d]^{\zeta}\\
d_*\O_{\sD} \ar[r]^-{r_v \circ \xi\dual \circ \id_{\cJ}} & \bR\hom_{\sX}(\cI,d_*\cJ) \ar[r]^-{\lambda} & \bR\hom_{\sX}(\cI,d_*\cJ)_\#
}\eeq
for some map $\zeta$, where the middle vertical distinguished triangle is induced by \eqref{Eq.LdIJ}, the middle horizontal distinguished triangle is the dual of the top horizontal distinguished triangle in \eqref{Eq.JJs.reduction}, and the bottom horizontal distinguished triangle is \eqref{N.Eq.IJs}.



\subsubsection{Maps $\alpha$, $\beta$, $\gamma$, $\delta$}

Define the maps $\alpha$ and $\beta$ as the compositions:
\begin{align*}
\alpha &: \bR\hom_{\sX}(d_*\cJ,d_*\cJ)_\# \mapright{\nu} \bR\hom_{\sX}(d_*\cJ,d_*\cJ)_{L} \mapright{\zeta} \bR\hom_{\sX}(\cI,d_*\cJ)_\# \,,
\\
\beta &: \bR\hom_{\sX}(\cI,\cI)_0 \mapright{l_v\circ \iota} \bR\hom_{\sX}(\cI,d_*\cJ) \mapright{\lambda} \bR\hom_{\sX}(\cI,d_*\cJ)_\# \,. 
\end{align*}

We define the map $\gamma$ as follows:

\begin{lemma}\label{Lem.gamma}
There exists a commutative diagram of four distinguished triangles
\beq\label{Eq.gamma}
\xymatrix{
& d_*\O_{\sD} \ar@{=}[r] \ar[d]^{r_v\circ \xi\dual \circ \id_{\cJ}} & \bR\hom_{\sX}(d_*\O_{\sD}, \cL\dual)[1] \ar[d]^{r_i} \\
\bR\hom_{\sX}(\cI,\cI)_0 \ar[r]^-{l_v\circ \iota} \ar@{=}[d] & \bR\hom_{\sX}(\cI,d_*\cJ) \ar[r]^-{\eta} \ar[d]^{\lambda} & \bR\hom_{\sX} (d_*\cJ,\cL\dual)[1] \ar[d]^{r_e} \\
\bR\hom_{\sX}(\cI,\cI)_0 \ar@{.>}[r]^-{\beta}  & \bR\hom_{\sX}(\cI,d_*\cJ)_\# \ar@{.>}[r]^-{\gamma} & \bR\hom_{\sX} (d_*\cF,\cL\dual)[2] 
}\eeq
for some map $\gamma$, where the middle vertical distinguished triangle is  \eqref{N.Eq.IJs}, the right vertical distinguished triangle is induced by \eqref{Eq.JOF}, and the middle horizontal distinguished triangle is the bottom horizontal distinguished triangle in \eqref{Eq.JJtoII0}.
\end{lemma}

\begin{proof}
We will use the octahedral axiom to obtain the diagram \eqref{Eq.gamma}. Thus it suffices to show that the upper right square in \eqref{Eq.gamma} commutes. 

We claim that the squares in the diagram
\beq\label{N11}\xymatrix{
d_*\O_{\sD} \ar@{=}[r] \ar[d]^{ \xi\dual \circ \id_{\cJ}} & \bR\hom_{\sX}(d_*\O_{\sD} , \cL\dual[1]) \ar[d]^{r_i} \\
\bR\hom_{\sX}(d_*\cJ,d_*\cJ) \ar[r]^{l_w} \ar[d]^{r_v}  & \bR\hom_{\sX}(d_*\cJ,L\dual[1]) \ar[d]^{r_v} \\
\bR\hom_{\sX}(\cI, d_*\cJ) \ar[r]^{l_w} & \bR\hom_{\sX} ( \cI, \cL\dual[1]) 
}\eeq
commute. Indeed, under the adjunction $d_* \dashv d^* \otimes \cL[-1]$, we have a correspondence
\[\left(\cJ \mapright{i} \O_{\sD}\right) \longleftrightarrow \left(d_*\cJ \mapright{w} \cL\dual[1]\right)\]
between the two canonical maps in \eqref{Eq.JOF} and \eqref{Eq.LdIJ}. Hence we have a commutative diagram
\[\xymatrix{
\O_{\sD} \ar@{=}[r] \ar[d] & \bR\hom_{\sD} ( \O_{\sD}, \O_{\sD}) \ar[d]^{r_i} \\
\bR\hom_{\sD} (\cJ, d^*d_*\cJ \otimes \cL [-1]) \ar[r]^-{l_{d^*w}} & \bR\hom_{\sD}(\cJ,\O_{\sD})
}\]
on $\sD$, which proves the commutativity of the upper square in \eqref{N11}. The commutativity of the lower square in \eqref{N11} is just the naturality of the composition functors.

If we compose the two maps 
\beq\label{N12}
d_*\O_{\sD} \rightrightarrows \bR\hom_{\sX}(d_*\cJ,\cL\dual[1])
\eeq
in the upper right square in \eqref{Eq.gamma} with the map 
\[r_v:\bR\hom_{\sX}(d_*\cJ,\cL\dual[1]) \to \bR\hom_{\sX}(\cI,\cL\dual[1]),\]
then we obtain the total square in \eqref{N11} since $r_v\circ \eta=l_w$ by Lemma \ref{Lem.L.Prototype}. Hence the difference of the two maps \eqref{N12} comes from a map 
\beq\label{N14}d_*\O_{\sD} \to \bR\hom_{\sX}(\cL\dual[1],\cL\dual[1]) = \O_{\sX}.\eeq
Since $\Hom_{\sX}(d_*\O_{\sD} , \O_{\sX} ) = 0$, the map \eqref{N14} is zero and the two maps \eqref{N12} coincide.
\end{proof}

We next define the map $\delta$ as follows:

\begin{lemma}
There exists a commutative diagram of four distinguished triangles
\beq\label{Eq.delta}
\xymatrix{
& d_*\O_{\sD} \ar@{=}[r] \ar[d] & \bR\hom_{\sX}(d_*\O_{\sD}, \cL\dual)[1] \ar[d]^{r_i} \\
\bR\hom_{\sX}(d_*\cJ,\cI)_{\#} \ar[r]^-{\zeta\dual} \ar@{=}[d] & \bR\hom_{\sX}(d_*\cJ,d_*\cJ)_R \ar[r]^-{l_w\circ \mu\dual} \ar[d]^{\nu\dual} & \bR\hom_{\sX} (d_*\cJ,\cL\dual)[1] \ar[d]^{r_e} \\
\bR\hom_{\sX}(d_*\cJ,\cI)_{\#} \ar@{.>}[r]^-{\alpha\dual}  & \bR\hom_{\sX}(d_*\cJ,d_*\cJ)_{\#} \ar@{.>}[r]^-{\delta} & \bR\hom_{\sX} (d_*\cF,\cL\dual)[2] 
}\eeq
for some map $\delta$, where the middle vertical distinguished triangle is the dual of the bottom horizontal distinguished triangle in \eqref{Eq.JJs.reduction}, the right vertical distinguished triangle is given by \eqref{Eq.JOF}, and the middle horizontal distinguished triangle is the dual of the right vertical distinguished triangle in \eqref{Eq.IJs}.
\end{lemma}

\begin{proof}
As in Lemma \ref{Lem.gamma}, it suffices to show the commutativity of the upper right square in \eqref{Eq.delta}. This follows from the commutativity of the upper square in \eqref{N11}.
\end{proof}

\subsubsection{Reduction diagram}

Combining the four maps $\alpha$, $\beta$, $\gamma$, $\delta$, we can form the reduction diagram \eqref{Eq.JJstoII0} in Proposition \ref{Prop.L.Comp.OT}(1).


\begin{proof}[Proof of Proposition \ref{Prop.L.Comp.OT}(1)]
We claim that the diagram
\beq\label{N15}
\xymatrix{
\bR\hom_{\sX}(d_*\cJ,\cI)_\# \ar[r]^-{\alpha\dual} \ar[d]^{\beta\dual} & \bR\hom_{\sX}(d_*\cJ,d_*\cJ)_\# \ar[r]^-{\delta} \ar[d]^{\alpha} & \bR\hom_{\sX}(d_*\cF,\cL\dual)[2] \ar@{=}[d]\\ 
\bR\hom_{\sX}(\cI,\cI)_0 \ar[r]^-{\beta} & \bR\hom_{\sX}(\cI,d_*\cJ)_\# \ar[r]^-{\gamma} & \bR\hom_{\sX}(d_*\cF,\cL\dual)[2]
}\eeq
commutes. 

We will first show that the left square in \eqref{N15} commutes. Indeed, we have the commutative diagram
\beq\label{Eq.JJstoII0.left}
\xymatrix{
\bR\hom_{\sX}(d_*\cJ,\cI)_\# \ar[r]^{\zeta\dual} \ar[d]^{\lambda\dual} 
& \bR\hom_{\sX}(d_*\cJ,d_*\cJ)_R \ar[r]^{\nu\dual} \ar[d]^{\mu\dual} & \bR\hom_{\sX}(d_*\cJ,d_*\cJ)_\# \ar[d]^{\nu} 
\\
\bR\hom_{\sX}(d_*\cJ,\cI) \ar[r]^{l_v} \ar[d]^{r_v} & \bR\hom_{\sX}(d_*\cJ,d_*\cJ) \ar[r]^{\mu} \ar[d]^{r_v} & \bR\hom_{\sX}(d_*\cJ,d_*\cJ)_L \ar[d]^{\zeta}\\
\bR\hom_{\sX}(\cI,\cI) \ar[r]^{l_v} 
& \bR\hom_{\sX}(\cI,d_*\cJ) \ar[r]^{\lambda}  & \bR\hom_{\sX}(\cI,d_*\cJ)_\#
}\eeq
by \eqref{Eq.JJs.reduction} in Lemma \ref{Lem.JJs} and \eqref{Eq.IJs}. Note that the composition
\beq\label{L44}\O_{\sX} \mapright{\id_{\II}} \bR\hom_{\sX}(\cI,\cI) \xrightarrow{\lambda \circ l_v} \bR\hom_{\sX}(\cI,d_*\cJ)_\#\eeq
vanishes since we have a commutative diagram
\[\xymatrix{
\O_{\sX} \ar[r] \ar[d]^{\id_{\II}} & d_*\O_{\sD} \ar[d]^{r_v \circ \xi\dual \circ \id_{\cJ}}\\
\bR\hom_{\sX}(\cI,\cI) \ar[r]^-{l_v} & \bR\hom_{\sX}(\cI,d_*\cJ)
}\]
and an equation $\lambda \circ (r_v\circ \xi\dual \circ \id_{\cJ})=0$ by \eqref{N.Eq.IJs}. Since $\beta=\lambda\circ l_v\circ \iota$ by definition, and \eqref{L44} is zero, the diagram
\beq\label{N16}
\xymatrix{
\bR\hom_{\sX}(\cI,\cI) \ar[r]^-{l_v} \ar[d]^{\rho} & \bR\hom_{\sX}(\cI,d_*\cJ) \ar[d]^{\lambda} \\
\bR\hom_{\sX}(\cI,\cI)_0 \ar[r]^-{\beta} & \bR\hom_{\sX}(\cI,d_*\cJ)_\#
}\eeq
commutes. The two commutative diagrams \eqref{Eq.JJstoII0.left} and \eqref{N16} give us the commutativity of the left square in \eqref{N15}.

Now we will prove that the right square in \eqref{N15} commutes. Since 
\begin{align*}
 & \Hom_{\sX}(d_*\O_{\sD}[1], \bR\hom_{\sX}(d_*\cF,\cL\dual[2])) \\
=& \Hom_{\sX}(d_*\cF,d_*\O_{\sD})=\Hom_{\sD}(\cF,\O_{\sD}) \oplus \Ext^{-1}_{\sD}(\cF,\cL|_{\sD}) =0   ,
\end{align*}
it suffices to show that
\[\gamma \circ \alpha \circ \nu\dual = \delta \circ \nu\dual.\]
Consider a diagram
\[\xymatrix{
\bR\hom_{\sX}(d_*\cJ,d_*\cJ)_R \ar[r]^-{\nu\dual} \ar[d]^{\mu\dual} & \bR\hom_{\sX}(d_*\cJ,d_*\cJ)_\# \ar[d]^{\nu} \ar@/^2cm/[dd]^{\alpha} \ar@<6ex>@/^2cm/[ddd]^{\delta} \\
\bR\hom_{\sX}(d_*\cJ,d_*\cJ) \ar[r]^-{\mu} \ar[d]^{r_v} \ar@/_2cm/[dd]_{l_w} & \bR\hom_{\sX}(d_*\cJ,d_*\cJ)_L \ar[d]^{\zeta} \\
\bR\hom_{\sX}(\cI,d_*\cJ) \ar[r]^-{\lambda} \ar[d]^{\eta} & \bR\hom_{\sX}(\cI,d_*\cJ)_\# \ar[d]^{\gamma} \\
\bR\hom_{\sX}(d_*\cJ,\cL\dual)[1] \ar[r]^-{r_e} & \bR\hom_{\sX}(d_*\cF,\cL\dual)[2]
}\]
of commutative squares induced by \eqref{Eq.JJs.reduction}, \eqref{Eq.IJs}, and \eqref{Eq.gamma}. The two maps $\eta\circ r_v$ and $l_w$ are not necessarily equal, but their compositions with $r_v$,
\[\xymatrix{ \bR\hom_{\sX}(d_*\cJ,d_*\cJ) \ar@<.5ex>[r]^-{\eta\circ r_v} \ar@<-.5ex>[r]_-{l_w} & \bR\hom_{\sX}(d_*\cJ,\cL\dual)[1] \ar[r]^-{r_v} & \bR\hom_{\sX}(\cI,\cL\dual)[1]}\]
are equal since $r_v \circ \eta = l_w$ by Lemma \ref{Lem.L.Prototype}. Hence the compositions of the two maps $\eta\circ r_v$ and $l_w$ with $r_e$ are also equal since $d_*e = v \circ f$. By \eqref{Eq.delta}, we have
\begin{align*}
    \gamma \circ \alpha \circ \nu\dual = \gamma \circ \zeta \circ \nu \circ \nu\dual =r_e \circ \eta \circ r_v \circ \mu\dual = r_e \circ l_w \circ \mu\dual = \delta \circ \nu\dual,
\end{align*}
which proves the claim. 

If we apply $\bR(\pi_X)_*$ to \eqref{N15}, then we obtain the desired diagram \eqref{Eq.JJstoII0} (see Lemma \ref{Lem.R}). It completes the proof.
\end{proof}

\subsubsection{Atiyah classes}

Finally, we compare the Atiyah classes of $\cJ$ and $\II$ using derived algebraic geometry. 

\begin{proof}[Proof of Proposition \ref{Prop.L.Comp.OT}(2)]
We first claim that the diagram
\beq\label{Eq4.42}
\xymatrix@C-1pc{
\bR\hom_{\pi_X}(d_*\cJ,\cI)[3] \ar[r]^-{l_v} \ar[d]^{r_v} & \bR\hom_{\pi_X}(d_*\cJ,d_*\cJ)[3] \ar[rdd]|-{At_{P(D) \times X}(d_*\cJ)} \ar[r]^{\xi} & \bR\hom_{\pi_D}(\cJ,\cJ\otimes L)[2] \ar[dd]^{At_{P(D)\times D}(\cJ)} \\ \bR\hom_{\pi_X}(\cI,\cI)[3] \ar[d]^{At_{P(X)\times X}(\II)} \ar[rrd]|-{At_{P(D) \times X}(\cI)} \\
\LL_{P(X)}|_{P(D)} \ar[rr] && \LL_{P(D)}
}\eeq
commutes. Indeed, the commutativity of the middle square in \eqref{Eq4.42} follows from the functoriality of the Atiyah classes
\[\xymatrix{
\cI \ar[r]^v \ar[d]_{At_{P(D) \times X}(\cI)} & d_*\cJ \ar[d]^{At_{P(D) \times X}(d_*\cJ)}\\
\cI \otimes \LL_{P(D) \times X}[1] \ar[r]^v & d_*\cJ \otimes \LL_{P(D) \times X}[1] .
}\]
The lower triangle in \eqref{Eq4.42} is the commutative diagram of tangent maps for
\[\xymatrix{ P(D) \ar[r]^-{j} \ar@/_.6cm/[rr]_-{\cI} & P(X) \ar[r]^-{\II} & \mathbf{Perf}(X) }\]
by \cite[Appendix A]{STV}, where $\mathbf{Perf}(X)$ denotes the derived moduli stack of perfect complexes on $X$. The right triangle in \eqref{Eq4.42} is the commutative diagram of tangents maps for
\[\xymatrix{
P(D) \ar[r]^-{\cJ} \ar@/_.6cm/[rr]_-{d_*\cJ} & \mathbf{Perf}(D) \ar[r] & \mathbf{Perf}(X)
}\]
where the map $\mathbf{Perf}(D) \to \mathbf{Perf}(X)$ between the derived moduli stacks is given by the derived pushforward.\footnote{The commutativity of the right triangle in \eqref{Eq4.42} for the truncated cotangent complex $\trunc \LL_{P(D)}$ can be deduced by \cite[Theorem 2.5]{DSY} without using derived algebraic geometry.} 

By the square \eqref{Eq.JJs.compatibility}, the dual of the right bottom square in \eqref{Eq.IJs}, and the dual of the square \eqref{N16}, the commutative diagram \eqref{Eq4.42} implies the desired commutative diagram \eqref{Eq.L.Atiyah}.
\end{proof}

\subsection{Comparison of virtual cycles}\label{ss.ComparisonofVirtualCycles}

Finally, we prove the Lefschetz principle in Theorem \ref{Thm.Lefschetz} using the results in the previous subsections.

\begin{proof}[Proof of Theorem \ref{Thm.Lefschetz}]
By Proposition \ref{Prop.L.Comp.OT}, we can form two vertical morphisms
\beq\label{Eq4.44}
\xymatrix@C-1.2pc{
\bR\hom_{\pi_X}(d_*\cJ,\cI)_\#[3] \ar[r]^-{\alpha\dual} \ar[d]^{\beta\dual} & \bR\hom_{\pi_X}(d_*\cJ,d_*\cJ)_\#[3] \ar[r]^-{\delta} \ar[d]^{\alpha} & \bR\hom_{\pi_X}(d_*\cF,\cL\dual)[5] \ar@{=}[d]\\
\bR\hom_{\pi_X}(\cI,\cI)_0[3] \ar[r]^-{\beta} \ar[d]^{\phi_X} & \bR\hom_{\pi_X}(\cI,d_*\cJ)_\#[3] \ar[r]^-{\gamma} \ar@{.>}[d]^{\phi_D'} & \bR\hom_{\pi_X}(d_*\cF,\cL\dual)[5] \ar@{.>}[d]^{\phi'}\\
\trunc\LL_{P(X)}|_{P(D)} \ar[r]^-{a} & \trunc\LL_{P(D)} \ar[r] & \bL'_{P(D)/P(X)}
}\eeq
of the three horizontal distinguished triangles for some maps $\phi'_D$ and $\phi'$. Indeed, there exists a map
\[\phi'_D : \bR\hom_{\pi_X}(\cI,d_*\cJ)_\#[3] \to \trunc \LL_{P(D)}\]
such that $\phi'_D\circ\alpha=\phi_D$ since $\phi_D \circ \delta\dual =a \circ \phi_X \circ \beta\dual \circ \gamma\dual =0$. Then we have
\[a\circ \phi_X = \phi_D ' \circ \beta\]
since $\Hom_{P(D)}(\bR(\pi_D)_*(\cF\otimes L)[2] , \trunc\LL_{P(D)}) = 0.$ Hence we can also find a map $\phi'$ that fits into the diagram \eqref{Eq4.44}.

The long exact sequence associated to \eqref{Eq4.44} assures that the composition
\[\phi : (\bR(\pi_D)_*(\cF\otimes L)[-1])\dual \mapright{\phi'} \bL'_{P(D)/P(X)} \to \bL_{P(D)/P(X)}\]
is a perfect obstruction theory. Note that the associated virtual pullback
\[j^!_{\phi} : A_* (P(X)) \to A_* (P(D))\]
only depends on the vector bundle $\bR(\pi_D)_*(\cF\otimes L)$ and is independent of the map $\phi$ by \cite[Example 4.1.8]{Ful} (see Lemma \ref{Lem.VPindepofPhi}).  Therefore, the virtual pullback formula (Theorem \ref{Thm.Functoriality}) and the reduction formula \eqref{L32} gives us 
\[j^![P(X)]\virt_{OT}=j^!_{\phi}[P(X)]\virt_{OT}=[P(D)]\virt_{OT}=\sum_e(-1)^{\sigma(e)}[P(D)^e]\virt_{BF},\]
where $j^!$ is the Gysin pullback \eqref{L.GysinPullback} given by the tautological section.
\end{proof}

We end this section with a remark on orientations:

\begin{remark}\label{Rmk.Orientation}
For any given orientation on $\bR\hom_{\pi}(\II,\II)_0[3]$, there exists a unique orientation on $\bR\hom_{\pi_X}(d_*\cJ,d_*\cJ)_\#[3]$ induced from the reduction diagram \eqref{Eq.JJstoII0}. Then the signs $\sigma(e)$ in \eqref{L32} is uniquely determined by this orientation and the self-dual distinguished triangle \eqref{N.Eq.JJs}.
\end{remark}

\section{Pairs/Sheaves correspondence}\label{S.Pairs/Sheaves}

In this section, we compare the Oh-Thomas virtual cycles for moduli spaces of pairs and moduli spaces of sheaves by combining the virtual pullback formula in \S\ref{S.Functoriality} and the pushforward formula for virtual projective bundles in \S\ref{ss.VPBformula}.

\subsection{Virtual projective bundles}\label{ss.VPBformula}

In this preliminary subsection, we provide a pushforward formula for {\em virtual projective bundles}. This pushforward formula is just a reformulation of a classical result in intersection theory \cite{Ful} from the point of view of virtual intersection theory \cite{BeFa,Man}.


\begin{definition}[Virtual projective bundle]\label{Def.VirtualProjectiveBundle}
Let $\KK$ be a perfect complex of tor-amplitude $[0,1]$ on a quasi-projective scheme $\cX$. The {\em virtual projective bundle} $\PP(\KK)$ associated to $\KK$ is a pair of 
\begin{enumerate}
\item [D1)] a projective cone
\[p : \PP (\KK) := \mathrm{Proj} \Sym^\bullet h^0(\KK\dual) \to \cX,\]
\item [D2)] and a relative perfect obstruction theory
\beq\label{Eq.VPB.POT}
\LL\virt_{\PP(\KK)/\cX}:=\cone(\O_{\PP(\KK)} \to p^*\KK (1))\dual \to \LL_{\PP(\KK)/\cX}
\eeq
given as follows: Choose a resolution $\KK \cong [K_0 \to K_1]$ by vector bundles. Then we have a commutative diagram
\beq\label{Eq.VPB.Factorization}
\xymatrix{
& K_1(1)|_{\PP(K_0)} \ar[d]\\
\PP (\KK) \ar@{^{(}->}[r]^i \ar[rd]_p & \PP (K_0) \ar[d]^q \ar@/_0.5cm/[u]_{t} \\
& \cX
}\eeq
where $\PP(\KK)$ is the zero locus of the tautological section $t$. Thus
\[\xymatrix{
\LL_{\PP(\KK)/\cX}\virt \ar[d] \ar@{}[r]|{=} &(K_1(1))|_{\PP(\KK)}\dual \ar[r]^{dt} \ar[d]^{t} & \Omega_{\PP(K_0)/\cX}|_{\PP(\KK)} \ar@{=}[d]\\
\LL_{\PP(\KK)/\cX} \ar@{}[r]|{=}  & I/I^2 \ar[r]^d & \Omega_{\PP(K_0)/\cX}|_{\PP(\KK)}
}\]
gives us a perfect obstruction theory, where $I$ denotes the ideal sheaf of $\PP(\KK)$ in $\PP(K_0)$.
\end{enumerate}
\end{definition}
The perfect obstruction theory \eqref{Eq.VPB.POT} is independent of the choice of the global resolution and exists without assuming the global resolution, but we will not need these facts in this paper.\footnote{The virtual projective bundle is the classical truncation of the {\em derived} projective bundle.}

\begin{proposition}[Pushforward formula for virtual projective bundles]\label{Prop.VPB.Pushforward}
Let $p:\PP(\KK) \to \cX$ be a virtual projective bundle over a quasi-projective scheme $\cX$. For any cycle class $\alpha \in A_*(\cX)$ and a K-theory class $\xi \in K^0(\cX)$, we have
\[p_* (c_m(p^*\xi(1)) \cap p^! \alpha) = \sum_{0 \leq i \leq m} \binom{s-i}{m-i} \cdot c_i(\xi) \cap c_{m-i+1-r}(-\KK) \cap \alpha\]
where $r$ is the rank of $\KK$ and $s$ is the rank of $\xi$.
\end{proposition}

\begin{proof}
Fix a global resolution $\KK \cong [K_0 \to K_1]$ and consider the factorization \eqref{Eq.VPB.Factorization}. Manolache's virtual pullback formula $p^! = i^! \circ q^*$ implies
\[p_* (c_m(\xi(1)) \cap p^! \alpha)  = q_* ( c_m(\xi(1)) \cap c_{r_1}(K_1(1)) \cap q^*\alpha)\]
where $r_0$ and $r_1$ are the ranks of $K_0$ and $K_1$, respectively. Note that 
\[c_m( \xi(1) ) = \sum_{0 \leq i \leq m} \binom{s-i} {m-i} c_i(\xi) c_1(\O(1))^{m-i}\]
by \cite[Example 3.2.2]{Ful}.\footnote{In \cite{Ful}, it is stated for vector bundles, but it is easy to show that the result also holds for any $K$-theory classes.}
Therefore, we have
\begin{align*}
    & p_* (c_m(\xi(1)) \cap p^! \alpha)\\
    & = \sum_{0 \leq i \leq m} \sum_{0 \leq j \leq r_1} \binom{s-i}{m-i}\cdot c_i(\xi) \cap c_j(K_1) \cap q_*(c_1(\O(1))^{m+r_1 -i-j} \cap q^* \alpha)\\
    & = \sum_{0 \leq i \leq m} \sum_{0 \leq j \leq r_1} \binom{s-i}{m-i}\cdot c_i(\xi) \cap c_j(K_1)\cap s_{m+r_1 - i - j -r_0 + 1}(K_0) \cap \alpha \\
    & = \sum_{0 \leq i \leq m} \binom{s-i}{m-i}\cdot c_i(\xi) \cap c_{m - i + 1-r}(-\KK) \cap \alpha
\end{align*}
where $s_\bullet(K_0)$ denotes the Segre class of $K_0$.
\end{proof}



\subsection{Main result}

Let $X$ be a smooth projective Calabi-Yau 4-fold. Let $\beta \in H_2(X,\Q)$ be a curve class and $n \in \Z$ be an integer. Let
\begin{align*}
P_{n,\beta}(X) &= \{\text{stable pairs }(F,s) \text{ on } X \text{ with } \ch(F) = (0,0,0,\beta,n)\}\\
M_{n,\beta}(X) &= \{\text{stable sheaves }G \text{ on } X \text{ with } \ch(G) = (0,0,0,\beta,n)\}
\end{align*}
be the moduli spaces of stable pairs and stable sheaves. Then they both have Oh-Thomas virtual cycles
\[[P_{n,\beta}(X)]\virt \in A_n (P_{n,\beta}(X)) \and [M_{n,\beta}(X)]\virt \in A_1 (M_{n,\beta}(X)).\]

Assume that the curve class $\beta \in H_2(X,\Q)$ is irreducible.
Then all pure sheaves $G$ with $\ch_3(G)=\beta$ are stable so that the moduli space $M_{n,\beta}(X)$ is proper and does not depend on the polarization $\O_X(1)$. Moreover, there is a well-defined forgetful map
\[p : P_{n,\beta}(X) \to M_{n,\beta}(X) : (F,s) \mapsto F.\]
When there is a universal sheaf $\GG$ on $M_{n,\beta}(X)\times X$, then the moduli space $P_{n,\beta}(X)$ of stable pairs can be identified to the virtual projective bundle
\[P_{n,\beta}(X) = \PP (\bR\pi^M_*\GG) \to M_{n,\beta}(X) \]
of the perfect complex $\bR\pi^M_*\GG$, where $\pi^M : M_{n,\beta}(X) \times X \to M_{n,\beta}(X)$ denotes the projection map.



\begin{theorem}[Pairs/Sheaves correspondence]\label{Thm.P/S}
Let $X$ be a Calabi-Yau 4-fold, $\beta \in H_2(X,\Q)$ be an irreducible curve class, and $n \in \Z$ be an integer. Assume that there exists a universal sheaf $\GG$ on $M_{n,\beta}(X)$. For any orientation on $M_{n,\beta}(X)$, there exists an induced orientation on $P_{n,\beta}(X)$ such that the following formulas hold:
\begin{enumerate}
\item {\em (Pullback formula)} We have 
\begin{equation}\label{Eq.P/S.Pullback}
[P_{n,\beta}(X)]\virt = p^![M_{n,\beta}(X)]\virt    
\end{equation}
where $p^!$ is the virtual pullback of the virtual projective bundle.
\item {\em (Pushforward formula)} For any vector bundle $E$ on $X$, we have 
\begin{align}\label{Eq.P/S.Pushforward}
& p_*\left(c_m(\bR\pi^P_*(\FF\otimes E)) \cap [P_{n,\beta}(X)]\virt\right) \\ 
= & \nonumber \begin{cases}
\binom{N }{n-1} \cdot [M_{n,\beta}]\virt    & \text{if } m=n-1 \\
\left(\begin{array}{l}
\binom{N -1}{n-1} \cdot c_1(\bR\pi^M_*(\GG\otimes E)) \\ 
- \binom{N}{n}\cdot c_1(\bR\pi^M_*\GG)
\end{array}\right) \cap  [M_{n,\beta}(X)]\virt    & \text{if } m=n\\
0 & \text{otherwise}
\end{cases}
\end{align}
where $\pi^P : P_{n,\beta}(X) \times X \to P_{n,\beta}(X)$ is the projection map, $r$ is the rank of $E$, and $N=rn + \int_{\beta} c_1(E)$ is the rank of the tautological complex $\bR\pi^P_*(\FF\otimes E)$.
\end{enumerate}
\end{theorem}


We will prove the pullback formula \eqref{Eq.P/S.Pullback} through the virtual pullback formula in Theorem \ref{Thm.Functoriality} by comparing the symmetric obstruction theories of $P_{n,\beta}(X)$ and $M_{n,\beta}(X)$. Then the pushforward formula \eqref{Eq.P/S.Pushforward} will follow directly from the general pushforward formula for virtual projective bundles in Proposition \ref{Prop.VPB.Pushforward}.

Before we prove our main theorem in this section (Theorem \ref{Thm.P/S}), we first provide immediate corollaries. Recall \cite{CMT,CMT2} that the {\em primary invariants} for a cohomology class $\gamma \in H^4(X,\Q)$ are defined as follows:
\begin{align*}
P_{n,\beta}(\gamma) &:= \int_{[P_{n,\beta}(X)]\virt}   (\pi^P _* ( \ch_3(\FF) \cup \gamma))^n\\ M_{n,\beta}(\gamma) &:= \int_{[M_{n,\beta}(X)]\virt} \pi^M _* ( \ch_3(\GG) \cup \gamma)     
\end{align*}
where the primary invariant $M_{n,\beta}(\gamma)$ does not depend on the choice of the universal family $\GG$.

\begin{corollary}[Primary PT/GV correspondence]
Let $X$ be a Calabi-Yau 4-fold and $\beta \in H_2(X,\Q)$ be an irreducible curve class. Assume that $M_{n,\beta}(X)$ has a universal family. Then there exists a choice of orientations such that
\beq\label{Eq.PT/Katz}
P_{n,\beta}(\gamma) = 
\begin{cases}
M_{1,\beta}(\gamma) & \text{if } n=1\\
0 & \text{if }n\geq 2
\end{cases}
\eeq
for any $\gamma \in H^4(X,\Q)$.
\end{corollary}

\begin{proof}
By the pushforward formula \eqref{Eq.P/S.Pushforward}, we have
\[p_*[P_{n,\beta}(X)]\virt = \begin{cases} [M_{n,\beta}(X)]\virt & \text{ if } n=1\\ 0 & \text{ if } n \geq 2 \,. \end{cases}\]
Since we have
\[\ch_3(\FF) = \ch_3((p\times 1)^*\GG \otimes \O(1)) = \ch_3((p\times 1)^*\GG) = (p\times 1)^*\ch_3(\GG),\]
the projection formula proves \eqref{Eq.PT/Katz}.
\end{proof}

We can define the {\em tautological invariants} (cf. \cite{CKp,CTt}) for a vector bundle $E$ on $X$ as follows:
\begin{align*}
P_{n,\beta}(E) & := \int_{[P_{n,\beta}(X)]\virt} c_n(\bR\pi_*(\FF\otimes E))\\  
M_{n,\beta}(E) & := \int_{[M_{n,\beta}(X)]\virt} c_1(\bR\pi_*(\GG\otimes E))
\end{align*}
where the tautological invariant $M_{n,\beta}(E)$ {\em depends} on the choice of the universal family $\GG$.

\begin{corollary}[Tautological PT/GV correspondence]
Let $X$ be a Calabi-Yau 4-fold and $\beta$ be an irreducible curve class. Assume that there is a universal family $\GG$ of $M_{n,\beta}(X)$. For any vector bundle $E$, there exists a choice of orientations such that
\[P_{n,\beta}(E) = 
\begin{cases}
-\binom{N}{n} \cdot M_{n,\beta}(\O_X) & \text{if } n=0 \\
\binom{N-1}{n-1}\cdot M_{n,\beta}(E) - \binom{N}{n}\cdot M_{n,\beta}(\O_X) & \text{if }n\geq1
\end{cases}\]
where $N=n \cdot \rank(E) + \int_{\beta} c_1(E)$ is the rank of $\bR\pi^P_*(\FF\otimes E)$. 
\end{corollary}



We can generalize Theorem \ref{Thm.P/S} to {\em reducible} curve classes as follows: When $M_{n,\beta}(X)$ has no strictly semi-stable sheaves, the moduli space of Joyce-Song type stable pairs in \cite{CTd} (cf. \cite{Pot,JS}) can be expressed as
\[P^{JS}_{n,\beta}(X) = \{(F,s) : F \in M_{n,\beta}(X) \and s:\O_X \to F \text{ is non-zero}\}.\]
By \cite[Theorem 0.1]{CTd}, $P^{JS}_{n,\beta}(X)$ is an open subscheme of the moduli space $\Perf(X)_{\O_X}^{\mathrm{spl}}$ of simple perfect complexes with fixed trivial determinant. Hence $P^{JS}_{n,\beta}(X)$ carries an Oh-Thomas virtual cycle
$[P^{JS}_{n,\beta}(X)]\virt \in A_n (P^{JS}_{n,\beta}(X))$
by \cite{OT}. For any curve class $\beta \in H_2(X,\Q)$, there is a forgetful map
\[ p : P^{JS}_{n,\beta}(X) \to M_{n,\beta}(X) : (F,s) \mapsto F\]
which identifies $P^{JS}_{n,\beta}(X)$ with the virtual projective bundle of $\bR\pi^M _*\GG$. The proof of Theorem \ref{Thm.P/S} also works for $P^{JS}_{n,\beta}(X)$ and gives us analogous pullback formula \eqref{Eq.P/S.Pullback} and the pushforward formula \eqref{Eq.P/S.Pushforward}. In particular, we have the following corollary:

\begin{corollary}[JS/GV correspondence]
Let $X$ be a Calabi-Yau 4-fold, $\beta \in H_2(X,\Q)$ be a curve class, and $n$ be an integer. Assume that $\int_{\beta}c_1(\O_X(1))$ and $n$ are coprime. Then there exists a choice of orientations such that
\begin{align*}
P^{JS}_{n,\beta}(\gamma) &=
\begin{cases}
M_{1,\beta}(\gamma) & \text{if } n=1\\
0 & \text{if }n\geq 2
\end{cases}\\
P^{JS}_{n,\beta}(E) &= 
\begin{cases}
- \binom{N}{n}\cdot M_{0,\beta}(\O_X) & \text{if } n=0 \\
\binom{N-1}{n-1} \cdot M_{n,\beta}(E) - \binom{N}{n}\cdot M_{n,\beta}(\O_X) & \text{if }n\geq1
\end{cases}    
\end{align*}
for any cohomology class $\gamma$ and a vector bundle $E$, where the primary invariant $P^{JS}_{n,\beta}(\gamma)$ and the tautological invariant $P^{JS}_{n,\beta}(E)$ are defined analogously.
\end{corollary}



The rest of this section is devoted to the proof of Theorem \ref{Thm.P/S}(1).

\subsection{Comparison of obstruction theories}

Given $X$, $\beta$, $n$ as in Theorem \ref{Thm.P/S}, consider the cartesian diagram
\[\xymatrix{
P_{n,\beta}(X) \times X \ar[r]^{p\times 1} \ar[d]^{\pi^P} & M_{n,\beta}(X) \times X \ar[d]^{\pi^M} \\
P_{n,\beta}(X) \ar[r]^p & M_{n,\beta}(X)
}\]
of schemes. Recall that the Oh-Thomas virtual cycles on $P_{n,\beta}(X)$ and $M_{n,\beta}(X)$ are constructed from the symmetric obstruction theories
\begin{align}
\phi_P &: \bR\hom_{\pi^P}(\II,\II)_0[3] \xrightarrow{At(\II)} \LL_{P_{n,\beta}(X)} \label{Eq.phiP}\\
\phi_M &: (\tau^{[1,3]}\bR\hom_{\pi^M}(\GG,\GG))[3] \xrightarrow{At(\GG)} \trunc \LL_{M_{n,\beta}(X)}\label{Eq.phiM}
\end{align}
induced by the Atiyah classes of the universal complex $\II=[\O_{P_{n,\beta}(X)\times X} \to \FF]$ on $P_{n,\beta}(X)\times X$ and the universal sheaf $\GG$ on $M_{n,\beta}(X) \times X$.\footnote{Here the Atiyah class $At(\GG)$ depends on the choice of the universal family $\GG$, but the symmetric obstruction theory $\phi_M$ does not depend of the choice of $\GG$. See Remark \ref{Rem.Atiyah}.} Note that the universal sheaf $\FF$ on $P_{n,\beta}(X)\times X$ can be written as
\[\FF = (p\times 1)^*\GG \otimes (\pi^P)^*\O(1)\]
where $\O(1)$ is the dual of the tautological line bundle of the virtual projective bundle $P_{n,\beta}(X)=\PP(\bR\pi^M _* \GG)$.

\begin{proposition}\label{Prop.P/S.compatibility}
The two symmetric obstruction theories $\phi_P$ and $\phi_M$ in \eqref{Eq.phiP} and \eqref{Eq.phiM} are compatible as follows:
\begin{enumerate}
\item We have a reduction diagram (see Proposition \ref{Prop.Reduction}), i.e., a morphism of distinguished triangles
\beq\label{Eq.P/S.reduction}
\xymatrix{
\bR\hom_{\pi^P}(\II,\FF)_\#[2] \ar[r]^-{\alpha\dual} \ar[d]^{\beta\dual} & \bR\hom_{\pi^P}(\II,\II)_0[3] \ar[r]^-{\delta} \ar[d]^{\alpha} & \bR\hom_{\pi^P}(\FF,\O_{})_\#[4] \ar@{=}[d]\\
(\tau^{[1,3]}\bR\hom_{\pi^P}(\FF,\FF))[3] \ar[r]^-{\beta} & \bR\hom_{\pi^P}(\FF,\II)_\#[4] \ar[r]^-{\gamma} & \bR\hom_{\pi^P}(\FF,\O_{})_\#[4]
}\eeq
for some maps $\alpha$, $\beta$, $\gamma$, $\delta$ and perfect complexes $\bR\hom_{\pi^P}(\FF,\sO)_\#$, $\bR\hom_{\pi^P}(\FF,\II)_\#$, and $\bR\hom_{\pi^P}(\II,\FF)_\#:=(\bR\hom_{\pi^P}(\FF,\II)_\#)\dual[-4]$.
\item The Atiyah classes of $\II$ and $\GG$ are compatible, i.e., the diagram
\beq\label{Eq.P/S.Atiyah}
\xymatrix{
\bR\hom_{\pi^P}(\II,\FF)_\#[2] \ar[r]^{\alpha\dual} \ar[d]^{\beta\dual} & \bR\hom_{\pi^P}(\II,\II)_0[3] \ar[dd]^{At(\II)}\\
(\tau^{[1,3]}\bR\hom_{\pi^P}(\FF,\FF))[3] \ar[d]^{At(\GG)} \ar[rd]|{At(\FF)}& \\
\trunc\LL_{M_{n,\beta}(X)}|_{P_{n,\beta}(X)} \ar[r] & \trunc\LL_{P_{n,\beta}(X)}
}\eeq
\end{enumerate}
commutes. 
\end{proposition}

We will prove Proposition \ref{Prop.P/S.compatibility} through several steps. The proof is similar to that of Proposition \ref{Prop.L.Comp.OT} in \S\ref{ss.ComparisonofOT}.  To simplify the notations, let
\[\sP= P_{n,\beta}(X), \quad \sX = P_{n,\beta}(X) \times X, \quad \pi=\pi^P, \and \sM=M_{n,\beta}(X)\]
be the abbreviations of the schemes. Let
\beq\label{T.IOF}
\xymatrix{
\II \ar[r]^i & \O_{\sX} \ar[r]^s & \FF \ar[r]^e & \II[1]
}\eeq
be the canonical distinguished triangle on $\sX$.

\subsubsection{Prototype} 

We first consider a {\em prototype} of the reduction diagram \eqref{Eq.P/S.reduction} where we use
\[\bR\hom_{\pi}(\FF,\FF), \quad \bR\hom_{\pi}(\FF,\O_{\sX}) \and \bR\hom_{\pi}(\FF,\II)\]
instead of $\tau^{[1,3]}\bR\hom_{\pi}(\FF,\FF)$, $\bR\hom_{\pi}(\FF,\O_{\sX})_\#$, and $\bR\hom_{\pi}(\FF,\II)_\#$.

\begin{lemma}\label{Lem.P/S.prototype}
We have a (not necessarily commutative) diagram of two distinguished triangles
\beq\label{Eq.P/S.prototype}
\xymatrix@C+.6pc{
\bR\hom_{\sX}(\II,\FF)[2] \ar[r]^-{\rho \circ l_e} \ar[d]^{r_e} & \bR\hom_{\sX}(\II,\II)_0[3] \ar[r]^{r_e \circ l_i \circ \iota} \ar[d]^-{r_e \circ \iota} & \bR\hom_{\sX}(\FF,\O_{\sX})[4] \ar@{=}[d]  \ar[r]^-{\eta} &\\
\bR\hom_{\sX}(\FF,\FF)[3] \ar[r]^-{l_e} & \bR\hom_{\sX}(\FF,\II)[4] \ar[r]^-{l_i} & \bR\hom_{\sX}(\FF,\O_{\sX})[4] \ar[r]^-{l_s} &
}\eeq
for some map $\eta$ such that $l_e \circ \eta = \id_{\II} \circ r_s$. Here $r_i$, $r_s$, $r_e$ (resp. $l_i$, $l_s$, $l_e$) are the right (resp. left) composition maps with the maps $i$, $s$, $e$ in \eqref{T.IOF}, and $\iota$, $\rho$ are the canonical maps in the direct sum diagram 
\beq\label{D.II}
\xymatrix{\bR\hom_{\sX}(\II,\II)_0 \ar@<.5ex>[r]^-{\iota} & \bR\hom_{\sX}(\II,\II) \ar@<.5ex>[r]^-{tr} \ar@<.5ex>[l]^-{\rho} & \O_{\sX} \ar@<.5ex>[l]^-{\id_{\II}} } 
\eeq 
where $tr$ denotes the trace map.
\end{lemma}

\begin{proof}
Applying $\bR\hom_{\sX}(\FF,-)$ to \eqref{T.IOF}, we obtain the lower distinguished triangle in \eqref{Eq.P/S.prototype}. Applying the octahedral axiom to the commutative diagram
\[\xymatrix{
\O_{\sX} \ar@{=}[r] \ar[d]^{\id_{\II}} & \bR\hom_{\sX}(\O_{\sX},\O_{\sX}) \ar[d]^{r_i} \\
\bR\hom_{\sX}(\II,\II) \ar[r]^{l_i} & \bR\hom_{\sX}(\II,\O_{\sX}) 
}\]
with the obvious three distinguished triangles given by \eqref{T.IOF} and \eqref{D.II}, we also obtain the upper distinguished triangle in \eqref{Eq.P/S.prototype}.
\end{proof}

Clearly, the middle square in \eqref{Eq.P/S.prototype} commutes. However, the left square in \eqref{Eq.P/S.prototype} may not commute.

\subsubsection{}

We will replace the complexes in \eqref{Eq.P/S.prototype} as
\begin{align*}
\bR\hom_{\pi}(\FF,\FF) \leadsto& \tau^{[1,3]}\bR\hom_{\pi}(\FF,\FF) \\
& = \cone(\cone(\O_{\sP} \mapright{} \bR\hom_{\pi}(\FF,\FF)) \mapright{} \O_{\sP}[-4])[-1] \\
\bR\hom_{\pi}(\FF,\O_{\sX}) \leadsto & \bR\hom_{\pi}(\FF,\O_{\sX})_\# \\
& = \cone (  \bR\hom_{\pi}(\FF,\O_{\sX}) \mapright{} \O_{\sP}[-4])[-1] \\
\bR\hom_{\pi}(\FF,\II)  \leadsto & \bR\hom_{\pi}(\FF,\II)_\#\\
& = \cone (\O_{\sP}[-1] \mapright{} \bR\hom_{\pi}(\FF,\II)[4])
\end{align*}
to obtain the commutative diagram \eqref{Eq.P/S.reduction}. Here we fix some notations among them.
\begin{enumerate}
\item We can form a commutative diagram
\beq\label{R.FFs}
\xymatrix{
\O_{\sP} \ar@{=}[r] \ar[d]^{t\dual} & \O_{\sP} \ar[d]^{\id_{\FF}} & \\
\tau^{\leq3}\bR\hom_{\pi}(\FF,\FF) \ar[r]^{\mu\dual} \ar[d]^{\nu\dual} & \bR\hom_{\pi}(\FF,\FF) \ar[r]^>>>>>{tr^4} \ar[d]^{\mu} & \O_{\sP}[-4] \ar@{=}[d]\\
\tau^{[1,3]}\bR\hom_{\pi}(\FF,\FF) \ar[r]^{\nu} & \tau^{\geq1}\bR\hom_{\pi}(\FF,\FF) \ar[r]^>>>>>{t}  & \O_{\sP}[-4]
}\eeq
of four distinguished triangles for unique maps $\mu$, $\nu$, and $t$ (cf. Lemma \ref{Lem.Gen.Reduction}). Here $tr^4$ denotes the top trace map.
\item We can also form the following two distinguished triangles
\beq\label{T.FOs}
\xymatrix@+1pc{
\bR\hom_{\pi}(\FF,\O_{\sX})_\# \ar[r]^-{\kappa} & \bR\hom_{\pi}(\FF,\sO_{\sX}) \ar[r]^-{tr^4 \circ l_s} & \O_{\sP}[-4]
}\eeq
\beq\label{T.FIs}
\xymatrix@+1pc{
\O_{\sP}[-1] \ar[r]^-{l_e\circ \id_{\FF}} & \bR\hom_{\pi}(\FF,\II) \ar[r]^-{\lambda} & \bR\hom_{\pi}(\FF,\II)_\#
} \eeq
over $\sP$. Let $\bR\hom_{\pi}(\II,\FF)_\# := (\bR\hom_{\pi}(\FF,\II))_\#\dual[-4].$
\end{enumerate}

\subsubsection{Maps $\alpha$, $\beta$, $\gamma$, $\delta$} 

We now construct the maps $\alpha$, $\beta$, $\gamma$, $\delta$ and the distinguished triangles in \eqref{Eq.P/S.reduction} as follows:

\begin{lemma}\label{Lem.PS.Oct}
We have the following diagrams:
\begin{enumerate}
\item There is a commutative diagram of four distinguished triangles
\beq\label{Oct.zeta}
\xymatrix{
\O_{\sP}[3] \ar@{=}[r] \ar[d]^{\id_{\FF}} & \O_{\sP}[3] \ar[d]^{l_e \circ \id_{\FF}} & \\
\bR\hom_{\pi}(\FF,\FF)[3] \ar[r]^{l_e} \ar[d]^{\mu} & \bR\hom_{\pi}(\FF,\II)[4] \ar[r]^{l_i} \ar[d]^{\lambda} & \bR\hom_{\pi}(\FF,\O_{\sX}) [4] \ar@{=}[d]\\
(\tau^{\geq1}\bR\hom_{\pi}(\FF,\FF))[3] \ar@{.>}[r]^{\zeta} & \bR\hom_{\pi}(\FF,\II)_\#[4] \ar@{.>}[r]^{\omega} & \bR\hom_{\pi}(\FF,\O_{\sX})[4]
}\eeq
for some maps $\zeta$ and $\omega$. Here the three given distinguished triangles are those in \eqref{R.FFs}, \eqref{T.FIs}, and \eqref{T.IOF}.
\item There is a commutative diagram of four distinguished triangles
\beq\label{Oct.betagamma}
\xymatrix{
 & \bR\hom_{\pi}(\FF,\O_{\sX})[3] \ar@{=}[r] \ar[d]^{\mu\circ l_s} & \bR\hom_{\pi}(\FF,\O_{\sX})[3] \ar[d]^{tr^4\circ l_s} \\
(\tau^{[1,3]}\bR\hom_{\pi}(\FF,\FF))[3] \ar@{=}[d] \ar[r]^{\nu} & (\tau^{\geq1}\bR\hom_{\pi}(\FF,\FF))[3]  \ar[r]^{t} \ar[d]^{\zeta} & \O_{\sP}[-1]  \ar[d]\\
(\tau^{[1,3]}\bR\hom_{\pi}(\FF,\FF))[3] \ar@{.>}[r]^{\beta} & \bR\hom_{\pi}(\FF,\II)_\#[4] \ar@{.>}[r]^{\gamma} &  \bR\hom_{\pi}(\FF,\O_{\sX})_\#[4]
}\eeq
for some maps $\beta$ and $\gamma$ such that $\kappa \circ \gamma = \omega$. Here the three given distinguished triangles are those in \eqref{Oct.zeta}, \eqref{T.FOs}, and \eqref{R.FFs}.
\item There is a commutative diagram of four distinguished triangles
\beq\label{Oct.alphadelta}
\xymatrix{
& \bR\hom_{\pi}(\FF,\O_{\sX})[3] \ar[d]^{\eta} \ar@{=}[r] & \bR\hom_{\pi}(\FF,\O_{\sX})[3] \ar[d]^{tr^4\circ l_s}\\
\bR\hom_{\pi}(\II,\FF)_\#[2] \ar[r]^{\lambda\dual} \ar@{=}[d] & \bR\hom_{\pi}(\II,\FF)[2] \ar[r]^{tr^4\circ r_e} \ar[d]^{\rho\circ l_e} & \O_{\sP}[-1] \ar[d]\\
\bR\hom_{\pi}(\II,\FF)_\#[2] \ar@{.>}[r]^{\alpha\dual} & \bR\hom_{\pi}(\II,\II)_0[3] \ar@{.>}[r]^{\delta} & \bR\hom_{\pi}(\FF,\O_{\sX})_\#[4]
}\eeq
for some maps $\alpha$ and $\delta$ such that $\kappa \circ \delta =r_e \circ l_i \circ \iota$. Here the three given distinguished triangles are those in \eqref{Eq.P/S.prototype}, \eqref{T.FOs}, and \eqref{T.FIs}.
\end{enumerate}
\end{lemma}

\begin{proof}
(1) and (2) follow directly from the octahedral axiom. We will only prove (3). It suffices to show that the top right square in \eqref{Oct.alphadelta} commutes. Consider a diagram
\beq\label{PS2}
\xymatrix{
\bR\hom_{\sX}(\FF,\O_{\sX}) \ar[r]^-{\eta} \ar@/^0.8cm/[rr]^{l_s} \ar[d]^{r_s} & \bR\hom_{\sX}(\II,\FF) \ar[r]^-{r_e} \ar[d]^{l_e} & \bR\hom_{\sX}(\FF,\FF) \ar[d]^{tr} \\
\O_{\sX} \ar[r]^-{\id_{\II}} & \bR\hom_{\sX}(\II,\II) \ar[r]^-{tr} & \O_{\sX}
}\eeq
of commutative squares where the left square in \eqref{PS2} commutes by Lemma \ref{Lem.P/S.prototype}. Then we have
\[tr \circ r_e \circ \eta = tr \circ l_e \circ \eta = \tr \circ \id_{\II} \circ r_s = r_s = tr \circ l_s.\]
Applying $\bR\pi_*$ to \eqref{PS2}, we obtain the desired commutative diagram.
\end{proof}

\subsubsection{Reduction diagram}

Now we can form the reduction diagram \eqref{Eq.P/S.reduction}.

\begin{proof}[Proof of Proposition \ref{Prop.P/S.compatibility}(1)]
We first prove the commutativity of the left square in \eqref{Eq.P/S.reduction}. Indeed, consider the commutative diagram
\[\xymatrix{
\bR\hom_{\pi}(\II,\FF)_\#[2] \ar[r]^{\lambda\dual} \ar[d]^{\zeta\dual} & \bR\hom_{\pi}(\II,\FF)[2] \ar[r]^{l_e} \ar[d]^{r_e} & \bR\hom_{\pi}(\II,\II)[3] \ar[d]^{r_e} \\
(\tau^{\leq3}\bR\hom_{\pi}(\FF,\FF))[3] \ar[r]^{\mu\dual} \ar[d]^{\nu\dual} & \bR\hom_{\pi}(\FF,\FF)[3] \ar[r]^{l_e} \ar[d]^{\mu} & \bR\hom_{\pi}(\FF,\II)[4] \ar[d]^{\lambda}\\
(\tau^{[1,3]}\bR\hom_{\pi}(\FF,\FF))[3] \ar[r]^{\nu} & \tau^{\geq1}\bR\hom_{\pi}(\FF,\FF))[3] \ar[r]^{\zeta} & \bR\hom_{\pi}(\FF,\II)_\#[4]
}\]
given by \eqref{R.FFs} and \eqref{Oct.zeta}. Note that $\alpha$ and $\beta$ are defined as the compositions
\begin{align*}
\alpha &= \lambda  \circ r_e \circ \iota : \bR\hom_{\pi}(\II,\II)_0[3] \to \bR\hom_{\pi}(\FF,\II)_\#[4]\\
\beta &= \zeta \circ \nu : (\tau^{[1,3]} \bR\hom_{\pi}(\FF,\FF))[3] \to \bR\hom_{\pi}(\FF,\II)_\#[4]
\end{align*}
in \eqref{Oct.alphadelta} and \eqref{Oct.betagamma}. Thus it suffices to show
\beq\label{PS4}
\lambda \circ r_e \circ \iota \circ \rho \circ l_e \circ \lambda\dual =\lambda \circ r_e \circ l_e \circ \lambda\dual
\eeq
to prove $\alpha \circ \alpha\dual = \beta\circ \beta\dual$. By the dual diagram of \eqref{Oct.zeta}, we have
\beq\label{PS11}
tr^4\circ l_e\circ\lambda\dual = tr^4 \circ r_e \circ \lambda\dual =0.
\eeq
Since the pairing on $\bR\hom_{\pi}(\II,\II)[3]$ is given by the composition 
\[\bR\hom_{\pi}(\II,\II)[3] \otimes \bR\hom_{\pi}(\II,\II)[3] \mapright{\cup} \bR\hom_{\pi}(\II,\II)[6] \mapright{tr^4} \O_{\sP}[2],\]
the equation \eqref{PS11} proves \eqref{PS4}.


We then prove the commutativity of the right square in \eqref{Eq.P/S.reduction}. Since 
\[\Hom_{\sP}(\bR\hom_{\pi}(\II,\II)_0[3], \O_{\sP}[-1]) = 0,\]
it suffices to show that the square in the diagram
\[\xymatrix{
\bR\hom_{\pi}(\II,\II)_0[3] \ar[r]^{\delta} \ar[d]^{\alpha=\lambda\circ r_e \circ \iota} \ar@/^0.8cm/[rr]^{r_e \circ l_i \circ \iota} & \bR\hom_{\pi}(\FF,\O_{\sX})_\#[4] \ar[r]^{\kappa} & \bR\hom_{\pi}(\FF,\O_{\sX})[4] \ar@{=}[d] \\
\bR\hom_{\pi}(\FF,\II)_\#[4] \ar[r]^{\gamma}  \ar@/_0.8cm/[rr]_{\omega}& \bR\hom_{\pi}(\FF,\O_{\sX})_\#[4] \ar[r]^{\kappa} & \bR\hom_{\pi}(\FF,\O_{\sX})[4]
}\]
commutes. By Lemma \ref{Lem.PS.Oct}(2), Lemma \ref{Lem.PS.Oct}(3), and equation $\omega \circ \lambda =l_i$ in \eqref{Oct.zeta}, we have
\[\kappa \circ \gamma \circ \alpha = \omega \circ \lambda \circ r_e \circ \iota = l_i \circ r_e \circ \iota = r_e \circ l_i \circ \iota =\kappa \circ \delta,\]
which proves the claim.

By Lemma \ref{Lem.R}, the commutativity of the two squares in \eqref{Eq.P/S.reduction} suffices to form the desired reduction diagram. It completes the proof.
\end{proof}

\subsubsection{Atiyah classes}

Finally, we compare the Atiyah classes.

\begin{proof}[Proof of Proposition \ref{Prop.P/S.compatibility}(2)]

We first compare the Atiyah classes of $\II$ and $\FF$. The functoriality of Atiyah classes
\[\xymatrix{
\FF \ar[r]^e \ar[d]^{At(\FF)} & \II[1] \ar[d]^{At(\II)}\\
\FF\otimes \LL_{\sX} [1] \ar[r]^{e\otimes 1} & \II \otimes \LL_{\sX}[2]
}\]
gives us a commutative diagram
\beq\label{PS.A.1}
\xymatrix{
\bR\hom_{\pi}(\II,\FF)[2] \ar[r]^{l_e} \ar[d]^{r_e} & \bR\hom_{\pi}(\II,\II)[3] \ar[d]^{At(\II)} \\
\bR\hom_{\pi}(\FF,\FF)[3] \ar[r]^-{At(\FF)} & \LL_{\sP}\,.
}\eeq
Since $\lambda\circ l_e = \zeta\circ \mu$ by \eqref{Oct.zeta}, $\alpha\dual = \rho \circ l_e \circ \lambda\dual$ by \eqref{Oct.alphadelta}, and the composition
\[\bR\pi_*\O_{\sX}[3] \xrightarrow{\id_{\II}} \bR\hom_{\pi}(\II,\II)[3] \xrightarrow{At(\II)} \LL_{\sP}\] 
is zero by \cite[Proposition 3.2]{STV}, the commutative diagram \eqref{PS.A.1} induces a commutative diagram
\beq\label{PS.A.2}
\xymatrix{
\bR\hom_{\pi}(\II,\FF)_\#[2] \ar[r]^{\alpha\dual} \ar[d]^{\zeta\dual} & \bR\hom_{\pi}(\II,\II)_0[3] \ar[d]^{At(\II)} \\
(\tau^{\leq3}\bR\hom_{\pi}(\FF,\FF))[3] \ar[r]^>>>>>>>>>>{At(\FF)} & \LL_{\sP}\,.
}\eeq
Since $\Hom_{\sP}(\O_{\sP}[4], \trunc \LL_{\sP})=0$, the commutative diagram \eqref{PS.A.2} gives us the commutativity of the upper square in \eqref{Eq.P/S.Atiyah}.

We then compare the Atiyah classes of 
\[(p\times 1)^*\GG \and \FF = (p\times 1)^*\GG \otimes \pi^*\O_{\sP}(1).\] 
By \cite[p. 260]{HL}\footnote{In \cite{HL}, the formula \eqref{PS.A.3} is proved for Atiyah classes on smooth schemes, but \cite[Chapitre IV, 2.3.7]{Illusie} shows that it also holds for Atiyah classes on arbitrary schemes.}, we have
\beq\label{PS.A.3}
At(\FF) = At((p\times 1)^*\GG)\otimes \id_{\pi^*\O_{\sP}(1)} + \id_{(p\times1)^*\GG}\otimes At(\pi^*\O_{\sP}(1)) : \FF \to \FF\otimes \LL_{\sX} [1].
\eeq
Also, we can deduce
\[At(\pi^*\O_{\sP}(1)) = (\pi^*At(\O_{\sP}(1)),0) : \O_{\sX} \to \LL_{\sX}[1] = \pi^*\LL_{\sP}[1] \oplus (\pi^X)^*\LL_{X}[1]\]
from the commutative diagram
\[\xymatrix{ \sX \ar[r]^{\pi} \ar@/_.4cm/[rr]_-{\pi^*\O_{\sP}(1)} & \sP \ar[r]^-{\O_{\sP}(1)} & \mathbf{Perf} }\]
by \cite[Appendix A]{STV}, where $\mathbf{Perf}$ denotes the derived moduli stack of perfect complexes and $\pi^X : \sP \times X \to X$ denotes the projection map. 
Hence the difference of the two induced maps
\beq\label{PS.A.4}
\xymatrix{
\bR\hom_{\pi}(\FF,\FF)[3] \ar@<.4ex>[rr]^-{At(\FF)}  \ar@<-.4ex>[rr]_-{At((p\times1)^*\GG)} && \LL_{\sP}
}\eeq
is the composition
\[\bR\hom_{\pi}(\FF,\FF)[3] \mapright{tr^4} \O_{\sP}[-1] \xrightarrow{At(\O_{\sP}(1))} \LL_{\sP}.\]
By \eqref{R.FFs}, the compositions of the two maps in \eqref{PS.A.4} with the map
\[\mu\dual : (\tau^{\leq3}\bR\hom_{\pi}(\FF,\FF))[3] \to \bR\hom_{\pi}(\FF,\FF)[3]\]
coincide. Since $\Hom_{\sP}(\O_{\sP}[4], \trunc \LL_{\sP})=0$, the two maps
\beq\label{PS.A.5}\xymatrix{
(\tau^{[1,3]}\bR\hom_{\pi}(\FF,\FF))[3] \ar@<.4ex>[rr]^-{At(\FF)}  \ar@<-.4ex>[rr]_-{At((p\times1)^*\GG)} && \trunc \LL_{\sP}
}\eeq
given by the Atiyah classes of $\FF$ and $(p\times 1)^*\GG$ are equal.

Finally, we compare the Atiyah classes of $\GG$ and $(p \times 1)^*\GG$. From the commutative diagram
\[\xymatrix{ \sP \ar[r]^-p \ar@/_.5cm/[rr]_-{(p\times1)^*\GG} & \sM \ar[r]^-{\GG} & \mathbf{Perf}(X) }\]
we can deduce that the triangle
\beq\label{PS.A.6}
\xymatrix{
p^*\bR\hom_{\pi}(\GG,\GG)[3] \ar[rd]^-{At((p\times1)^*\GG)} \ar[d]_{At(\GG)}\\
p^*\LL_{\sM} \ar[r] & \LL_{\sP}
}\eeq
commutes by \cite[Appendix A]{STV}.

Combining the commutative triangle \eqref{PS.A.6} with the equality \eqref{PS.A.5}, we deduce the commutativity of the lower triangle in \eqref{Eq.P/S.Atiyah}.
\end{proof}

\begin{remark}\label{Rem.Atiyah}
By the arguments in the second paragraph of the proof of Proposition \ref{Prop.P/S.compatibility}(2), we can deduce that the symmetric obstruction theory $\phi_M$ in \eqref{Eq.phiM} is independent of the choice of the universal family $\GG$.
\end{remark}

\subsection{Comparison of virtual cycles}

Finally, we can prove the pullback formula \eqref{Eq.P/S.Pullback} in Theorem \ref{Thm.P/S} from the compatibility of the obstruction theories in Proposition \ref{Prop.P/S.compatibility} and the virtual pullback formula in Theorem \ref{Thm.Functoriality}.

\begin{proof}[Proof of Theorem \ref{Thm.P/S}]
By Proposition \ref{Prop.P/S.compatibility}, we can form morphisms of distinguished triangles
\beq\label{PS.11}
\xymatrix{
\bR\hom_{\pi}(\II,\FF)_\#[2] \ar[r]^-{\alpha\dual} \ar[d]^{\beta\dual} & \bR\hom_{\pi}(\II,\II)_0[3] \ar[r]^-{\delta} \ar[d]^{\alpha} & \bR\hom_{\pi}(\FF,\O_{\sX})_\#[4] \ar@{=}[d]\\
(\tau^{[1,3]}\bR\hom_{\pi}(\FF,\FF))[3] \ar[r]^-{\beta} \ar[d]^{\phi_M}& \bR\hom_{\pi}(\FF,\II)_\#[4] \ar[r]^-{\gamma} \ar@{.>}[d]^{\phi_P'} & \bR\hom_{\pi}(\FF,\O_{\sX})_\#[4] \ar@{.>}[d]^{\phi'}\\
\trunc p^* \LL_{\sM} \ar[r] & \trunc \LL_{\sP} \ar[r] & \bL'_{\sP/\sM}
}\eeq
for some maps $\phi_P'$ and $\phi$ such that $\phi_P ' \circ \alpha =\phi_P$, as in the proof of Theorem \ref{Thm.Lefschetz} in \S\ref{ss.ComparisonofVirtualCycles}. By the long exact sequence associated to \eqref{PS.11}, we deduce that
\beq\label{Eq.12}
\phi : \bR\hom_{\pi}(\FF,\O_{\sX})_\#[4] \mapright{\phi'} \bL'_{\sP/\sM} \to \trunc \bL'_{\sP/\sM} \cong \trunc \LL_{\sP/\sM}
\eeq
is a perfect obstruction theory. Since the virtual pullback 
\[p^! : A_* (\sM) \to A_*(\sP)\]
depends only on the virtual cotangent complex $\bR\hom_{\pi}(\FF,\O_{\sX})_\#[4]$, but not on the map $\phi$ by Lemma \ref{Lem.VPindepofPhi} below, the virtual pullback given by the perfect obstruction theory \eqref{Eq.12} is equal to the virtual pullback given by the perfect obstruction theory of the virtual projective bundle $\sP = \PP(\bR\pi^M _* \GG)$ in Definition \ref{Def.VirtualProjectiveBundle}.

Since we have a compatible triple of obstruction theories \eqref{PS.11}, we have a virtual pullback formula
\[p^![\sM]\virt =[\sP]\virt\]
by Theorem \ref{Thm.Functoriality}. It completes the proof.
\end{proof}

We need Lemma \ref{Lem.VPindepofPhi} below to complete the proof of Theorem \ref{Thm.P/S} above.

\begin{lemma}\label{Lem.VPindepofPhi}
Let $\psi_1,\psi_2 : \KK \rightrightarrows \bL_{\cX/\cY}$ be two perfect obstruction theories for a morphism $f:\cX \to \cY$ of quasi-projective schemes. Then the two associated virtual pullbacks coincide.
\end{lemma}

\begin{proof}
Since the virtual pullbacks commute with projective pushforwards, it suffices to show $f^!_{\psi_1}([\cY])=f^!_{\psi_2}([\cY])$ for an integral scheme $\cY$. Choose a global resolution $[K_1\dual \to K_0\dual] \cong \KK$ of vector bundles. Consider the fiber diagrams
\[\xymatrix{
C_1 \ar[r] \ar[d]^{p_1} & K_1 \ar[d]\\
\fC_{\cX/\cY} \ar[r]^-{\fC(\psi_1)} & [K_1/K_0]
} \qquad 
\xymatrix{
C_2 \ar[r] \ar[d]^{p_2} & K_1 \ar[d]\\
\fC_{\cX/\cY} \ar[r]^-{\fC(\psi_2)} & [K_1/K_0]
}\qquad
\xymatrix{
C_3 \ar[r] \ar[d] & C_1 \ar[d]^-{p_1}\\
C_2 \ar[r]^-{p_2} & \fC_{\cX/\cY}
}\]
where the two closed embeddings $\fC_{\cX/\cY} \rightrightarrows [K_1/K_0]$ are given by the two obstruction theories $\psi_1$ and $\psi_2$. Hence we have two short exact sequence
\[\xymatrix{K_0 \ar[r] & C_3 \ar[r] & C_1} \and \xymatrix{K_0\ar[r] & C_3 \ar[r] & C_2}\]
of cones over $\cX$. Therefore, we have
\[f_{\psi_1}^!([\cY]) = 0^!_{K_1}[C_1] = (s(C_1)\cdot c(K_1))_{vd} = (s(C_3)\cdot c(K_0)\cdot c(K_1))_{vd} = f^!_{\psi_2}([\cY])\]
by \cite[Example 4.1.8]{Ful}, where $vd=\dim(\cY) +\rank(\KK)$.
\end{proof}

\appendix

\section{Torus localization without quasi-projectivity}\label{A.quasi-projectivity}

Here we generalize square root virtual pullbacks in \S\ref{S.sqrtVP} and its functorliality in \S\ref{S.Functoriality} to DM morphisms between algebraic stacks. The Kimura sequence in Lemma \ref{Lem.Kimura} is the essential ingredient. As a corollary, we prove the torus localization formula without assuming the quasi-projectivity hypothesis.

\subsection{Kimura sequence}

\begin{lemma}[{\cite{BP,Kimura}}]\label{Lem.Kimura}
Let $p : \cY \to \cX$ be a projective surjective morphism between algebraic stacks with affine stabilizers. Then we have a right exact sequence
\[\xymatrix@C+1.5pc{
A_* (\cY \times_{\cX} \cY) \ar[r]^-{(p_1)_*-(p_2)_*} & A_*(\cY) \ar[r]^-{p_*} & A_*(\cX) \ar[r] & 0,
}\]
where $p_1,p_2:\cY\times_{\cX} \cY \rightrightarrows \cY$ denote the projection maps.
\end{lemma}

For schemes, Lemma \ref{Lem.Kimura} was proved in \cite{Kimura}. The proof of Lemma \ref{Lem.Kimura} for Artin stacks will appear in a forthcoming paper \cite{BP}.

\subsection{Square root virtual pullback and functoriality}

Note that the definitions of symmetric complexes (Definition \ref{Def.SymCplx}), quadratic functions (Definition \ref{Prop.QuadraticFunction}), symmetric obstruction theories (Definition \ref{Def.SOT}), and the isotropic condition (Definition \ref{Def.Isotropic}) can be generalized to algebraic stacks and DM morphisms between algebraic stacks in a straightforward manner.

\begin{definition}
Let $\EE$ be a symmetric complex on an algebraic stack $\cX$. Let $\fQ(\EE)$ be the zero locus of the quadratic function $\fq_{\EE} : \fC(\EE)\to \A^1_{\cX}$ on the associated abelian cone stack $\fC(\EE)$. We define the {\em square root Gysin pullback}
\[\sqrt{0_{\fQ(\EE)}^!} : A_*(\fQ(\EE)) \to A_*(\cX)\]
in the following cases:
\begin{enumerate}
\item Assume that $\cX$ is a separated DM stack. By the Chow lemma, there exists a projective surjective map $p:\tcX \to \cX$ from a quasi-projective scheme $\tcX$. Define the square root Gysin pullback as the unique map that fits into the commutative diagram
\[\xymatrix{
A_*(\fQ(\EE|_{\tcX\times_{\cX} \tcX})) \ar[r] \ar[d]^{\sqrt{0^!_{\fQ(\EE|_{\tcX\times_{\cX} \tcX})}}}  & A_*(\fQ(\EE|_{\tcX})) \ar[r] \ar[d]^{\sqrt{0^!_{\fQ(\EE|_{\tcX})}}} & A_*(\fQ(\EE)) \ar[r] \ar@{.>}[d]^{\sqrt{0_{\fQ(\EE)}^!}} & 0\\
A_*(\tcX\times_{\cX} \tcX) \ar[r] & A_*(\tcX) \ar[r]^{p_*} & A_*(\cX) \ar[r] & 0
}\]
where the other two square root Gysin pullbacks are well-defined as square root virtual pullbacks since $\tcX$ and $\tcX\times_{\cX}\tcX$ are quasi-projective schemes, and the two horizontal sequences are exact by the Kimura sequence (Lemma \ref{Lem.Kimura}). It is easy to show that $\sqrt{0^!_{\fQ(\EE)}}$ is independent of the choice of the projective cover $p:\tcX \to \cX$ using the Kimura sequence.

\item More generally, assume that $\cX=[P/G]$ is the quotient stack of a separated DM stack $P$ by an action of a linear algebraic group $G$. By \cite{Totaro}, there exist $G$-representations $V_i$ and a $G$-invariant open subschemes $U_i\subseteq V_i$ such that $U_i/G$ are quasi-projective schemes and $\mathrm{codim}_{V_i\setminus U_i}V_i \geq i$. We may regard
$r_i:\cX_i:= [(P\times U_i)/G] \to \cX$
as {\em approximations} of $\cX$ by separated DM stacks $\cX_i$. By the homotopy property of Chow groups \cite[Corollary 2.5.7]{Kresch}, we can define the square root Gysin pullback as
\[\xymatrix@C+2pc{
A_d(\fQ(\EE)) \ar@{.>}[r]^{\sqrt{0^!_{\fQ(\EE)}}} \ar[d]_{\cong} & A_d(\cX) \ar[d]^{r_i^*}_{\cong}\\
A_{d+d_i}(\fQ(\EE|_{\cX_i})) \ar[r]^{\sqrt{0^!_{\fQ(\EE|_{\cX_i})}}} & A_{d+d_i}(\cX_i)
}\]
for big enough $i$ for each $d$, where $d_i$ denotes the relative dimension of $r_i$. It is easy to show that $\sqrt{0^!_{\fQ(\EE)}}$ is independent of the choice of the approximation $r_i:\cX_i \to \cX$ using the homotopy property of Chow groups.
\end{enumerate}
\end{definition}

\begin{definition}
Let $f:\cX\to\cY$ be a DM morphism between algebraic stacks equipped with a symmetric obstruction theory $\phi:\EE \to \bL_f$ satisfying the isotropic condition. Then we have a closed embedding $a:\fC_f \hookrightarrow \fQ(\EE)$. Assume that $\cX$ is the quotient stack of a separated DM stack by an action of a linear algebraic group. We define the {\em square root virtual pullback} as the composition
\[\sqrt{f^!} :  A_*(\cY) \mapright{\sp_f} A_*(\fC_f) \mapright{a_*} A_*(\fQ(\EE)) \mapright{\sqrt{0^!_{\fQ(\EE)}}} A_*(\cX)\]
where $\sp_f$ denotes the specialization map.
\end{definition}

It is easy to show that $\sqrt{f^!}$ commutes with projective pushforwards, smooth pullbacks, and Gysin pullbacks for regular immersions.


\begin{theorem}\label{Thm.FunctorialityforStacks}
Consider a commutative diagram \eqref{C.XYZ} of DM morphisms between algebraic stacks equipped a compatible triple $(\phi_f,\phi_g,\phi_{g\circ f})$ of obstruction theories in the sense of Definition \ref{Def.Compatibility}. Assume that $\phi_g$ and $\phi_{g\circ f}$ satisfy the isotropic condition. We further assume the followings:
\begin{enumerate}
\item $\cY$ is the quotient stack of a separated DM stack by an action of a linear algebraic group,
\item $f:\cX\to\cY$ is quasi-projective, and
\item $\cX$ has the resolution property.
\end{enumerate}
Then we have $\sqrt{(g\circ f)^!} = f^! \circ \sqrt{g^!}.$
\end{theorem}

\begin{proof}
Since $\cX$ has the resolution property, Proposition \ref{Prop.DeformationtoNormalCone} also holds in this setting. Indeed, the resolution property for $\cX$ guarantees Lemma \ref{Lem.D2}. The first paragraph of Lemma \ref{Lem.KKP} also holds by \cite{KKP}. Hence we can construct a symmetric obstruction theory $\phi_h$ as in Lemma \ref{Lem.D3}. The isotropic condition follows from Lemma \ref{Lem.Deform.Isotropic} since the isotropic condition can be shown locally.

Now it suffices to show the functoriality for 
\[\cX\to \cY \to \fC_g.\]
Let $\cY=[P/G]$ be the quotient stack of a separated DM stack $P$ by an action of a linear algebraic group $G$. Let $U_i/G \to BG$ be approximations of the classifying stack $BG$ by quasi-projective schemes $U_i/G$ . Let $\cY_i=\cY\times_{BG}(U_i/G)$. By the homotopy property of Chow groups, it suffices to show the functoriality for 
\[\cX\times_{\cY} \cY_i \to \cY_i \to \fC_g\times_{\cY}\cY_i.\]
Since $\cY_i$ is a separated DM stack, we can choose a projective surjective map $\tcY_i \to \cY_i$ from a quasi-projective scheme $\tcY_i$ by the Chow lemma. By the Kimura sequence (Lemma \ref{Lem.Kimura}), it suffices to show the functoriality for 
\[\cX\times_{\cY} \tcY_i \to \tcY_i \to \fC_g\times_{\cY} \tcY_i.\] Since $f:\cX \to \cY$ is quasi-projective, $\cX\times_{\cY} \tcY_i$ is a quasi-projective scheme. Lemma \ref{Lem.conestackcase} completes the proof.
\end{proof}

We expect that Theorem \ref{Thm.FunctorialityforStacks} holds in a much greater generality. However, here we used assumptions that suffices to prove the torus localization formula below for the simplicity of the arguments.

\subsection{Torus localization}

\begin{proposition}
Let $\cX$ be a separated DM stack with a $T=\GG_m$-action. Let $\cX^T$ be the fixed locus \cite{Kresch}. Let $\phi:\EE\to \bL_{\cX}$ be a $T$-equivariant symmetric obstruction theory.
Then $\psi:\EE|_{\cX^T}^{f} \to \bL_{\cX}|_{\cX^T}^{f} \to \bL_{\cX^T}$ is a symmetric obstruction theory for $\cX^T$ by \cite{GrPa}. Assume that $[\cX^T/T]$ has the resolution property. Then we have
\[i_*\left(\dfrac{[\cX^T]\virt}{\sqrt{e}(N\virt)}\right) = [\cX]\virt  \in A^T_*(\cX) \otimes_{\Q[t]} \Q[t^{\pm1}]\]
where $i:\cX^T \hookrightarrow \cX$ denotes the inclusion map, $[B\to E \to B\dual]\cong\EE|_{\cX^T}$ is a $T$-equivariant symmetric resolution, 
$\sqrt{e}(N\virt):=e(B^m)/\sqrt{e}(E^m)$, and $t$ is the first Chern class of the one-dimensional weight one representation of $T$. 
\end{proposition}

\begin{proof}
By \cite[Theorem 5.3.5]{Kresch},
we may write  $[\cX]\virt = i_*(\alpha)$ for some $\alpha \in A^T_*(\cX^T) \otimes_{\Q[t]} \Q[t^{\pm1}]$. Consider a modified symmetric obstruction theory
\[\psi' = (\psi,0) : \EE|_{\cX^T}^f \oplus E^m[1] = [B^f \to E \to (B^f)\dual] \to \bL_{\cX^T}\]
for $\cX^T$. By the Whitney sum formula \cite[Lemma 4.8]{KP}, the homotopy property of Chow groups \cite[Corollary 2.5.7]{Kresch}, and the Kimura sequence (Lemma \ref{Lem.Kimura}), we can easily show that
\[ [\cX^T]\virt_{\psi'} = \sqrt{e}(E^m)[\cX^T]\virt\]
holds. Consider a reduction diagram given by the morphism
\[\xymatrix{
[0 \to 0 \to (B^m)\dual] \ar[r] \ar@{=}[d] & [B^f \to E \to B\dual] \ar[r] \ar[d] & [B^f \to E \to (B^f)\dual] \ar[r] \ar[d] & \\
[0 \to 0 \to (B^m)\dual] \ar[r] & [B \to E \to B\dual] \ar[r] & [B \to E \to (B^f) \dual] \ar[r] &
}\]
of short exact sequences.  Applying the functoriality in Theorem \ref{Thm.FunctorialityforStacks} to $[\cX^T/T] \to [\cX/T] \to BT,$ we deduce
\[ \sqrt{e}(E^m)[\cX^T]\virt= [\cX^T]\virt_{\psi'} = i^![\cX]\virt  = i^! i_* (\alpha )= e(B^m)(\alpha).\]
Therefore, we have $i_*\left(\frac{[\cX^T]\virt}{e(B^m)/\sqrt{e}(E^m)}\right) = [\cX]\virt$.
\end{proof}

\section{Virtual pullback in K-theory}\label{A.K-theory}



\subsection{Twisted virtual pullback}

For any algebraic stack $\cX$, let $K_0(\cX)$ (resp. $K^0(X)$) denote the Grothendieck group of coherent sheaves (resp. vector bundles) with $\Q$-coeffcients. For any line bundle $L$ on a scheme $\cX$, there exists a unique square root $\sqrt{L} \in K^0(\cX)$ such that $(\sqrt{L})^2=L$ and $\sqrt{L}-1$ is nilpotent (cf. \cite[Lemma 5.1]{OT}).

\begin{definition}
Let $f:\cX\to\cY$ be a morphism of quasi-projective schemes equipped with a perfect obstruction theory $\psi: \KK \to \bL_f$. Choose a global resolution $[K_1\dual \to K_0\dual] \cong \KK$ by vector bundles, and let $C= \fC_f\times_{[K_1/K_0]}K_1$. The {\em twisted virtual pullback} is defined as the composition
\[\widehat{f^!} : K_0(\cY) \mapright{\sp_f} K_0(\fC_f) \mapright{p^*} K_0(C) \xrightarrow{\fe(K_1|_C,\tau)} K_0(\cX) \xrightarrow{\sqrt{\det(\KK)}\cdot} K_0(\cX)\]
where $p: C\to \fC_f$ denotes the projection map, $\tau \in \Gamma(C,K_1|_C)$ denotes the tautological section, and $\fe(K_1|_C,\tau)$ denotes the localized K-theoretic Euler class.
\end{definition}

When $\cY=\spec(\C)$ is a point, then $\hO_{\cX}\virt= \widehat{f^!}\left(\O_{\spec(\C)}\right)$ is Nekrasov-Okounkov's twisted virtual structure sheaf \cite{NO}.

\subsection{Square root virtual pullback}
\begin{definition}
Let $f:\cX\to\cY$ be a morphism of quasi-projective schemes equipped with a symetric obstruction theory $\phi:\EE\to\bL_f$ satisfying the isotropic condition. Let $[B\to E\to B\dual] \cong \EE$ be a symmetric resolution (Proposition \ref{Prop.SymRes}) and let $C=\fC_f \times_{[E/B]}E$. We define the {\em twisted square root virtual pullback} as the composition
\[\sqrt{\widehat{f^!}} : K_0(\cY) \mapright{\sp_f} K_0(\fC_f) \mapright{p^*} K_0(C) \xrightarrow{\sqrt{\fe}(E|_C,\tau)} K_0(\cX) \xrightarrow{\sqrt{\det(B)\dual}\cdot} K_0(\cX)\]
where $p:C \to \fC_f$ denotes the projection map, $\tau \in \Gamma(C,E|_C)$ denotes the tautological section, which is isotropic by Lemma \ref{Lem.ComparisonofIsotropic}, and $\sqrt{\fe}(E|_C,\tau)$ denotes the localized square root Euler class \cite{OT}.

In particular, if $\cY=\spec (\C)$ is a point, then the {\em twisted virtual structure sheaf} is defined as
$\hO_{\cX}\virt :=\sqrt{\widehat{f^!}}\left(\O_{\spec(\C)}\right) \in K_0(\cX).$
\end{definition}

The square root virtual pullback $\sqrt{\widehat{f^!}}$ and the twisted virtual structure sheaf $\hO_{\cX}\virt$ are independent of the choice of the symmetric resolution (cf. \cite[Proposition 5.10]{OT}). Also, the square root virtual pullback $\sqrt{\widehat{f^!}}$ commutes with projective pushforwards and lci pullback. Moreover, in the situation of Proposition \ref{Prop.ReductionFormula}, we have a reduction formula
$\sqrt{\widehat{f^!_{\phi}}} = \sqrt{\fe}(G) \cdot \widehat{f^!_{\psi}}.$ 

\begin{theorem}[Functoriality]
Given a commutative diagram \eqref{C.XYZ} of quasi-projective schemes equipped with a compatible triple $(\phi_f,\phi_g,\phi_{g\circ f})$ of obstruction theories, if $\phi_g$ and $\phi_{g\circ f}$ satisfy the isotropic condition, then we have
\[\sqrt{\widehat{(g\circ f)^!}} = \widehat{f^!} \circ \sqrt{\widehat{g^!}}.\]
\end{theorem}

\begin{proof}
The proof is identical to that for Chow theory in \S\ref{S.Functoriality}.
\end{proof}

As corollaries, we have a K-theoretic Lefschetz principle and K-theoretic Pairs/Sheaves correspondence.

\begin{corollary}[Lefschetz principle]\label{Cor.K.L}
In the situation of Theorem \ref{Thm.Lefschetz}, we have
\[\sum_e (-1)^{\sigma(e)}(j_e)_*\left(\hO_{P(D)^e}\virt\right) =  \widehat{{\fe}}(\bR\pi_*(\FF\otimes L))\cdot  \hO_{P(X)}\virt\]
where $\widehat{\fe}(E)=\sqrt{\det(E)}\cdot\fe(E)$.
\end{corollary}

\begin{corollary}[Pairs/Sheaves correspondence]\label{Cor.K.P/S}
In the situation of Theorem \ref{Thm.P/S}, we have
\[\hO_{P_{n,\beta}(X)}\virt = \widehat{p^!}\left( \hO_{M_{n,\beta}(X)}\virt\right).\]
\end{corollary}

All the arguments in \S\ref{S.Lefschetz} and \S\ref{S.Pairs/Sheaves} immediately work for K-theory except Lemma \ref{Lem.VPindepofPhi}. The Lemma \ref{Lem.RR} below shows that Lemma \ref{Lem.VPindepofPhi} also holds in K-theory, which proves Corollary \ref{Cor.K.L} and Corollary \ref{Cor.K.P/S}.

\begin{lemma}[Virtual Grothendieck-Riemann-Roch]\label{Lem.RR}
Let $f:\cX\to\cY$ be a morphism of quasi-projective schemes equipped with a perfect obstruction theory $\psi : \KK \to \bL_f$. Then we have a commutative diagram
\[\xymatrix{
K_0(\cY) \ar[d]^{f^!} \ar[r]^{\tau_\cY}_{\cong} & A_*(\cY) \ar[d]^{\td(\KK\dual)\cdot f^!}\\
K_0(\cX) \ar[r]^{\tau_{\cX}}_{\cong} & A_*(\cX)
}\]
where $\tau_{\cX}$ and $\tau_{\cY}$ denote the Grothendieck-Riemann-Roch maps in \cite{Ful}.
\end{lemma}

\begin{proof}
Choose a global resolution $[K_0 \to K_1]\cong \KK\dual$ and a factorization of $f$ 
\[\xymatrix{
& \tcY \ar[d]^{\overline{f}} \\
\cX \ar@{^{(}->}[ru]^{\tf} \ar[r]^f & \cY
}\]
by a closed embedding $\tf$ and a smooth morphism $\overline{f}$. Form a fiber digram
\[\xymatrix{
C' \ar[r]^{\tq} \ar[d]^{\tp} & C \ar[r]  \ar[d]^p& K_1 \ar[d] \\
C_{\cX/\tcY} \ar[r]^q & \fC_{\cX/\cY} \ar[r] & [K_1/K_0]
}\]
where $p$ is a $K_0$-torsor, and $q$ is a $\TT_{\tcY/\cY}$-torsor. Then we have
\begin{align}
f^! = 0_{[K_1/K_0]}^! \circ \sp_f &=  0^!_{K_1} \circ p^* \circ (q^*)^{-1} \circ \sp_{\cX/\tcY} \circ \overline{f}^* \nonumber\\
&=0^!_{K_1} \circ (\tq^*)^{-1}\circ \tp^* \circ \sp_{\cX/\tcY} \circ \overline{f}^* \label{B1}
\end{align}
both in K-theory and Chow theory.
Since all the operations in \eqref{B1} are for schemes, we have
\begin{align*}
\tau_{\cX} \circ f^! &= \td(-K_1)\cdot \td(\TT_{\tcY/\cY})^{-1}\cdot \td(K_0)\cdot  1 \cdot \td(\TT_{\tcY/\cY}) \cdot f^! \circ \tau_{\cY}\\
&=\td(\KK\dual)\cdot f^! \circ \tau_{\cY}
\end{align*}
by \cite[Theorem 18.2]{Ful}. It completes the proof.
\end{proof}


\section{Reduction of symmetric complexes}\label{A.Reduction}

Here we will slightly generalize the notions of {\em symmetric complexes} and {\em isotropic subcomplexes} in \S\ref{S.sqrtVP}.

\begin{definition}\label{Def.Gen.SymIso}
Let $\cX$ be an algebraic stack. 
\begin{enumerate}
\item A {\em $d$-shifted symmetric complex} is a pair $(\EE,\theta)$ of a perfect complex $\EE$ on $\cX$ and an isomorphism $\theta:\EE\dual \cong\EE[d]$ such that $\theta\dual[d]=\theta$. 
\item An {\em isotropic subcomplex} of $\EE$ is a pair $(\KK,\delta)$ of a perfect complex $\KK$ on $\cX$ and a map $\delta:\EE\to\KK$ such that $\delta_*(\theta) = \delta[d] \circ \theta \circ \delta\dual =0$ and $\Hom(\KK\dual[-d],\KK[-1])=0$.  
\end{enumerate}
\end{definition}

\begin{lemma}\label{Lem.Gen.Reduction}
Let $\EE$ be a $d$-shifted symmetric complex on an algebraic stack $\cX$ and $\KK$ be an isotropic subcomplex. Then there exists a unique $d$-shifted symmetric complex $\GG$ that fits into a morphism of distinguished triangles
\beq\label{C.Eq1}
\xymatrix{
\DD\dual[-d] \ar[r]^-{\alpha\dual} \ar[d]^{\beta\dual} & \EE \ar[r]^{\delta} \ar[d]^{\alpha} & \KK \ar@{=}[d] \ar[r] & \DD\dual[-d+1] \ar[d]\\
\GG \ar[r]^{\beta} & \DD \ar[r]  & \KK \ar[r] & \GG[1]
}\eeq
for some maps $\alpha$, $\beta$ and a perfect complex $\DD$. We call $\GG$ the {\em reduction}.
\end{lemma}

\begin{remark} 
The above definitions are compatible with the definitions in \S\ref{S.sqrtVP} as follows:
\begin{enumerate}
\item The symmetric complexes in Definition \ref{Def.SymCplx} are the (-2)-shifted symmetric complexes (in the sense of Definition \ref{Def.Gen.SymIso}(1)) of tor-amplitude $[-2,0]$ with orientations. 
\item The isotropic subcomplexes in Definition \ref{Def.IsotropicSubcomplex} are the isotropic subcomplexes (in the sense of Definition \ref{Def.Gen.SymIso}(2)) of tor-amplitude $[-1,0]$ whose reductions are of tor-amplitude $[-2,0]$.
\end{enumerate}
\end{remark}

\begin{proof}[Proof of Lemma \ref{Lem.Gen.Reduction}]
Form a distinguished triangle
\[\xymatrix{
\KK\dual[-d] \ar[r]^-{\delta\dual} & \EE \ar[r]^-{\alpha} & \DD \ar[r]^-{\epsilon} & \KK\dual[-d+1]
}\]
for some perfect complex $\DD$ and maps $\alpha$ and $\epsilon$. Since $\delta \circ \delta\dual=0$, there exists a map $\gamma:\DD \to \KK$ that fits into the commutative square
\beq\label{r2}
\xymatrix{\EE \ar[r]^{\delta} \ar[d]^{\alpha} & \KK \ar@{=}[d]\\ \DD \ar@{.>}[r]^{\gamma} & \KK}
\eeq
as the dotted arrow. Moreover, the map $\gamma:\EE\to\KK$ is uniquely determined by the commutative square \eqref{r2} since $\Hom(\KK\dual[-d+1],\KK)=0$. Now form a distinguished triangle
\[\xymatrix{
\GG \ar[r]^{\beta} & \DD \ar[r]^{\gamma} & \KK \ar[r] & \GG[1]
}\]
for some perfect complex $\GG$ and a map $\beta$. Applying the octahedral axiom to the diagram \eqref{r2}, we obtain a commutative diagram
\beq\label{r3}
\xymatrix{
 & \KK\dual[-d] \ar@{=}[r]\ar@{.>}[d]^{\gamma\dual} & \KK\dual[-d] \ar[d]^{\delta\dual} & \\
\KK[-1] \ar[r]^{\epsilon\dual} \ar@{=}[d] & \DD\dual[-d] \ar[r]^{\alpha\dual} \ar@{.>}[d]^{b\dual} & \EE \ar[r]^{\delta} \ar[d]^{\alpha} & \KK \ar@{=}[d]\\
\KK[-1] \ar[r]^{} & \GG \ar[r]^{\beta} \ar@{.>}[d]^{} & \DD \ar[r]^{\gamma} \ar[d]^{\epsilon} & \KK\\
& \KK\dual[-d+1] \ar@{=}[r] & \KK\dual[-d+1]
}\eeq
consists of four distinguished triangles for some map $b\dual : \DD\dual[-d] \to \GG$, where the map $\KK\dual[-d] \to \DD\dual[-d]$ is $\gamma\dual$ since it factors $\delta\dual$. Note that if we dualize the diagram \eqref{r3}, then we obtain the same diagram where $\GG$ is replaced by $\GG\dual[-d]$, and $\beta$ and $b$ are replaced by each other. Form two morphisms of distinguished triangles
\beq\label{r4}
\xymatrix{
\DD\dual[-d] \ar[r]^{\beta\dual} \ar@{=}[d] & \GG\dual[-d] \ar[r]^{\epsilon b} \ar@{.>}[d]^{\eta_1} & \KK\dual[-d+1] \ar@{=}[d]\\
\DD\dual[-d] \ar[r]^{b\dual} & \GG \ar[r]^{\epsilon\beta} & \KK\dual[-d+1]
}\qquad
\xymatrix{
\KK\dual[-1] \ar[r]^{\beta\dual\epsilon\dual} \ar@{=}[d] & \GG\dual[-d] \ar[r]^{b} \ar@{.>}[d]^{\eta_2} & \DD \ar@{=}[d]\\
\KK\dual[-1] \ar[r]^{b\dual\epsilon\dual} & \GG \ar[r]^{\beta} & \DD
}\eeq 
for some maps $\eta_1$ and $\eta_2$. 

We claim that we can choose a map $\eta=\eta_1=\eta_2$ that fits into the two commutative diagrams in \eqref{r4} simultaneously. Indeed, first let $\eta=\eta_1$ be the map that fits into the left commutative diagram in \eqref{r4}. Then we have
\[\eta \circ  \beta\dual = b\dual  \and \epsilon \circ  \beta \circ  \eta = \epsilon \circ  b.\]
Since $\epsilon \circ (\beta \circ \eta -b) =0$, there exists a map $x:\GG\dual[-d] \to \EE$ such that $\beta \circ \eta - b =\alpha \circ x.$ Since $\gamma \circ (\beta\circ \eta - b) = 0$, we have $\gamma\circ \alpha \circ x = \delta \circ x=0$. Hence there exists a map $y: \GG\dual[-d] \to \DD\dual[-d]$ such that $x = \alpha \dual \circ y.$ Also, $(\beta \circ \eta - b) \circ \beta\dual = \beta \circ b\dual - b \circ \beta\dual = \alpha \circ \alpha\dual - \alpha\circ \alpha\dual =0$ implies $\alpha\circ\alpha\dual \circ y\circ  \beta\dual =0$. Hence there exists $z:\DD\dual[-d] \to \KK\dual[-d]$ such that $\alpha\dual \circ y \circ \beta\dual = \delta\dual \circ z.$ Then we have $\alpha\dual \circ (y\circ \beta\dual - \gamma\dual \circ z) = 0$ so that there exists $w : \DD\dual[-d] \to \KK[-1]$ such that $y \circ \beta\dual - \gamma\dual \circ z = \epsilon \dual \circ w.$ Note that $w \circ \gamma\dual = 0 $ since $\Hom(\KK\dual[-d], \KK[-1])=0$. Hence there exists a map $u : \GG\dual[-d] \to \KK[-1]$ such that $w = u \circ \beta\dual.$ Now if we replace $\eta$ as
\[\eta \mapsto \eta - b\dual \circ (y -\epsilon\dual \circ u),\]
then we have
\[\eta \circ \beta \dual = b\dual \and \beta \circ \eta = b.\]
Hence $\eta$ fits into the two commutative diagrams in \eqref{r4} simultaneously.

Note that $\eta\dual$ also fits into the two commutative diagrams in \eqref{r4}. Replace $\eta$ by $\dfrac{\eta+ \eta\dual}{2}$. Then $\eta$ still fits into the two commutative diagrams in \eqref{r4} and $\eta=\eta\dual$. This proves the existence.

Now we will prove the uniqueness. Let $\GG'$ be another $d$-shifted symmetric complex that fits into the diagram \eqref{C.Eq1} with $\beta$ replaced by $\beta'$. Here we can assume that $\DD$, $\gamma$, and $\epsilon$ are fixed. By repeating the arguments in the second paragraph, we can form morphisms
\beq\label{r5}
\xymatrix{
\DD\dual[-d] \ar[r]^-{\beta\dual} \ar@{=}[d] & \GG \ar[r]^-{\epsilon \beta} \ar@{.>}[d]^{f} & \KK\dual[-d+1] \ar@{=}[d]\\
\DD\dual[-d] \ar[r]^-{(\beta')\dual} & \GG' \ar[r]^-{\epsilon\beta'} & \KK\dual[-d+1]
}\qquad
\xymatrix{
\KK\dual[-1] \ar[r]^-{\beta\dual\epsilon\dual} \ar@{=}[d] & \GG \ar[r]^-{\beta} \ar@{.>}[d]^{f} & \DD \ar@{=}[d]\\
\KK\dual[-1] \ar[r]^-{(\beta')\dual\epsilon\dual} & \GG' \ar[r]^-{\beta'} & \DD
}\eeq
of distinguished triangles for some map $f:\GG \to \GG'$ which fits into the both diagram simultaneously. Here we identified $\GG\dual[-d]=\GG$ and $(\GG')\dual[-d] = \GG'$ by their symmetric forms. Then $(f\dual)^{-1} :\GG \to \GG'$ also fits into the both commutative diagrams in \eqref{r5}. Hence after replacing $f$ by $\frac{f+ (f\dual)^{-1}}{2}$, we may assume that $f = (f\dual)^{-1}$. Equivalently, we have a commutative diagram
\[\xymatrix{
\GG \ar[r]^f \ar@{=}[d] & \GG' \ar@{=}[d] \\
\GG\dual[-d] & (\GG')\dual[-d] \ar[l]^{f\dual}
}\]
which means that $\GG$ and $\GG'$ are isomorphic as symmetric complexes.
\end{proof}

\begin{lemma}\label{Lem.R}
Let $\EE$ and $\GG$ be symmetric complexes (in the sense of Definition \ref{Def.SymCplx}) and $\KK$ be an isotropic subcomplex of $\EE$ with respect to $\delta:\EE\to\KK$ (in the sense of Definition \ref{Def.IsotropicSubcomplex}). Assume that we have a (not necessarily commutative) diagram \eqref{C.Eq1} where the two horizontal sequences are distinguished triangles and only the first two squares commute. Then we can form the full reduction diagram \eqref{C.Eq1} with the same maps $\alpha$, $\beta$, $\delta$.
\end{lemma}

\begin{proof}
Form a morphism of distinguished triangles 
\beq\label{R2}
\xymatrix{
\DD\dual[2] \ar[r]^{\alpha\dual} \ar[d]^{\beta\dual} & \EE \ar[r]^{\delta} \ar[d]^{\alpha} & \KK \ar@{.>}[d]^{\phi} \ar[r]^x & \DD\dual[3] \ar[d]^{\beta\dual} \\
\GG \ar[r]^{\beta} & \DD \ar[r]^{\gamma} & \KK \ar[r]^y & \GG[1]
}\eeq
for some map $\phi:\KK \to \KK$. Since $\KK\dual[2]$ is of tor-amplitude $[-2,-1]$, $h^0(\beta\dual)$, $h^0(\alpha)$ are bijective and $h^{-1}(\beta\dual)$, $h^{-1}(\alpha)$ are surjective. From the morphism of the long exact sequences associated to \eqref{R2}, we deduce that $h^0(\phi)$ is bijective and $h^{-1}(\phi)$ is surjective. We now claim that $h^{-1}(\phi)$ is injective. Since $\phi\circ \delta = \gamma \circ \alpha =\delta$, there exists $\psi : \DD\dual[3] \to \KK$ such that $\phi -1 = \psi \circ x$. If $h^{-1}(\phi)(a)=0$, then $h^{-1}(x)(a)=0$ since $h^0(\beta\dual)$ is bijective. Hence 
\[h^{-1}(\phi)(a) = a - h^{-1}(\psi) \circ h^{-1}(x) (a) = a =0.\]
Hence $\phi:\KK \to \KK$ is an isomorphism. Replace $y$ by $y \circ \phi = \beta\dual \circ x$, then we obtain the desired reduction diagram.
\end{proof}

\bibliographystyle{amsplain}

\begin{thebibliography}{99}



\bibitem{BKP} Y. Bae, M. Kool and H. Park. {\em Counting surfaces on Calabi-Yau 4-folds.} In preparation.

\bibitem{BP} Y. Bae and H. Park. {\em A comparison theorem for cycle theories for algebraic stacks.} In preparation.

\bibitem{Behrend} K. Behrend. {\em Donaldson-{T}homas type invariants via microlocal geometry.} Ann. of Math. (2) \textbf{170} (2009), no. 3, 1307--1338.

\bibitem{BeFa} K. Behrend and B. Fantechi. {\em The intrinsic normal cone.} Invent. Math. \textbf{128} (1997), no. 1, 45-88.

\bibitem{BBBJ} O. Ben-Bassat, C. Brav, V. Bussi and D. Joyce. {\em A `{D}arboux theorem' for shifted symplectic structures on derived {A}rtin stacks, with applications.} Geom. Topol. \textbf{19} (2015) no. 3, 1287--1359.

\bibitem{Bojko} A. Bojko. {\em Wall-crossing for zero-dimensional sheaves and Hilbert schemes of points on Calabi-Yau 4-folds.} Preprint, arXiv:2102.01056.

\bibitem{BO} A. Bondal and D. Orlov. {\em Semi-orthogonal decomposition for algebraic varieties.} Preprint.

\bibitem{BJ} D. Borisov and D. Joyce. {\em Virtual fundamental classes for moduli spaces of sheaves on {C}alabi-{Y}au four-folds.} Geom. Topol. \textbf{21} (2017), no. 6, 3231--3311.

\bibitem{BBJ} C. Brav, V. Bussi and D. Joyce. {\em A {D}arboux theorem for derived schemes with shifted symplectic structure.} J. Amer. Math. Soc. \textbf{32} (2019) no. 2, 399--443.

\bibitem{Bridgeland} T. Bridgeland. {\em Hall algebras and curve-counting invariants.} J. Amer. Math. Soc. \textbf{24} (2011), no. 4, 969--998.

\bibitem{CGJ} Y. Cao, J. Gross and D. Joyce. {\em Orientability of moduli spaces of {S}pin(7)-instantons and coherent sheaves on {C}alabi-{Y}au 4-folds.} Adv. Math. \textbf{368} (2020), 107134, 60pp.

\bibitem{CKp} Y. Cao and M. Kool. {\em Zero-dimensional Donaldson-Thomas invariants of Calabi-Yau 4-folds.} Adv. Math. \textbf{338} (2018), 601--648.

\bibitem{CKc} Y.Cao and M. Kool. {\em Curve counting and {DT}/{PT} correspondence for {C}alabi-{Y}au 4-folds.} Adv. Math. \textbf{375} (2020), 107371, 49pp.

\bibitem{CKMk} Y. Cao, M. Kool and S. Monavari. {\em K-theoretic DT/PT correspondence for toric Calabi-Yau 4-folds.} Preprint, arXiv:1906.07856.


\bibitem{CL} Y. Cao and N. C. Leung. {\em Donaldson-Thomas theory for Calabi-Yau 4-folds.} Preprint, arXiv:1407.7659.

\bibitem{CMT} Y. Cao, D. Maulik and Y. Toda. {\em Genus zero Gopakumar-Vafa type invariants for Calabi-Yau 4-folds.} Adv. Math. \textbf{338} (2018), 41-92.

\bibitem{CMT2} Y. Cao, D. Maulik and Y. Toda. {\em Stable pairs and Gopakumar-Vafa type invariants for Calabi-Yau 4-folds.} J. Eur. Math. Soc. (JEMS) 10.4171/JEMS/1110.

\bibitem{CQ} Y. Cao and F. Qu. {\em Tautological Hilbert scheme invariants of Calabi-Yau 4-folds via virtual pullback.} Preprint, arXiv:2012.04415.

\bibitem{CTt} Y. Cao and Y. Toda. {\em Tautological stable pair invariants of Calabi-Yau 4-folds.} Preprint, arXiv:2009.03553.

\bibitem{CTd} Y. Cao and Y. Toda. {\em Curve counting via stable objects in derived categories of Calabi-Yau 4-folds.} Preprint, arXiv:1909.04897.

\bibitem{CKL}  H.-L. Chang, Y.-H. Kiem and J. Li. {\em Torus localization and wall crossing for cosection localized virtual cycles.}  Adv. Math. \textbf{308} (2017), 964-986.


\bibitem{DSY} D.-E. Diaconescu, A. Sheshmani and S.-T. Yau. {\em Atiyah class and sheaf counting on local Calab-Yau fourfolds.} Adv. Math. \textbf{368} (2020), 107132, 54 pp.

\bibitem{EG} D. Edidin and W. Graham. {\em Characteristic classes and quadric bundles.} Duke Math. J. \textbf{78} (1995), no. 2, 277--299. 


\bibitem{Ful} W. Fulton. {\em Intersection theory.} Ergebnisse der Mathematik und ihrer Grenzgebiete. 3. Folge. 2. Springer-Verlag, Berlin, 1998.

\bibitem{GS} A. Gholampour and A. Sheshmani. {\em Donaldson-{T}homas invariants of 2-dimensional sheaves inside threefolds and modular forms.} Adv. Math. \textbf{326} (2018), 79--107.

\bibitem{GrPa} T. Graber and R. Pandharipande. {\em Localization of virtual cycles.}  Invent. Math. \textbf{135} (1999), no. 2, 487--518. 

\bibitem{GJT} J. Gross, D. Joyce and Y. Tanaka. {\em Universal structures in C-linear enumerative invariant theories. I.} Preprint, arXiv:2005.05637.


\bibitem{Huybrechts} D. Huybrechts. {\em Fourier-{M}ukai transforms in algebraic geometry.} Oxford Mathematical Monographs (2006), viii+307 pp.

\bibitem{HL} D. Huybrechts and M. Lehn. {\em The geometry of moduli spaces of sheaves.} Aspects of Mathematics, E31 (1997), xiv+269 pp.

\bibitem{HT} D. Huybrechts and R. P. Thomas. {\em Deformation-obstruction theory for complexes via {A}tiyah and {K}odaira-{S}pencer classes.} Math. Ann. \textbf{346} (2010), no. 3, 545--569.

\bibitem{Illusie} L. Illusie. {\em Complexe cotangent et d\'{e}formations. {I}.} Lecture Notes in Mathematics, Vol. 239 (1971) xv+355 pp.

\bibitem{Inaba} M.-a. Inaba. {\em Toward a definition of moduli of complexes of coherent sheaves on a projective scheme.} J. Math. Kyoto Univ. \textbf{42} (2002), no. 2, 317--329.

\bibitem{JS} D. Joyce and Y. Song. {\em A theory of generalized {D}onaldson-{T}homas invariants.} Mem. Amer. Math. Soc. \textbf{217} (2012), no. 1020, iv+199 pp.


\bibitem{KL} Y.-H. Kiem and J. Li. {\em Localizing virtual cycles by cosections.} J. Amer. Math. Soc. \textbf{26} (2013), no. 4, 1025-1050.



\bibitem{KPVIT} Y.-H. Kiem and H. Park. {\em Virtual intersection theories.} Adv. Math. \textbf{388} (2021), Paper No. 107858, 51 pp.

\bibitem{KP} Y.-H. Kiem and H. Park, {\em Localizing virtual cycles for Donaldson-Thomas invariants of Calabi-Yau 4-folds}, Preprint, arXiv:2012.13167.

\bibitem{KKP} B. Kim, A. Kresch and T. Pantev, {\em Functoriality in intersection theory and a conjecture of Cox, Katz and Lee}, J. Pure Appl. Algebra 179 (2003), 127-136.

\bibitem{Kimura} S.-i. Kimura. {\em Fractional intersection and bivariant theory.} Comm. Algebra \textbf{20} (1992) no. 1, 285--302.

\bibitem{KlemmPandharipande} A. Klemm and R. Pandharipande. {\em Enumerative geometry of {C}alabi-{Y}au 4-folds.} Comm. Math. Phys. \textbf{281} (2008), no. 3, 621--653.

\bibitem{Kresch} A. Kresch. {\em Cycle groups for Artin stacks.} Invent. Math. \textbf{138} (1999), no. 3, 495-536.



\bibitem{LePa} M. Levine and R. Pandharipande. {\em Algebraic cobordism revisited.} Invent. Math. \textbf{176} (2009), no. 1, 63--130.

\bibitem{Li} J. Li. {\em Zero dimensional {D}onaldson-{T}homas invariants of threefolds.} Geom. Topol. \textbf{10} (2006), 2117--2171.




\bibitem{LiTi} J. Li and G. Tian. {\em Virtual moduli cycles and Gromov-Witten invariants of algebraic varieties.} J. Amer. Math. Soc. \textbf{11} (1998), no. 1, 119-174.

\bibitem{Lieblich} M. Lieblich. {\em Moduli of complexes on a proper morphism.} J. Algebraic Geom. \textbf{15} (2006), no. 1, 175--206.


\bibitem{Man} C. Manolache. {\em Virtual pull-backs.} J. Algebraic Geom. \textbf{21} (2012), no. 2, 201-245.

\bibitem{MNOP1}  D. Maulik, N. Nekrasov, A. Okounkov and R. Pandharipande. {\em Gromov-{W}itten theory and {D}onaldson-{T}homas theory. {I}.} Compos. Math. \textbf{142} (2006), no. 5, 1263--1285.

\bibitem{MNOP2} D. Maulik, N. Nekrasov, A. Okounkov and R. Pandharipande. {\em Gromov-{W}itten theory and {D}onaldson-{T}homas theory. {II}.} Compos. Math. \textbf{142} (2006), no. 5, 1286--1304.

\bibitem{NO} N. Nekrasov and A. Okounkov. {\em Membranes and sheaves.} Algebr. Geom. \textbf{3} (2016) no. 3, 320--369.

\bibitem{OT} J. Oh and R. Thomas. {\em Counting sheaves on Calabi-Yau fourfolds, I}. Preprint, arXiv:2009.05542.


\bibitem{Olsson} M. Olsson. {\em Sheaves on {A}rtin stacks.} J. Reine Angew. Math. \textbf{603} (2007) 55--112.

\bibitem{PT} R. Pandharipande and R. P. Thomas. {\em Curve counting via stable pairs in the derived category.} Invent. Math. \textbf{178} (2009), 407--447.

\bibitem{PT3} R. Pandharipande and R. P. Thomas.  {\em Stable pairs and {BPS} invariants.} J. Amer. Math. Soc. \textbf{23} (2010), no. 1, 267--297.

\bibitem{PTVV} T. Pantev, B. To\"{e}n, M. Vaqui\'{e} and G. Vezzosi. {\em Shifted symplectic structures.} Publ. Math. Inst. Hautes \'{E}tudes Sci. \textbf{117} (2013), 271--328.

\bibitem{Pot} J. Le Potier. {\em Systemes coh{\'e}rents et structures de niveau.} Asterisque \textbf{214} (1993).

\bibitem{Qu} F. Qu. {\em Virtual pullbacks in K-theory.} Ann. Inst. Fourier, Grenoble. \textbf{68} (2018), no. 4, 1609-1641.

\bibitem{STV}  T. Sch\"{u}rg, B. To\"{e}n and G. Vezzosi. {\em Derived algebraic geometry, determinants of perfect complexes, and applications to obstruction theories for maps and complexes.} J. Reine Angew. Math. \textbf{702} (2015), 1--40.


\bibitem{TT} Y. Tanaka and R. Thomas. {\em Vafa-{W}itten invariants for projective surfaces {I}: stable case.} J. Algebraic Geom. \textbf{29} (2020), no. 4, 603--668.

\bibitem{Tho} R. Thomas. {\em A holomorphic {C}asson invariant for {C}alabi-{Y}au 3-folds, and bundles on {$K3$} fibrations}, J. Differential Geom. \textbf{54} (2000), no. 2, 367--438.


\bibitem{Toda} Y. Toda. {\em Curve counting theories via stable objects {I}. {DT}/{PT} correspondence.} J. Amer. Math. Soc. \textbf{23} (2010), no. 4, 1119--1157.

\bibitem{ToVa} B. To\"{e}n and M. Vaqui\'{e}. {\em Moduli of objects in dg-categories.} Ann. Sci. \'{E}cole Norm. Sup. (4) \textbf{40} (2007), no. 3, 387--444.


\bibitem{Totaro} B. Totaro. {\em The Chow ring of a classifying space.} Proc. Sympos. Pure Math. 67, Amer. Math. Soc., Providence, RI, 1999.


\end{thebibliography}

\end{document}